\documentclass[10pt]{amsart}
\usepackage{graphics}
\usepackage[english]{babel}
\usepackage{amssymb,amsmath,amsfonts}
\usepackage{datetime}
\usepackage{color}
\usepackage{ulem}
\usepackage{lipsum}
\usepackage{amsfonts}
\usepackage{enumerate}
\usepackage{cite}
\usepackage{tikz}
\usetikzlibrary{matrix}
\usepackage[pdftex]{hyperref}
\RequirePackage{amscd}
\RequirePackage{epic}
\RequirePackage[all]{xy}
\RequirePackage{url}
\RequirePackage[shortlabels]{enumitem}
\usepackage{empheq}
\usepackage{epsfig}
\usepackage{lineno} 
 \usepackage{perpage}
 \MakePerPage{linenumbers}
\usepackage[english]{babel}

\usepackage[utf8]{inputenc}
\usepackage{cite}
\RequirePackage{amscd}
\RequirePackage{epic}
\RequirePackage{eepic}
\RequirePackage[all]{xy}
\RequirePackage{url}
\RequirePackage[shortlabels]{enumitem}
\usepackage{empheq}
\usepackage{nomencl}
\usepackage{epsfig}
\usepackage{graphicx}

\newcommand{\red}[1]{\textcolor{red}{#1}}

\newcommand{\comment}[1]{}
   
    \newcommand{\set}[1]{{\left\{#1\right\}}}
\newcommand{\pa}[1]{{\left(#1\right)}}
\newcommand{\sq}[1]{{\left[#1\right]}}

\newcommand{\abs}[1]{{\left|#1\right|}}
\newcommand{\norm}[1]{{\left |#1\right |}}

\newcommand{\T}{\mathbb{T}}
\newcommand{\Z}{\mathbb{Z}}
\newcommand{\R}{\mathbb{R}}
\newcommand{\C}{\mathbb{C}}
\newcommand{\teta}{\theta}
\newcommand{\dg}{{\mathtt{D}_\g}}
\newcommand{\dgp}{{\mathtt{D}_{\g,\fp}}}
\newcommand{\dgpab}{{\mathtt{D}_{\g,\fp}^{\mu_1,\mu_2}}}

\newcommand{\eps}{\varepsilon}

\renewcommand{\Re}{\operatorname{Re}}
\renewcommand{\Im}{\operatorname{Im}}

\newcommand{\im}{I}
\newcommand{\na}{\widehat{n}}

\newcommand{\co}[1]{\textit{#1}}
\newcommand{\gr}[1]{\textbf{#1}}

\newcommand{\dede}{{\mathtt{d}_\SO}}
\newcommand{\id}{\operatorname{id}}

\newcommand{\nAR}[2]{\abs{#2}_{a_{#1},r_{#1}}^{\dg,\gamma}}
\newcommand{\ad}{\operatorname{ad}}

\usepackage{amsthm}

\RequirePackage{url}
\newtheorem{prop}{Proposition}[section]
    \newtheorem{thm}{Theorem}
    \newtheorem*{thm*}{Theorem}
    \newtheorem*{cor*}{Corollary}
 \newtheorem{hyp}{Assumption}
    \newtheorem{cor}{Corollary}
    \newtheorem{lemma}{Lemma}
    \theoremstyle{remark}
     \newtheorem{ex}{Example}
\newtheorem{rmk}{Remark}[section]
\theoremstyle{definition}
\newtheorem{defn}{Definition}

\numberwithin{equation}{section}
\numberwithin{thm}{section}
\numberwithin{defn}{section}
\numberwithin{prop}{section}
\numberwithin{cor}{section}
\numberwithin{lemma}{section}
\numberwithin{rmk}{section}

\newcommand{\g}{\gamma}

\newcommand{\s}{{\sigma}}

\def\wc{ {}}

\newcommand{\N}{{\mathbb N}}

\newcommand{\cA}{{\mathcal A}}

\newcommand{\cH}{{\mathcal H}}
\newcommand{\cI}{{\mathcal I}}

\newcommand{\cK}{{\mathcal K}}

\newcommand{\cN}{{\mathcal N}}

\newcommand{\cR}{{\mathcal R}}

\newcommand{\fm}{{\mathfrak{m}}}

\newcommand{\fp}{{q}}

\newcommand{\ta}{{\mathtt{a}}}

\newcommand{\td}{{\mathtt{d}}}

\newcommand{\tr}{{\mathtt{r}}}

\newcommand{\tH}{{\mathtt{H}}}

\usepackage{bm}

\newcommand{\al}{{\alpha}}
\newcommand{\bt}{{\beta}}

\renewcommand{\d}{\partial}

\renewcommand{\im}{{\rm i}}
\newcommand{\jap}[1]{\langle #1 \rangle}
\newcommand{\und}[1]{\underline{#1}}

\newcommand{\e}{{\varepsilon}}

\newcommand{\jml}[1]{\lfloor #1 \rfloor}

\newcommand{\tw}{{\mathtt{w}}}
\renewcommand{\th}{{\mathtt{h}}}
\newcommand{\nnorm}[1]{{\left\vert\kern-0.25ex\left\vert\kern-0.25ex\left\vert #1 
    \right\vert\kern-0.25ex\right\vert\kern-0.25ex\right\vert}}

\newcommand{\buu}{{u^\bal}{\bar{u}^\bbt}}

\newcommand{\pon}{{{\Pi^{0,\mathcal{K}}}}} 
\newcommand{\por}{{\Pi^{0,\mathcal{R}}}}
\newcommand{\pd}{{\Pi^{-2}}}
\newcommand{\ono}[1]{#1^{0,\mathcal{N}}}
\newcommand{\oro}[1]{#1^{0,\mathcal{R}}}
\newcommand{\odo}[1]{#1^{-2}}

\newcommand{\es}{e^{\set{S,\cdot}}}
\newcommand{\bal}{{\bm \al}}
\newcommand{\bbt}{{\bm \bt}}

\newcommand{\gnor}[1]{\abs{#1}^{\wc,\gamma}}  

\usepackage{framed,enumitem}

\newcommand{\suca}{\mathtt N}
\newcommand{\ri}{r}
\newcommand{\rs}{{r^*}}
\newcommand{\rf}{{r'}}
\newcommand{\etai}{\eta}
\newcommand{\etaf}{{\eta'}}
\newcommand{\twi}{\tw}
\newcommand{\twf}{{\tw'}}

\newcommand{\GE}{\mathtt G}
\newcommand{\SO}{\mathtt S}
\newcommand{\MS}{\mathtt M}

\newcommand{\Cmon}{{C_{\mathtt{mon}}}}

\newcommand{\tauSO}{{\tau_\SO}}
\newcommand{\tauMS}{{\tau_\MS}}

\newcommand{\deSO}{{\delta}_\SO}

\newcommand{\Calg}{{C_{\mathtt{alg}}(p)}}
\newcommand{\Calgo}{{C_{\mathtt{alg}}(1)}}
\newcommand{\CalgM}{{C_{\mathtt{alg},\mathtt M}(p)}}
\newcommand{\CalgMo}{{C_{\mathtt{alg},\mathtt M}(1)}}

\newcommand{\CNem}{{C_{\mathtt{Nem}}}}

\newcommand{\Ctame}{{C_{\mathtt{tame}}}}

\newcommand{\CM}{{C_0}}
\newcommand{\Cmeas}{{C_{\mathtt{meas}}}}

\newcommand{\Cuno}{{\mathcal C_1}}
\newcommand{\Cdue}{{\mathcal C_2}}

\begin{document}
\author{Luca Biasco}
\address{Università degli Studi Roma Tre}
\email{biasco@mat.uniroma3.it}

\author{Jessica Elisa Massetti}
\address{Scuola Normale Superiore di Pisa}
\email{jessica.massetti@sns.it}

\author{Michela Procesi}
\address{Università degli Studi Roma Tre}
\email{procesi@mat.uniroma3.it}

 \title{Exponential stability estimates for the 1d NLS}

\begin{abstract}
We study stability times for  a family of parameter dependent nonlinear Schr\"odinger equations on the circle, close to the origin.
 Imposing a suitable Diophantine condition (first introduced by Bourgain),
we prove a rather flexible Birkhoff Normal Form theorem, 
which implies, e.g.,  exponential and  sub-exponential time estimates 
in the Sobolev and Gevrey class respectively.  
\end{abstract} 
\maketitle
\setcounter{tocdepth}{2} 
\tableofcontents

\section{Introduction}\label{intro}
We consider families of NLS equations on the circle with external parameters of the form:
\begin{equation}\label{NLSb}
	\im u_t + u_{xx} - V\ast u +  f(x,|u|^2)u=0\,,
\end{equation}
where $\im=\sqrt{-1}$ and $V\ast$ is a Fourier multiplier 
$$
V\ast u = \sum_{j\in\Z} V_j u_j e^{\im j x}\,,\quad \pa{V_j}_{j\in\Z}\in \tw^\infty_\fp\,,
$$
living in the weighted $\ell^\infty$ space
$$
\tw^\infty_\fp:=\{ V = \pa{V_j}_{j\in\Z}\in\ell^\infty\ \ |
\quad |V|_\fp:= \sup_{j\in\Z}|V_j|\jap{j}^\fp <\infty\}\,, \qquad
\fp\geq 0\,, 
$$
where $\jap{j}:=\max\{ |j|,1\},$
while $f(x,y)$ is $2\pi$ periodic and real analytic in $x$ and  is real analytic in $y$ in a neighborhood of $y=0$. We shall assume that $f(x,y)$ has a zero  in $y=0$. By  analyticity, for some $\mathtt a,R>0$ we have
\begin{equation}\label{analitico}
f(x,y)= 
\sum_{d=1}^\infty f^{(d)}(x) y^d\,,\quad
  |f|_{\ta,R}:=\sum_{d=1}^\infty|f^{(d)}|_{\T_{\mathtt a}}R^d <\infty \,,
\end{equation}
where,  given a real analytic function  $g(x)=\displaystyle \sum_{j\in \Z}g_j e^{\im j x},$ we 
set\footnote{Namely $g$ is a holomorphic function on the domain
$\T_a := \set{x\in \C/2\pi\Z\, :\, \abs{\Im x}< a}$
with $L^2$-trace on the boundary.}
$ |g|^2_{\T_{\ta}}:=\displaystyle\sum_{j\in\Z}|g_j|^2e^{2\mathtt a|j|}\,.$
Note that if $f$ is independent of $x$ \eqref{analitico} reduces
to
\begin{equation}\label{analiticobis}
 |f|_{R}:= \sum_{d=1}^\infty|f^{(d)}|R^d <\infty \,.
\end{equation}
Equation \eqref{NLSb} is at least locally well-posed
 (say in a neighborhood of $u=0$ in $H^1$, see e.g. Lemma \ref{cobra}) and has an elliptic fixed point at $u=0$, so that an extremely natural question is to understand {\it stability times} for small initial data. One can informally state the problem as follows: let $E\subset H^1$ be some Banach space and consider \eqref{NLSb} with initial datum $u_0$ such that $|u_0|_E\le \delta\ll1$. By local well posedness, the solution $u(t,x) $ of \eqref{NLSb}
with such initial datum exists and is in $H^1$.
 \begin{defn}
 	We call 
{\it stability time } $T=T(\delta)$ the supremum of the times $t$ such that
{\it  for all $|u_0|_E\le \delta$ one has $u(t,\cdot)\in E$ with $|u(t,\cdot)|_E\le 4\delta$}.
 \end{defn}
Computing the stability time  $T(\delta)$ is  out of reach, so the goal is to give  lower (and possibly  upper) bounds.
\\
 A good comparison is with the case of a finite dimensional Hamiltonian system with a non-degenerate  elliptic fixed point, which in the standard complex symplectic coordinates
$u_j= \frac{1}{\sqrt2}(q_j+ \im p_j)$ 
 is described by the Hamiltonian 
\begin{equation}
\label{finito}
\sum_{j=1}^n \omega_j |u_j|^2 + O(u^3)\,,\quad \mbox{where $\omega_j\in \R$ are the {\it linear frequencies}}.
\end{equation}
Here if the frequencies $\omega$ are sufficiently non degenerate, say Diophantine\footnote{A vector $\omega\in\R^n$ is called Diophantine when it is badly approximated by rationals, i.e. it satisfies, for some $\gamma,\tau>0$, $\abs{k\cdot \omega} \ge \gamma \abs{k}^{-\tau},\quad \forall k\in\Z^n\setminus \set{0}\,$.}, then one can prove exponential lower bounds on $T(\delta)$  and,
if the nonlinearity satisfies some suitable hypothesis ({e.g. convexity or steepness }), even super-exponential ones, see for instance \cite{Giorgilli-Morbidelli:1995}, \cite{Bounemoura-Fayad-Niederman:2015} and reference therein. 
\\
	The strategy for obtaining exponential bounds is made of two main steps. The first one consists in the so-called Birkhoff normal form procedure: after $\suca\geq 1$ steps the 
	Hamiltonian \eqref{finito} is transformed into 
\begin{equation}
\label{finitobirk}
\sum_{j=1}^n \omega_j |u_j|^2 +Z +R\,
\,,
\end{equation}
where $Z$ depends only on the actions $(|u_i|^2)_{i=1}^n$ while $ R= O(|u|^{2\suca +3})$ contains terms of order at least $2\suca + 3$ in $\abs{u}$. \\  It is well known that this procedure generically diverges in $\suca$, so the second step consists in 
finding $\suca=\suca(\delta)$ which minimizes the size of the remainder $R$.
 \\
	The problem of {\it long-time} stability for equations  \eqref{NLSb} has been studied by many authors. In the context of infinite chains with a finite range coupling, we mention \cite{BenFroGior}. Regarding applications to PDEs (and particularly the NLS) the first results were given in \cite{Bourgain:1996} by Bourgain, who proved polynomial bounds for the stability times in the following terms: for any $M$ there exists $s=s(M)$ such that initial data which are $\delta$-small in the $H^{r+s}$ norm stay small  in the $H^{r}$ norm, for times of order $\delta^{-M}$. Afterwards, Bambusi in \cite{Bambu:1999b} proved that superanalytic initial data stay small in analytic norm, for times of order $e^{(\ln(\delta^{-1})^{1+b})}$, where $b>1$.
\\
Bambusi and Grebert in \cite{Bambusi-Grebert:2006} proved 
polynomial bounds for a class of {\it tame-modulus} PDEs, which includes \eqref{NLSb}. More precisely, they proved that for any $\suca\gg 1$ there exists $p(\suca) $ (tending to infinity as $\suca\to \infty$) such that for all $p\ge p(\suca)$ and initial datum in $H^p$ one has 
$T\ge C(\suca,p)\delta^{-\suca}$. For an application to the present model we refer also to \cite{Zh}. \\ Similar results were also proved for the Klein Gordon equation on Zoll manifolds in \cite{BGDS}. Successively Faou and Grebert  in \cite{Faou-Grebert:2013} considered the case of analytic initial data and proved subexponential bounds of the form $T\ge e^{c\ln(\frac1\delta)^{1+\beta}}$ for classes of NLS equations in $\T^d$ (which include \eqref{NLSb} by taking $d=1$). Finally, Feola and Iandoli  in \cite{FI} prove polynomial lower bounds for the stability times of reversible NLS equations with two derivatives in the nonlinearity.

\smallskip

A closely related topic is the study of orbital stability times close to periodic or quasi-periodic solutions of \eqref{NLSb}.
In the case $E=H^1$, Bambusi in \cite{Bambu:1999}  proved a lower bound of the form $T\ge e^{c\delta^{-\beta}}$ for perturbations of the integrable cubic NLS close to a quasi-periodic solution.
Regarding higher Sobolev norms, most results are in the periodic case. See \cite{Faou-etc} (polynomial bounds for Sobolev initial data) and the preprint \cite{Mi-Sun-Wang:2018} 
(subexponential bounds for subanalytic initial data).
\\
A dual point of view is to construct special orbits for which the Sobolev norms grow as fast as possible (thus giving an upper bound on the stability times). As far as we are aware  such  results  are mostly on $\T^2$ and in parameterless cases (for instance \cite{CKSTT}, \cite{GuaKa}, \cite{GHP}) and the time scales involved are much longer than our stability times (see  \cite{Guardia} for the instability of \eqref{NLSb} on $\T^2$ and \cite{Hani} for the istability of the plane wave in $H^s$ with $s<1$).

\subsection{The stability results}

In this paper we  recover and improve the results in \cite{Bambusi-Grebert:2006} 
({\it Sobolev} initial data) and\cite{Faou-Grebert:2013} ({\it analytic and subanalytic} initial data)
under a  different Diophantine non-resonance condition on the linear frequencies,  by application of a  different Birkhoff normal form approach (see the comments 
after Theorem \ref{sob}).
More precisely, following Bourgain \cite{Bourgain:2005}, we set
\begin{equation}
	\Omega_\fp:=\set{\omega=\pa{\omega_j}_{j\in \Z}\in \R^\Z,\quad \sup_j|\omega_j-j^2|\jap{j}^\fp < 1/2 }
\end{equation}
and, for  $\gamma>0$ we define the set of "good frequencies"
as
\begin{equation}\label{diofantinoBISintro}
	\dgp:=\set{\omega\in \Omega_\fp\,:\;	|\omega\cdot \ell|> \gamma \prod_{n\in \Z}\frac{1}{(1+|\ell_n|^2 \jap{n}^{2+\fp})}\,.\quad \forall \ell\in \Z^\Z: |\ell|<\infty}\,,
\end{equation} 
Note that $\dgp$ is large with respect to a natural probability product
measure on $\Omega_\fp$ (see \cite{Bourgain:2005} or  Lemma \ref{misura} in the present paper). 
\begin{rmk}
	From now on we shall fix $\gamma>0$ $q\geq 0$ and assume that 
	$\omega\in\dgp.$
\end{rmk}

\noindent
We note that some non-resonance condition on the frequencies is inevitable if one
wants to prove {\it long-time} stability, indeed if one takes $V=0$ and $f(x,|u|^2)=|u|^4$ then one can exhibit orbits in which the Sobolev norm is unstable in times of order $\delta^{-4}$, see \cite{GT}, \cite{HP}.   

\medskip
\noindent
{\bf Sobolev initial data}.
In the case of Sobolev initial data it is fundamental to have a good control on the dependence of the stabiliy time $T$ on the the regularity $p$.
This means that  results  are very sensitive to which (of various equivalent) Sobolev norms one considers. 
Recalling that the $L^2$-norm is invariant for the equation \ref{NLSb},	
we will consider two cases:
\begin{itemize}[leftmargin=*]
\item In the first case we deal with the usual norm
$\ |u_0|_{L^2}+ |\partial_x^p u_0|_{L^2} $, for $p> 1$. We denote this case as $\SO$  (Sobolev case) and, by fixing $p=p(\delta)$, we prove {\it sub-exponential lower bound} for the stability time $T(\delta)$ .
\\
\item In the second case, denoted by $\MS$ (Modified-Sobolev case), we consider the equivalent norm
$\ 2^p|u_0|_{L^2}+ |\partial_x^p u_0|_{L^2}$. In order to simplify the exposition and obtain better bounds, in this case we consider \eqref{NLSb} with $f$ independent of $x$ (translation invariance). Again, fixing  $p=p(\delta)$, we prove 
{\it exponential lower bound} on the stability time $T(\delta)$.
\end{itemize}
Of course, the norms in $\SO$ and $\MS$ are equivalent with constants depending on $p$. Note that when  $p$ depends on $\delta$ such constants become very important. \\ 
The main qualitative difference between $\SO$ and $\MS$ is that in the latter we are requiring that the Fourier modes $0,1,-1$ of the initial datum have {\it very little energy}. Indeed, passing to the Fourier side $u_0(x)= \sum_{j\in\Z} u_{0,j}e^{\im j x}$, if	 both
$|u_0|_{L^2}+ |\partial_x^p u_0|_{L^2} 
\le 	 \delta/2$ 
and the extra condition
$ |u_{0,0}|^2+ |u_{0,1}|^2+|u_{0,-1}|^2\le \delta^22^{-2p-2}
$
hold,
then one has 	$2^{p}|u_0|_{L^2}\le \delta$. 

Below we formally state our first result, which depends on some constants,  denoted by
$\tau_\SO,\deSO,\mathtt k_\SO, \mathtt{T}_{\SO}, \mathtt K_\SO,  \tau_\MS,\delta_\MS, \mathtt{T}_\MS$
{\it explicitly defined} in Section \ref{ossobuco} of the Appendix. These constants
depend only on $\gamma,\fp,\ta,R,|f|_{\ta,R} $
in the  case $\SO$ and on 
$\gamma, \fp,R,|f|_R$ in the  case $\MS$.

\begin{thm}[Sobolev stability]\label{tarzanello} $\empty$
	Consider equation \eqref{NLSb} with $f$ satisfying \eqref{analitico} for $\ta,R>0$.

	\noindent
$(\SO)$    
For any $p>1$ such that $(p-1)/\tauSO\in\mathbb N$ and 
any  initial datum $u(0)=u_0$ satisfying
\begin{equation}
	\label{moro}
|u_0|_{L^2}+ |\partial_x^p u_0|_{L^2} 
\le
\delta\le   \min\set{ \deSO ({\mathtt k_\SO p})^{ -3 p }\,,\ 
\frac{\sqrt{R}}{20}  },
\end{equation}
the solution $u(t)$ of \eqref{NLSb} with  initial datum $u(0)=u_0$ exists for all times 
\begin{equation}\label{tommaso3}
|t|\le \frac{\mathtt T_\SO}{\delta^2}(\mathtt K_\SO p)^{ -5 p }
\left(\frac{\deSO}{\delta}\right)^{\frac{2(p-1)}{\tauSO}}
\quad {\it and \ satisfies}\quad
|u(t)|_{L^2}+ |\partial_x^p u(t)|_{L^2}
\le
4\delta \,.
\end{equation}

\smallskip
\noindent
$(\MS)$   Assume that $f$ in \eqref{NLSb} is independent of $x$.
For any $p>1$ such that $(p-1)/\tauMS\in\mathbb N$ and 
for any initial datum $u(0)=u_0$ satisfying 
	\begin{equation}\label{eisenach}
2^p|u_0|_{L^2}+ |\partial_x^p u_0|_{L^2} 
\le
\delta\leq 
\min\left\{
\frac{2\sqrt \tauMS \delta_\MS}{\sqrt p}\,,\ 
\frac{\sqrt R}{4\sqrt{10}}
\right\}\,,
\end{equation}
the solution $u(t)$ of \eqref{NLSb}  exists for all times 
\begin{equation}\label{boh3}
|t|\le 
 \frac{\mathtt T_\MS}{\delta^{2}}
  \left( 
  \frac{4\tauMS \delta_\MS^2}{(p-1) \delta^2}
 \right)^{\frac{p-1}{\tauMS}}
\quad {\it and \ satisfies}\quad
2^p|u(t)|_{L^2}+ |\partial_x^p u(t)|_{L^2}
\le
4\delta \,.
\end{equation}

\end{thm}	
 \begin{rmk}\label{rmk:sob0} Some remarks on the optimality of Theorem \ref{tarzanello}  are in order.

	%

	\smallskip
	
1. We stress the fact that estimates \eqref{moro} of case $\SO$ is optimal in some sense. The simplest way of showing this fact is to construct a Hamiltonian which does not preserves momentum and exhibits fast drift. 	
	In fact, if we take $\delta  > (e^{-1}p)^{-p/2}$  then
	orbits performing  ``fast drift'' in a time of order 1 may occur. 
	Indeed consider e.g. , for $ 2 \le j\in \N$  the  family of Hamiltonians:
	$$
	H^{(j)}(u_1,u_j):=|u_1|^2+(j^2+V_j) |u_j|^2+e^{-\ta j}\Re (|u_1|^2u_1\bar u_j)\,.
	$$			
	Passing to action-angle variables $u_i=\sqrt{I_i}e^{\im \vartheta_i}$		
	we get the new Hamiltonian
	$$
	I_1+\omega I_j+e^{-\ta j}I_1^{3/2} \sqrt{I_j}\cos(\vartheta_1-\vartheta_j)
	=J_1+\omega(J_2- J_1)+e^{-\ta j}J_1^{3/2} \sqrt{J_2-J_1}\cos\varphi_1
	$$
	in the new symplectic variables 
	$J_1=I_1,$ $J_2=I_1+I_j,$ $\varphi_1=\vartheta_1-\vartheta_j$,
	$\varphi_2=\vartheta_j.$\\
	Note that this Hamiltonian has $J_2$ as constant of motion while
	$$\dot J_1=e^{-\ta j} J_1^{3/2} \sqrt{J_2-J_1}\sin\varphi_1.$$ 
	In this case the norm in \eqref{checco} reads
	$$
	\sqrt{|u_1|^2+|u_j|^2}+\sqrt{|u_1|^2+j^{2p} |u_j|^2}
	=\sqrt{J_2}+\sqrt{(1-j^{2p})J_1+j^{2p}  J_2}\,.
	$$ 
	Taking the initial datum $u(0)=(u_1(0),u_j(0))$
	with $u_1(0)=\delta/4$, $u_j(0)=j^{-p}\delta/4 ,$
	we have that its norm is smaller than $\delta,$ while
	$ J_1 $ can have a drift of order $\delta^4 j^{-p}e^{-\ta j}$
	in a time $T$ of order 1. 
	This means that the Sobolev norm of $u(T)$ is of order 
	%
	$\delta^3 e^{-\ta j}j^p$  hence greater than $4\delta$ if $\delta^2 e^{-\ta j} j^p$ is large. Maximizing on $j$ we get a constraint of the form
	$\delta^2 e^{-p}(\ta^{-1}p)^p< 1$.	
	\\
	Of course this pathological ''fast diffusion'' phenomenon comes from the non conservation of momentum\footnote{indeed the term $e^{-\ta j}$ is added in order to ensure that monomials with very high momentum give an exponenially small contribution to the Hamiltonian}, and would appear (with similar constants) also in the case $\MS$. 
	\smallskip
	
2. 	It is very important  to stress that in the case $\SO$ restricting to translation invariant Hamiltonians would not result in signficantly weaker constraints on the smallness of $\delta$ w.r.t. $p$.   This can be seen in the following example. 
	Consider the familiy of Hamiltonians (in three degrees of freedom)
	\[
	K^{(j)}:= |u_1|^2 + j^2|u_j|^2 + \Re (\bar u_0^{j-1} u_1^j\bar u_j)
	\]
	with the constants of motion
	\[
	L = |u_0|^2+ |u_1|^2 + |u_j|^2\,,\quad M=  |u_1|^2 + j |u_j|^2\,.
	\]
	Following the same approach as in the previous example one shows that	$|u_j|^2 $ can have a drift of order $j^{-p} \delta^{2j }$
	in a time $T$ of order 1. 	This means that the Sobolev norm of $u(T)$ is of order 
	%
	$\delta^{2j} j^p$. Maximizing on $j$ we get a constraint of the form
	$\delta e^{p^{1^-}}< 1$.
	We point out  that  the Hamiltonian discussed above  is stable in the $\MS$ norm for all times and for $\delta$ small independent of $p$. This is the main reason for restricting in $\MS$ to translation invariant Hamiltonians.
\end{rmk}
From Theorem \ref{tarzanello} it is straightforward to maximize over $p$ and find an optimal regularity.
We stress that in the case $\SO$ our estimate on the stability time is an increasing function of $p$, so the maximum is obtained by just fixing $p$ so that $\delta = (\mathtt C_\SO p)^{-3p}$. On the other hand in the case
$\MS$ there is a proper maximum.
\\
We thus have the following result.
As before our statements  depend on some constants,  denoted by
$\bar{\delta}_\SO,
\bar{\delta}_\MS$ 
  {\it explicitly defined} in Subsection \ref{ossobuco}. These constants
depend only on $\gamma,\fp,\ta,R,|f|_{\ta,R} $
in the  case $\SO$ and on 
$\fp,R,|f|_R$ in the  case $\MS$.
By $[\cdot]$ we denote the integer part.
\begin{thm}[Sobolev stability: optimization]\label{sorbolev!} $\empty$
	
	\noindent
		$(\SO)$  
		For any   $0<\delta\leq\bar{\delta}_\SO$
	and any $u_0$ such that
		\begin{equation}\label{checco}
		|u_0|_{L^2}+ |\partial_x^p u_0|_{L^2} 
		\le 
		\delta\,,
		\qquad\qquad
	 	p=p(\delta):=1+\tauSO \left[\frac{1}{6\tauSO}\frac{\ln (\deSO/\delta)}{\ln \ln (\deSO/\delta)} \right]
	 \,,
	 \end{equation}
	 the solution $u(t)$ of \eqref{NLSb} with  initial datum $u(0)=u_0$ exists for all times 
	 \begin{equation}\label{tommaso}
	 |t|\le
	 \frac{\mathtt T_\SO}{\delta^2} e^{ \ \frac{\ln^2 (\deSO/\delta)}{4\tauSO  \ln \ln (\deSO/\delta)}}
	 \quad {\it and \ satisfies}\quad
	 |u(t)|_{L^2}+ |\partial_x^p u(t)|_{L^2}
	 \le
	 4\delta 
	 \,.
	 \end{equation}
		$(\MS)$  Assume that  $f$ in \eqref{NLSb} is independent of $x$. 
		For any $0<\delta\leq \bar\delta_\MS$ and
	\begin{equation}\label{elisabetta}
	\forall\, p\geq p(\delta):=
		1+\tauMS \left[\frac{\delta_\MS^2}{ \delta^2}\right]\,,
		\quad \forall u_0 \quad \text{s.t.}\quad
		2^p|u_0|_{L^2}+ |\partial_x^p u_0|_{L^2} 
		\le 
		\delta\,,\qquad\qquad
		\,
\end{equation}
 the solution $u(t)$ of \eqref{NLSb} with  initial datum $u(0)=u_0$ exists for all times
		\begin{equation}\label{boh}
		   |t|\le
	\frac{\mathtt T_\MS}{\delta^2} e^{(\delta_\MS/ \delta)^2}
	\quad {\it and \ satisfies}\quad
	2^p|u(t)|_{L^2}+ |\partial_x^p u(t)|_{L^2}
		 \le
		  4\delta 
	\,.
		\end{equation}
\end{thm}
 \begin{rmk}\label{rmk:sob} Some remarks on Theorem \ref{sorbolev!}  are in order.

	Note that \eqref{tommaso} is the stability time computed in \cite{BenFroGior} for short range couplings.
	\smallskip

1. In our study we have only considered {\it Gauge preserving} equations, that is PDEs which preserve the $L^2$ norm. We believe that this is just a technical question and that we could deal with more general cases. Similarly in the case $\MS$ we have assumed that $f$ in \eqref{NLSb}
	is independent of $x$, namely {\it momentum preserving}. Not only this simplifies the proof
	but  as explained after Theorem \ref{tarzanello} allows us much better estimates. Of course we could prove the theorem (with different constants)
	also for $x$-dependent $f$, as in the case $\SO$.\smallskip

	 2. We will prove the case $\MS$ only for
	$p= p(\delta)$, the general case being 
	analogous\footnote{Indeed, thanks to the monotonicity
		property of our norms (see Proposition \ref{crescenza} below) the canonical transformation
		putting the system in Birkhoff Normal Form (see Theorem \ref{sob} below) in the $p$-case is simply
		the restriction to $H^p$ of the one of the $p(\delta)$-case.
	} (with the same constants!)
	also if $p\ge p(\delta)$. 		\smallskip
	
	3. One can easily restate Theorem \ref{sorbolev!} in terms of the Sobolev exponent $p$, 
	instead of $\delta$, since the map $\delta\to p(\delta)$ is injective. 	
\end{rmk}
\smallskip

In this paper we have considered the simplest possible example of dispersive PDE on the circle. One can easily see that the same strategy can be followed word by word in more general cases provided that the non-linearity does not contain derivatives. 
A much more challenging question is to consider NLS models with derivatives in the non-linearity. As we have mentioned a semilinear case was discussed by \cite{Cong}. A very promising approach to Birkhoff normal form for quasilinear PDEs is the one of \cite{Delort-2009}-\cite{Berti-Delort} which was applied to
fully-nonlinear reversible	NLS equations in \cite{FI}. It seems very plausible that one can adapt their methods (based on paralinearizations and paradifferential calculus) to our setting, however it seems that in this case one must give up the Hamiltonian structure. 

\medskip

\noindent
{\bf Analytic and Gevrey initial data} 
\\ In this case our result is similar to \cite{Faou-Grebert:2013} in the sense that we also prove {\it subexponential bounds} on the time. We mention however that in \cite{Faou-Grebert:2013} the control of the Sobloev norm in time is in a lower regularity space w.r.t. the initial datum. 
 Recently we have been made aware of a preprint by Cong, Mi and Wang \cite{Cong} in which the authors give subexponential bounds for subanalytic initial data of a model like \eqref{NLSb}, very similar to ours. A difference is that in their case the   non linearity contains a derivative (see the comments after Theorem \ref{gegge}) but satisfies momentum conservation. \\
Let us fix 
$
0<\theta<1$, 
and define the  function spaces
\begin{equation}
\label{prodi}
\tH_{p,s,a}:= \set{u(x)=\sum_{j\in\Z} u_j e^{\im j x}\in L^2\,:\; |u|_{p,s,a}^2:= \sum_{j\in\Z}\abs{u_j}^2 \jap{j}^{2 p}e^{2 a \abs{j}+ 2s\jap{j}^{\teta}}< \infty}\,.
\end{equation}
with the assumption $a\ge 0, s>0, p>1/2$. We remark that  if $a>0$ this is a space of analytic functions, while if  $a=0$ the functions have Gevrey regularity. Note that for technical reasons connected to the way in which we control the small divisors, we {\it cannot } deal with the purely analytic case  $\theta=1$, see Lemmas \ref{constance general}, \ref{constance 2 gen}. For
this reason we denote this result as $\GE$ (Gevrey case). The main important difference with the cases $\SO,\MS$ is that now the regularity $p,s,a$ is {\it independent} of $\delta$.\\
As before our result, stated below, depends on some  constants
$\bar\delta_\GE,  
\delta_\GE,\mathtt T_\GE$,  {\it explicitely defined} in Subsection \ref{ossobuco}, and 
depending only on $\gamma, \fp,\ta,R,|f|_{\ta,R}, p,s,a,\theta $.

\begin{thm}[Gevrey Stability]\label{gegge}
	Fix any $a\ge 0$, $s>0$ such that $a+s< \mathtt a$ and any $p>1/2$.  
	For  any $0<\delta\leq \bar\delta_\GE$
	and  any $u_0$ such that
	\[
	|u_0|_{p,s,a} \le \delta
	\,,
	\] 
	the solution $u(t)$ of \eqref{NLSb} with  initial datum $u(0)=u_0$ exists for all times
	\begin{equation*}
	|t|\le
	\frac{\mathtt T_\GE}{\delta^2} 
	e^{\pa{\ln\frac{\delta_\GE}{\delta}}^{1+\theta/4}}
	\quad {\it and \ satisfies}\quad
	|u(t)|_{p,s,a} \le 2\delta
	\,.
	\end{equation*}
\end{thm}

\begin{rmk}
	Some comments on Theorem \ref{gegge} are in order.
	
	1. We did not make an effort to maximize the exponent $1+\theta/4$ in the stability time. In fact, by trivially modifying the proof, one could get $1+\theta/(2^+)$. We remark that in  \cite{Cong}, in which $\theta=1/2$,  the exponent is better, i.e. it is $1+1/(2^+)$. 
	
	\smallskip

	2.  As we mentioned before, the main  difference w.r.t. the cases $\SO,\MS$ is that now the regularity $p,s,a$ is independent of $\delta$, with the only requirement that $p>1/2$ and $s>0$. If instead we took $s$ appropriately large with $\delta$ we would get an exponential bound just like in case $\MS$.
	
	\smallskip
	
	3. One could consider initial data with an {\it intermediate} regularity between Sobolev and Gevrey and compute stability times. 
	A good example (suggested to us by Z. Hani) could be the space
	\[
	\tH_c :=\left\{  u= \sum_j u_j e^{\im j x} \in L^2:\quad  \sum_j |u_j|^2 e^{c (\ln(\jml{j})^2)}
	<\infty\right\}
	\]  
	where $c>0$ and $\jml{j}:=
	\max\{|j|,2\}$.  Following the proof of Theorem \ref{gegge} almost verbatim one 
	can get an estimate of the type $T\ge C\delta^{-3+ \ln(\ln(1/\delta))}$. 
\end{rmk}

\subsection{The Birkhoff Normal Form}

Our results are based on a Birkhoff normal form procedure, which we now describe.
Let us pass to the Fourier side via the identification
\begin{equation}\label{silvacane}
u(x)=\sum_{j\in\Z}u_j e^{\im jx}\ \mapsto \
u=(u_j)_{j\in\Z}\,,
\end{equation}
where $u$ belongs to some complete subspace  of $\ell^2$.
More precisely, given a real sequence $\mathtt w=(\mathtt w_j)_{j\in\Z},$
with $\mathtt w_j\geq 1$ we consider the 
Hilbert space\footnote{
	Endowed with the scalar product
	$(u,v)_{\th_\tw}:=\sum_{j\in\Z} \mathtt w_j^2 u_j \bar v_j.$}
\begin{equation}\label{pistacchio}
\th_{\mathtt w}
:=
\set{u:= \pa{u_j}_{j\in\Z}\in\ell^2(\C)\,: \quad \abs{u}_{\tw}^2
	:= 
	\sum_{j\in\Z} \mathtt w_j^2 \abs{u_j}^2 < \infty}\,,
\end{equation}
and fix  the symplectic structure to be 
\begin{equation}\label{nduja}
\im\sum_j d u_j\wedge d \bar u_j\,.
\end{equation}
In this framework the Hamiltonian of \eqref{NLSb} is
\begin{eqnarray}
\label{hamNLS}
&&\quad H_{\rm NLS}(u):=
D_\omega+P\,, \qquad {\rm where}
\\
&&\quad
D_\omega:=\sum_{j\in\Z} \omega_j |u_j|^2\,,
\quad 
P:= \int_\T F(x,|u(x)|^2) dx \,,\quad F(x,y):=\int_0^y f(x,s) ds\,.
\nonumber
\end{eqnarray}
As examples of $\th_{\mathtt w}$ we consider:

$\SO$) (Sobolev case)		
$\tw_j=\jap{j}^{p}$, which is isometrically isomorphic, by Fourier transform, to
$\mathtt H_{p,0,0}$ defined in \eqref{prodi} and is equivalent to $H^p$
equipped with the norm $|\cdot|_{L^2} +|\partial_x^p \cdot|_{L^2}$
with equivalence constants independent of $p$ (see \eqref{vallinsu})

$\MS$) (Modified-Sobolev case)
$\tw_j=	\jml{j}^{p}, $  where $\jml{j}:=
\max\{|j|,2\}$;  this space is equivalent to $H^p$ equipped with the norm $2^p|\cdot|_{L^2} +|\partial_x^p \cdot|_{L^2}$ with equivalence constants independent of $p$ (see \eqref{vallinsuMS})

$\GE$) (Gevrey case)
$\mathtt w_j=\jap{j}^{ p}e^{ a \abs{j}+ s\jap{j}^{\teta}},$
which is isometrically isomorphic, by Fourier transform, to
$\mathtt H_{p,s,a}$ defined in \eqref{prodi}.  
\smallskip

\noindent Here and in the following, given $r>0$, by $B_r(\th_{\mathtt w})$ we mean the closed
ball of radius $r$ centered at the origin of $\th_{\mathtt w}.$
\begin{defn}[majorant analytic Hamiltonians]\label{Hr}
	For $ r>0$,  let
	$\mathcal{A}_r(\th_{\mathtt w})$
	be the space of  
	Hamiltonians 
	$$
	H : B_r(\th_{\mathtt w}) \to \R
	$$ 
	such that there exists a pointwise  absolutely convergent power series expansion\footnote{As usual given a vector $k\in \Z^\Z$, 
		$|k|:=\sum_{j\in\Z}|k_j|$.}
	$$ 
	H(u)  = \sum_{\substack{\bal,\bbt\in\N^\Z\,, \\
			|\bal|+|\bbt|<\infty} }H_{\bal,\bbt}u^\bal \bar u^\bbt\,,
	\qquad
	u^\bal:=\prod_{j\in\Z}u_j^{\bal_j}
	$$ 
	with the following properties: 
	\begin{enumerate}[(i)]
		\item Reality condition:
		\begin{equation}\label{real}
		H_{\bal,\bbt}= \overline{ H}_{\bbt,\bal}\,;
		\end{equation}
		\item Mass conservation:
		\begin{equation}
		H_{\bal,\bbt}= 0 \quad\mbox{if}\;\, |\bal|\neq |\bbt| \,,
		\end{equation}namely the Hamiltonian Poisson commutes with the { \sl mass} $\sum_{j\in \Z}|u_j|^2$;
	\end{enumerate}
	Finally, given $H$ as above, we define its majorant
	$\und H:  B_r(\th_{\mathtt w}) \to \R$  as
	\begin{equation}\label{betta}
	\und H(u)  = \sum_{\substack{\bal,\bbt\in\N^\Z\,, \\
			|\bal|+|\bbt|<\infty} }|H_{\bal,\bbt}|u^\bal \bar u^\bbt\,.
	\end{equation}
\end{defn}
We also define the subspace of normal form Hamiltonians
\begin{equation}\label{nocciolina}
\cK:=\left\{Z\in \mathcal{A}_r(\th_{\mathtt w})\, : \, Z(u) = \sum_{\bal\in \N^\Z} Z_{\bal,\bal}|u|^{2\bal}\right\}\,.
\end{equation}
Note that $Z_{\bal,\bal}\in\R$ for every $\bal\in\N^\Z$
by condition \eqref{real}.\smallskip\\
In the following we will also deal with a smaller class of Hamiltonians, namely the ones which have the {\sl{momentum}} $\sum_{j\in} j\abs{u_j}^2$ as additional first integral. 
\begin{defn}
	We say that a Hamiltonian $H\in \mathcal{A}_r(\th_{\mathtt w})$ preserves momentum when
	\[
	H_{\bal,\bbt}=0\quad \mbox{if}\quad \sum_{j\in\Z}j\pa{\al_j - \bt_j}\neq 0\,,
	\]
	namely the Hamiltonian $H$ Poisson commutes with $\sum_{j\in} j\abs{u_j}^2$.
\end{defn}
Note that if the nonlinearity $f$ in equation \eqref{NLSb} does not depend on the variable $x$, then the Hamiltonian $P$ in \eqref{hamNLS} preserves momentum.

\smallskip

We now state a Birkhoff Normal Form Theorem
for the Hamiltonian in  \eqref{hamNLS}. Fix any $\suca\ge 1$
and consider 
the space $\th_{\mathtt w}$  where $\tw$ is one of the following three cases, where $\tau, \tau_1$ are fixed positive constants defined in \eqref{ossobuco}:

$\SO$) (Sobolev case)		
$\tw_j=\jap{j}^{1+ \tauSO \suca}$; 

$\MS$) (Modified-Sobolev case)
$\tw_j=	\jml{j}^{1+ \tauMS \suca}, $  where $\jml{j}:=
\max\{|j|,2\}$;

$\GE$) (Gevrey case)
$\tw_j=e^{a|j|+ s \jap{j}^\theta}
\jap{j}^p$ with $p>1/2, s>0$,  $0\le a<\ta$.

\noindent		
As before
we define in Subsection \ref{ossobuco} below the constants
$\mathtt r, \mathtt C_1, \mathtt C_2, \mathtt C_3,$ corresponding to the cases $\SO,\MS,\GE$
respectively. We remark that these constants depend 
on $\suca\geq 1.$

\begin{thm}[Birkhoff Normal Form]\label{sob}
	Fix  any $\suca \ge 1$ and consider the space $\th_{\mathtt w}$  where $\tw$ is one of the three above cases:
	$\SO,\MS,\GE.$
	Consider the Hamiltonian \eqref{hamNLS}, assuming,
	only in the case $\MS,$ that $f$ does not depend on $x$
	(momentum conservation). 
	Then for any $0<r\leq \mathtt r$ 
	there exists two close to identity invertible symplectic change of variables 
	\begin{eqnarray}
	\nonumber
	&\Psi,\Psi^{-1}:\quad B_{r}(\th_\tw)\mapsto \th_\tw \,,\quad 
	\sup_{|u|_{\tw}\leq r}|\Psi^{\pm 1}(u)-u|_{\tw} \le \mathtt C_1 r^3 
	\leq \frac18 r
	\,,\\
	&\Psi\circ\Psi^{-1}u= \Psi^{-1}\circ\Psi u= u \,,\quad \forall u\in B_{\frac78 r}(\th_\tw)
	\label{stracchinobis}
	\end{eqnarray}
	such that in the new coordinates
	\[
	H\circ \Psi= D_\omega + Z+ R\,,
	\]
	for suitable majorant analytic Hamiltonians
	$Z,R \in \cA_r(\th_{\mathtt w}),$ 
	$Z\in \cK,$  
	satisfying the estimate
	\begin{equation}\label{duspaghibis}
	\sup_{|u|_{\tw}\leq  r }|X_{\und Z}|_{\tw} \le \mathtt C_2 r^{3}\,,
	\quad  
	\sup_{|u|_{\tw}\leq  r }|X_{\und R}|_{\tw} \le \mathtt C_3 r^{2\suca +3}\,,
	\end{equation}
	$X_{\und Z}$ (resp. $X_{\und R}$), being the hamiltonian vector 
	field generated by the the majorant of $Z$ (resp. $R$).
	Moreover, in the  case $\MS$, $R$ preserves momentum.
\end{thm}
The proof of our  Birkhoff normal form result is based on a  procedure which, while following the line of previous works such as \cite{Bambusi-Grebert:2006} and \cite{Faou-Grebert:2013}, it takes a slightly different point of view. Broadly speaking the core is the following: as already noticed in \cite{Faou-Grebert:2013} small divisor estimates and hence stability are  simpler to prove for traslation invariant PDEs (i.e. Hamiltonian systems which preserve the momentum). Considering this fact we introduce an appropriate norm, which weights non-momentum preserving monomial exponentially. This norm is rather cumbersome and depends on many parameters (see comments in the next page) but we show that it is very well suited for performing Birkhoff normal form steps for dispersive PDEs on the circle. This rather simple idea, allows us a very good control of the small divisors by generalizing the estimates by Bourgain in \cite{Bourgain:2005}. As a byproduct our normal forms are {\sl simpler}, in the sense that they are functions only of the linear actions, and it is relatively easy to compute all the constants.
\\
Above we stated Theorem \ref{sob} only in the  cases 
$\SO,\MS,\GE,$ but our method is quite  versatile
 we thus formulate a Birkhoff Norma Form
result in the general contest of weighted Hilbert\footnote{The Banach
	case could be treated as well.} spaces,
 see Theorem \ref{weisserose}. 
Once we have the Birkhoff theorem, Theorem \ref{tarzanello} follows directly. 

\subsection{About the norms.}

After a brief presentation of the symplectic structure relevant for NLS equations we start, in Section \ref{FS}, by defining the subspace of $\mathcal{A}_{r}(\th_{\mathtt w})$ of {\sl $\eta$-majorant regular Hamiltonians} denoted by $\cH_{r,\eta}(\th_\tw)$ and defined by the condition that an appropriate majorant norm  of the associated Hamiltonian vector field is bounded (see Definition \ref{nonna}).  The parameters $r>0,\eta\geq 0$ have the following role:
 $r$ controls  the radius of analiticity of the Hamiltonian vector field as a function from $\th_{\tw}$ to itself,
 while
 $\eta$ ensures that the terms (monomials) which do not preserve momentum are exponentially small.
\\
	We note that  for Hamiltonians which preserve momentum the dependence on the parameter $\eta$ is trivial and can be omitted. Indeed in the latter case the norm coincides with the usual majorant norm, see for instance \cite{BBiP2}.
	Another interesting point is that on the space of  polynomials  our norm controls the majorant-tame norm defined in \cite{Bambusi-Grebert:2006} (see Proposition \ref{maspero}). Although this fact is not needed in our proof we find it an interesting remark (it was pointed out to us by A. Maspero), since most proofs of Sobolev stability strongly rely on majorant-tame properties of the Hamiltonians.
	
	Our norm is well suited for measuring Hamiltonian vector fields, indeed in Subsection \ref{pesce} we show that it is closed with respect to Poission brackets (as a scale in $r$). Moreover a Hamiltonian with small norm defines  a well posed symplectic change of variables on a ball of $\th_{\tw}$.
	Furthermore if $\th_{\mathtt w}$ is closed by convolution then the nonlinearity $P$ of the NLS Hamiltonian \eqref{hamNLS} is in $\cH_{r,\eta}(\th_\tw)$ for appropriate $r,\eta$. This is discussed in Subsection \ref{nemo} in the cases $\SO,\MS,\GE$.
	
Anyway we think that the main point is that our norm	has  explicit (and for us quite surprising) {\sl monotonicity} properties. In Section \ref{monotono} we first give an abstract result, which ensures that $ \cH_{r,\eta}(\th_\tw) \subseteq \cH_{r',\eta'}(\th_{\tw'})$ under appropriate relations among $r',r, \eta', \eta, \tw, \tw'$, while in Proposition \ref{crescenza} we specify to the three cases $\SO,\MS,\GE$.

Finally our norm is well suited for the control of the solution of the homological equation  $\{D_\omega,S\}=R$. In Section \ref{omologo}	
we first give an abstract result, which ensures that   if $ R\in  \cH_{r,\eta}(\th_\tw) $ then $S\in  \cH_{r',\eta'}(\th_{\tw'})$ (for an appropriate choice of $r',\eta',\tw'$) and satisfies a quantitative bound. Then, in Proposition \ref{Lieder} we specify to our three cases $\SO,\MS,\GE$.	

\medskip
At this point we have all the ingredients needed to perform the steps of Birkhoff normal form; this is done  in Section \ref{birk} in an abstract setting. Finally we specify to our three cases $\SO,\MS,\GE$ and prove Theorem \ref{sob} in Section \ref{provasob} and Theorems \ref{sorbolev!} and \ref{gegge} in Section \ref{fine2}. 
\\
The appendices are devoted to giving full details of the most technical proofs.

\subsection*{Acknowledgements} The three authors have been supported by the ERC grant HamPDEs under FP7 n. 306414 and the PRIN Variational Methods in Analysis, Geometry and Physics .   The authors would also like to thank D. Bambusi, M. Berti, B. Grebert, Z. Hani and A. Maspero for helpful suggestions and fruitful discussions.

\section{Functional setting and symplectic structure}\label{FS}

\subsection{Spaces of Hamiltonians}
As explained in the Introduction our wheighted spaces $\th_\tw$ are contained in $\ell^2(\C)$, so we endow them with the  standard symplectic structure  coming from the Hermitian product on $\ell^2(\C)$. 
\\
We identify $\ell^2(\C)$ with $\ell^2(\R)\times \ell^2(\R)$ through $u_j= \pa{x_j+ i y_j}/\sqrt{2}$ and induce on $\ell^2(\C)$ the structure of a real symplectic Hilbert space\footnote{We recall that given a  complex Hilbert space $H$ with a Hermitian product $(\cdot,\cdot)$, its realification is a real symplectic Hilbert space with scalar product  and symplectic form given by
\[
\langle u,v\rangle = 2{\rm Re}(u,v)\,,\quad  \omega(u,v)= 2{\rm Im}(u,v)\,.
\] } by setting, for any $(u^{(1)}, u^{(2)}) \in \ell^2(\C)\times \ell^2(\C)$,
\[
\langle u^{(1)},u^{(2)}\rangle = \sum_j \pa{x_j^{(1)}x_j^{(2)}+ y_j^{(1)}y_j^{(2)}} \,,\quad \omega(u^{(1)},u^{(2)})= \sum_j \pa{y_j^{(1)}x_j^{(2)}- x_j^{(1)}y_j^{(2)}},
\] 
which are the standard scalar product and symplectic form $\Omega= \sum_j dy_j\wedge d x_j$. \\
For convenience and to keep track of the complex structure, one often writes the vector fields and the differential forms in complex notation, that is
\[
\Omega = \im \sum_j d u_j\wedge d \bar u_j \,,\quad X_H^{(j)}  = \im \frac{\partial}{\partial \bar u_j} H\,
\]
where the one form and vector field are defined through the identification between $\C$ and $\R^2$, given by
\begin{align*}
d u_j = \frac{1}{\sqrt 2}\pa{d x_j+ \im d y_j}\,,\quad d \bar u_j = \frac{1}{\sqrt 2}\pa{d x_j- \im d y_j}\,,\\  \frac{\partial}{\partial  u_j} =  \frac{1}{\sqrt 2}\pa{\frac{\partial}{\partial x_j} - \im \frac{\partial}{\partial y_j}}\,,\quad  \frac{\partial}{\partial  \bar u_j} =  \frac{1}{\sqrt 2}\pa{\frac{\partial}{\partial x_j} + \im \frac{\partial}{\partial y_j}}.
\end{align*}

\begin{defn}[$\eta$-majorant analytic Hamiltonians]\label{Hreta}
For $\eta\ge 0, r>0$ let 
$\mathcal{A}_{r,\eta}(\th_{\mathtt w})$
$\subseteq \mathcal{A}_{r}(\th_{\mathtt w})$ be the subspace of  
majorant analytic Hamiltonians (recall Definition \ref{Hr})
such that
	 \begin{equation}\label{etamag}
	\und { H}_\eta (u)= \sum_{\bal,\bbt\in\N^\Z} \abs{{H}_{\bal,\bbt}}e^{\eta|\pi(\bal-\bbt)|}\buu
	\end{equation} 
	is 	point-wise  absolutely convergent on 
	$B_r(\th_{\mathtt w})$ and
\begin{equation}\label{momento}
\pi(\bal - \bbt) := \sum_{j\in\Z}j\pa{\bal_j - \bbt_j}.
\end{equation}
\end{defn}

\noindent
Note that $\pi(\bal-\bbt)$ is the eigenvalue of the adjoint action of the momentum Hamilonian $ \sum_{j\in \Z}j|u_j|^2$ on the monomial $\buu$. The exponential weight $e^{\eta|\pi(\bal-\bbt)|}$ is added in order to ensure that the monomials which do not preserve momentum must have an exponentially small coefficient.

{\sl
The Hamiltonian functions being defined modulo a constant term, we shall assume without loss of generality that $H(0)=0$. 
}

\smallskip
We will say that a Hamiltonian $H(u)\in \mathcal{A}_{r,\eta}(\th_{\mathtt w})$ is 
$\eta$-\co{regular} if 
 $X_{\und { H}_\eta} : B_r(\th_{\mathtt w})
 \to \th_{\mathtt w}$ 
 and is uniformly bounded,
 where ${X}_{{\underline H}_\eta}$ is the vector field associated to the $\eta$-majorant Hamiltonian in \eqref{etamag}.
  More precisely we 
 give the following
 \begin{defn}[$\eta$-regular Hamiltonians]\label{nonna}
For $\eta\ge 0, r>0$ let 
$\cH_{r,\eta}(\th_{\mathtt w})$ be
the subspace of 
$\cA_{r,\eta}(\th_{\mathtt w})$ of those Hamiltonians $H$ such that
$$
|H|_{\cH_{r,\eta}(\th_{\mathtt w})}
=
|H|_{r,\eta,\tw}
:=
r^{-1} \pa{\sup_{\norm{u}_{\th_{\mathtt w}}\leq r} 
\norm{{X}_{{\underline H}_\eta}}_{\th_{\mathtt w}} } < \infty\,.
$$
\end{defn}

We shall show that this guarantees that   the Hamiltonian flow of $H$ exists at least locally and generates a symplectic transformation on $\th_\tw$, i.e. $\th_\tw$ is an invariant subspace for the dynamics. 

\begin{rmk}
We note that if $H$ preserves momentum, then $\abs{H}_{r,\eta,\tw}= \abs{H}_{r,0,\tw}$ does not depend on $\eta$ and coincides with the
majorant norm of a  regular Hamiltonian as defined in \cite[Definition 2.6]{BBiP1}, when $\cI=\emptyset$. 
	Actually this is nothing but the restriction to Hamiltonian vector fields of the {\sl $\eta$-momentum} majorant norm defined in \cite[Definition 2.3]{BBiP2} when $\cI=\emptyset$.
\end{rmk}

\begin{rmk}\label{cacioepepe}
By mass conservation and since $H(0)=0,$ it is straightforward to prove that the norm 
$|\cdot|_{r,\eta,\tw}$ is increasing in the radius parameter $r$ (see also Proposition\ref{cacioricotta}). 
\end{rmk}

Note that if  $|H|_{\cH_{r,\eta}(\th_{\mathtt w})}<\infty$ then $H$ admits an analytic extension $\widehat H,$ that is  \[ (u_+, u_{-})\in B_r(\ell^2(\C))\times B_r(\ell^2(\C)) \to \widehat H(u_+,u_-)\, :\quad \quad H(u)=\widehat H (u,\bar u),\]   whose Taylor series expansion is 
	\[
	 \widehat H(u_+,u_-)  = \sum^\ast_{\bal,\bbt\in\N^\Z} H_{\bal,\bbt}u_+^\bal u_-^\bbt\,.
	\]
	where we denote by $\sum^\ast$ the sum restricted to those $\bal,\bbt: |\bal|=|\bbt|<\infty$.\\
One can see that
\[
\frac{\partial}{\partial \bar u_j} H(u) = \frac{\partial \widehat H(u_+,u_-)}{\partial u_{-,j}} \Big\vert_{u_+=\bar u_-=u}\,.
\]

Let us now define two fundamental subspaces of $\cH^\wc_{r,\eta}(\th_\tw)$.
\begin{defn}[Range and Kernel] Let $\cR$ (for Range) and $\cK$ for (Kernel) the following subspaces of $\cH^\wc_{r,\eta}(\th_\tw)$
\begin{align}
	\label{RgKer}
	\cR=
\cR^\wc_{r,\eta}(\th_\tw) & 
:= \{ H\in \cH^\wc_{r,\eta}(\th_\tw)\ \ \vert\quad H = \sum_{\bal\neq \bbt} H_{\bal,\bbt}\buu \}	\\
\cK=
\cK^\wc_{r,\eta}(\th_\tw) & 
:= \{ H\in \cH^\wc_{r,\eta}(\th_\tw)\ \ 
\vert\quad H = \sum_{\bal\in \N^\Z} H_{\bal,\bal}|u|^{2\bal} \}
\end{align}
We thus can write $\cH^\wc_{r,\eta}(\th_\tw)=
	 \cR^\wc_{r,\eta}(\th_\tw)\oplus 
	 \cK^\wc_{r,\eta}(\th_\tw)$.
	\end{defn}
Note that $ \cR$ and $ \cK$
 define continuous projection operators with
\begin{equation}\label{fame}
|\Pi_{\cK}H|^\wc_{r,\eta,\tw},
 |\Pi_{\cR}H|^\wc_{r,\eta,\tw} \le |H|^\wc_{r,\eta,\tw}
\end{equation}

\begin{ex}[Notation for the Gevrey case]
In the case 
	${\th_{\mathtt w}}=\th_{p,s,a}$
	we denote the space of $\eta$-regular Hamiltonians
	$\cH_{r,\eta}(\th_{p,s,a})$
	by $\cH_{r,p,s,a,\eta}$
	with norm
	\begin{equation}\label{norma1}
	\abs{H}_{r,p,s,a,\eta} 
	:=
	r^{-1} \pa{\sup_{\norm{u}_{p,s,a}\leq r} 
		\norm{{X}_{{\underline H}_\eta}}_{p,s,a} },
	\end{equation}
	namely
	$$
	|\cdot|_{r,p,s,a,\eta}
	=|\cdot|_{\cH_{r,\eta}(\th_{p,s,a})}\,.
	$$
\end{ex}

\begin{defn}[Modified Sobolev space] \label{farfa}
	Fix $\tw=\jml{j}^p$  where 
	$$
	\jml{j}:=
	\max\{|j|,2\}
	$$ 
	 and consider $\th^p:= \th_{\tw}$ endowed with the norm
	\begin{equation}\label{mompeo}
	\|u\|_p^2:=\sum_{j\in\mathbb Z} \jml{j}^{2p} |u_j|^2 \,.
	\end{equation}
	This norm is equivalent to the norm
	$$
	|u|_p^2:=\sum_{j\in\mathbb Z} \jap{j}^{2p} |u_j|^2\,,
	$$
	since
	$$
	|\cdot|_p\leq \|\cdot\|_p\leq 2^p |\cdot|_p\,.
	$$
		\end{defn}
		
\begin{rmk}
	Note that, identifying 
	by the Fourier transform  \eqref{silvacane}
	 the function $u(x)$ 
	with the sequence of its Fourier coefficients $u$,
	 we have 
	$$
	|u(x)|_{p,0,0}=|u|_p\,.
	$$
	Moreover
	\begin{equation}\label{vallinsu}
|u|_p\leq |u(x)|_{L^2}+|\partial_x^p u(x)|_{L^2}\leq 2 |u|_p
\end{equation}
and
\begin{equation}\label{vallinsuMS}
\|u\|_p\leq 2^p |u(x)|_{L^2}+|\partial_x^p u(x)|_{L^2}\leq 2\|u\|_p
\end{equation}
	Note that here we write $u(x)$ and $u$
	to distinguish the function $u(x)$ from the 
	sequence $u$ of its Fourier coefficients;
	however in the rest of the paper we simply write $u$
	to denote the function too.
\end{rmk}	
	
We now introduce the subspace of $\cH_{r,0}\pa{\th^p}$ of those Hamiltonians
	preserving momentum.\footnote{Note that on the 
		preserving momentum subspace 
		$\cH_{r,\eta}(\th_{\mathtt w})$
		coincides with
		$\cH_{r,0}(\th_{\mathtt w})$
		for every $\eta.$
	}

\begin{defn}[momentum preserving regular Hamiltonians]
	Given $r>0,p\geq 0$ let $\cH^{r,p}$ be
	the space of point-wise absolutely convergent Hamiltonians 
	on $\|u\|_p\le r$ which preserves momentum and such that  
	\begin{equation}\label{normabarrette}
	\|H\|_{r,p} := r^{-1} \pa{\sup_{\|u\|_{p}\le r} \|{X}_{{\underline H}}\|_{p} }< \infty\,,
	\end{equation}
	namely
	$$
	\|\cdot\|_{r,p}=|\cdot|_{\cH_{r,0}(\th^p)}\,.
	$$
\end{defn}	

\subsection{Poisson structure and hamiltonian flows}\label{pesce}

The scale $\{\cH_{r,\eta}(\th_{\mathtt w})\}_{r>0} $ is a  Banach-Poisson algebra in the following sense
\begin{prop}\label{fan}
	For $0 <\rho\leq r$ and $\eta>0$ we have
	\begin{equation}\label{commXHK}
	|\{F,G\}|_{r,\eta,\tw}
	\le 
	4\pa{1+\frac{r}{\rho}}
	|F|_{r+\rho,\eta,\tw}
	|G|_{r+\rho,\eta,\tw}\,.
	\end{equation}
\end{prop}
\begin{proof}
	It is essentially contained in 
	\cite{BBiP1}.
	See in particular Lemma 2.16 of \cite{BBiP1} with $n=0$ (no action variables here) and no $s$ and $s'$ (no actions variable here).
	Note that the constant in Lemma 2.16
	is 8, instead of 4 in the present paper,
	because of the presence there of action variables
	which scale different from the cartesian ones
	(namely $(2r)^2$ instead of $2r$).
	Recall also the required properties of the space $E$
	(named $\th_{\mathtt w}$ in the present paper)
	mentioned after Definition  2.5.
\end{proof}

The following Lemma is a simple corollary
\begin{lemma}[Hamiltonian flow]\label{ham flow}
	Let $0<\rho< r $,  and $S\in\cH_{r+\rho,\eta}(\th_{\mathtt w})$ with 
	\begin{equation}\label{stima generatrice}
	\abs{S}_{r+\rho,\eta,\tw} \leq\delta:= \frac{\rho}{8 e\pa{r+\rho}}. 
	\end{equation} 
	Then the time $1$-Hamiltonian flow 
	$\Phi^1_S: B_r(\th_{\mathtt w})\to
	B_{r + \rho}(\th_{\mathtt w})$  is well defined, analytic, symplectic with
	\begin{equation}
	\label{pollon}
	\sup_{u\in  B_r(\th_{\mathtt w})} 	\norm{\Phi^1_S(u)-u}_{\th_{\mathtt w}}
	\le
	(r+\rho)  \abs{S}_{r+\rho,\eta,\tw}
	\leq
	\frac{\rho}{8 e}.
	\end{equation}
	For any $H\in \cH_{r+\rho,\eta}(\th_{\mathtt w})$
	we have that
	$H\circ\Phi^1_S= e^{\set{S,\cdot}} H\in\cH_{r,\eta}(\th_{\mathtt w})$ and
	\begin{align}
	\label{tizio}
	\abs{\es H}_{r,\eta,\tw} & \le 2 \abs{H}_{r+\rho,\eta,\tw}\,,
	\\
	\label{caio}
	\abs{\pa{\es - \id}H}_{r,\eta,\tw}
	&\le  \delta^{-1}
	\abs{S}_{r+\rho,\eta,\tw}
	\abs{H}_{r+\rho,\eta,\tw}\,,
	\\
	\label{sempronio}
	\abs{\pa{\es - \id - \set{S,\cdot}}H}_{r,\eta,\tw} &\le 
	\frac12 \delta^{-2}
	\abs{S}_{r+\rho,\eta,\tw}^2
	\abs{H}_{r+\rho,\eta,\tw}	\end{align}
	More generally for any $h\in\N$ and any sequence  $(c_k)_{k\in\N}$ with $| c_k|\leq 1/k!$, we have 
	\begin{equation}\label{brubeck}
	\abs{\sum_{k\geq h} c_k \ad^k_S\pa{H}}_{r,\eta,\tw} \le 
	2 |H|_{r+\rho,\eta,\tw} \big(|S|_{r+\rho,\eta,\tw}/2\delta\big)^h
	\,,
	\end{equation}
	where  $\ad_S\pa{\cdot}:= \set{S,\cdot}$.
\end{lemma}
\begin{proof}
	For brevity we set, for every $r'>0$
	$$
	|\cdot|_{r'}:=|\cdot|_{r',\eta,\tw}\,.
	$$
	We use Lemma \ref{palis}, with $E\to \th_{\mathtt w}$, $X\to X_S$,
	$\delta_0\to (r+\rho) |S|_{r+\rho},$ $r\to r+\rho,$ $r_1\to r,$	
	$T\to 8e.$
	Then the fact that the time $1$-Hamiltonian flow 
	$\Phi^1_S: B_r(\th_{\mathtt w})
	\to B_{r + \rho}(\th_{\mathtt w})$  is well defined, analytic, symplectic  follows,	 since for any $\eta\ge 0$
	\[
	\sup_{u\in  B_{r+\rho}(\th_{\mathtt w})}
	|X_S|_{\th_{\mathtt w}}
	\le (r+\rho) |S|_{r+\rho}<\frac{\rho}{8 e}\,.
	\]
	Regarding the estimate \eqref{pollon}, again by 
	Lemma \ref{palis}  (choosing $t=1$), we get
	\[
	\sup_{u\in  B_{r}(\th_{\mathtt w})}
	\abs{\Phi^1_S(u)-u} _{\th_{\mathtt w}}
	\le
	(r+\rho) |S|_{r+\rho}
	<\frac{\rho}{8 e}
	\,.
	\]

	Estimates \eqref{tizio},\eqref{caio},\eqref{sempronio} directly follow by \eqref{brubeck} with $h=0,1,2,$
	respectively and $c_k=1/k!$, 
	recalling that by Lie series 
	$$
	H \circ \Phi^1_S = e^{\rm ad_S} H = \sum_{k=0}^\infty \frac { {\rm ad}_S^k H}{k!} =
	\sum_{k=0}^\infty \frac {  H^{(k)}}{k!}\,,
	$$
	where
	$ H^{(i)} := {\rm ad}_S^i (H)= {\rm ad}_S ( H^{(i-1)}) $,  $ H^{(0)}:=H $.\\
	Let us prove \eqref{brubeck}.
	Fix $k\in\N,$ $k>0$ and set
	$$
	r_i := r +\rho(1 - \frac{i}{k}) \,
	\, ,  \qquad  i = 0,\ldots,k \, .
	$$
	Note that, by the monotonicity of the norm (recall Remark \ref{cacioepepe})
	\begin{equation}\label{morello}
	|S|_{r_i}\leq |S|_{r+\rho}\,,\qquad
	\forall\,  i = 0,\ldots,k\,.
	\end{equation}
	Noting that
	\begin{equation}\label{dave}
	1+\frac{k r_i}{\rho} \,
	\leq
	k \pa{1+\frac{r}{\rho }}\,,
	\qquad \forall\, i=0,\ldots,k\,,
	\end{equation}
	by using $k$ times \eqref{commXHK}  we have
	\begin{eqnarray*}
		| {H^{(k)}}|_r
		&=& 
		|   \{S, {H^{(k-1)}}\} |_r
		\leq  
		4 (1+\frac{ k r}{\rho})
		|{H^{(k-1)}}|_{r_{k-1}}|
		S|_{r_{k-1}}
		\\
		&\stackrel{\eqref{morello}}\leq&
		|H|_{r+\rho}
		|S|_{r+\rho}^k
		4^k
		\prod_{i=1}^k
		(1+\frac{ k r_i}{\rho})
		\stackrel{\eqref{dave}}\leq
		|H|_{r+\rho}
		\left(
		4k \pa{1+\frac{r}{\rho }}|S|_{r+\rho}
		\right)^k
		\,.
	\end{eqnarray*} 
	Then, using $ k^k\leq e^k k!, $
	we  get
	\begin{eqnarray*}
		\left|\sum_{k\geq h} c_k {H^{(k)}}\right|_{r} &\leq&
		\sum_{k\geq h} |c_k| |{H^{(k)}}|_{r}
		\leq
		|H|_{r+\rho} \sum_{k\geq h} \left(
		4e \pa{1+\frac{r}{\rho }}|S|_{r+\rho}
		\right)^k
		\\
		&= & | H|_{r+\rho} \sum_{k\geq h} (|S|_{r+\rho}/2\delta)^k
		\stackrel{\eqref{stima generatrice}}\leq 2 |H|_{r+\rho} (|S|_{r+\rho}/2\delta)^h\,.
	\end{eqnarray*}	
Finally, if $ S $ and $ H $ satisfy mass conservation so does  each
	$ {\rm ad}_S^k H $, $ k \geq 1 $, hence $ H \circ \Phi^1_S $ too.
\end{proof}

\subsection{Nemitskii operators}\label{nemo}
Now we show that the nonlinearities in \eqref{NLSb} are bounded in the norm $|\cdot|_{r,\eta,\tw}$ in the cases $\SO,\MS,\GE$.
For any $ 0\leq p \le p', 0 \le s \le s', 0 \le a \le a' $ we have 
\begin{equation}
\th_{p', s', a'} \subseteq \th_{p, s, a}
\end{equation}
and 
\begin{equation*}
\abs{v}_{p, s, a} \le \abs{v}_{p', s', a'},\quad \forall v\in \th_{p', s', a'}.
\end{equation*}
\smallskip
For $p>1$ let
$\star:\th_{p, s, a} \times \th_{p, s, a} \to \th_{p, s, a}$ be the convolution operation defined as
\[
\pa{f,g} \mapsto f\star g :=\pa{\sum_{{j_1,j_2\in \Z \,,\;  j_1+j_2=j}} f_{j_1}g_{j_2}}_{j\in \Z}.
\]
The map $\star: \pa{f,g} \mapsto f\star g$ is continuous in the following sense:
\begin{lemma}\label{veronese}
	For $p>1/2$ we have 
	\begin{equation}\label{Calg}
	\norm{f\star g}_{p,s,a} \le \Calg \norm{f}_{p,s,a}\norm{g}_{p,s,a}\,,
	\,.
	\end{equation}
\end{lemma}
The proof is given in Appendix \ref{tec1}.

\begin{lemma}\label{tiepolo}
	For $p>1/2$ and $f,g\in \th^p$
	\begin{equation}\label{CalgM}
	\|f\star g\|_p\leq \CalgM\|f\|_p \|g\|_p
	\,.
	\end{equation}
\end{lemma}
The proof is given in Appendix \ref{latempesta}.

\begin{lemma}[Nemitskii operators]\label{neminchia}
	Let $p> 1/2.$
	(i) Fix $s\ge 0, a_0\ge 0$. Consider  a sequence $F^{(d)}=\pa{F^{(d)}_j}_{j\in \Z}\in \th_{p,s,a_0}$, $d\geq 1,$ such that
	\begin{equation}
	\label{piffero}
	\sum_{d=1}^\infty d |F^{(d)}|_{p,s,a_0} R^{d} < \infty
	\end{equation} 
	for some $R>0$.
	\\ 
	For $u=\pa{u_j}_{j\in \Z}$ let 
	$\bar u= \pa{\overline{u_{-j}}}_{j\in \Z}$  and  consider the  Hamiltonian
	\[	H(u)=   \sum_{d=1}^\infty\pa{F^{(d)}\star\underbrace{ u \star \cdots\star u}_{d \;\mbox{times}} \star \underbrace{\bar u \star \cdots\star \bar u}_{d \;\mbox{times}}}_0\,.
	\]
	For all  $(\eta,a,r)$ such that  $\eta+a \le  a_0$ 
	and
	$(\Calg r)^2 \leq R$, 
	we have that $H\in \cH_{r,p,s,a,\eta}$ and
	\[
	|H|_{r,p,s,a,\eta}
	\le 
	r^{-1}\sum_{d=1}^\infty d |F^{(d)}|_{p,s,a_0}  (\Calg r)^{2d-1}
	<\infty.
	\]
	(ii) Analogously if $F^{(d)}$ are constants satisfying
	\begin{equation}
	\label{piffero2}
	\sum_{d=1}^\infty d 
	|F^{(d)}| R^{d} < \infty
	\end{equation}
	and $(\CalgM r)^2\leq R,$
	then
	$H\in \cH^{r,p}$ with 
	\begin{equation}\label{soralella}
	\|H\|_{r,p}
	\le 
	2^p r^{-1}\sum_{d=1}^\infty d |F^{(d)}|
	(\CalgM r)^{2d-1}<\infty.
	\end{equation}

\end{lemma}
\begin{proof}
	In Appendix \ref{tec2}
\end{proof}
\begin{cor}\label{neminchione}
	Consider the correction term $P= \int_{\T}F(x,|u|^2)dx$ in the NLS Hamiltonian \eqref{hamNLS}, where the argument $f$ in $F$ satisfies\eqref{analitico}.
	Let $p> 1/2$.	
	\\
	(i)
	For any   $a,s,\eta\ge 0$  such that $a+\eta<\ta$  and any $r>0$ 
	such that\footnote{$R$ defined in \eqref{analitico}.} 
	$(\Calg r)^2\le  R$, we have 
	\begin{equation}
	\label{stimazero}
	| P|_{r,p,s,a,\eta} 
	\leq \CNem(p,s,\ta- a -\eta)\frac{(\Calg r)^2}{R}|f|_{\ta,R}< \infty.
	\end{equation} 
	where $f$ and $|f|_{\ta,R}$ are defined in \ref{analitico}.
	
	(ii) If $F$ is independent of\footnote{i.e. $P$ preserves momentum and we are assuming
		\eqref{analiticobis}.} $x$,  for $(\CalgM r)^2\leq R$ we have
	\begin{equation}
	\label{stimazero2}
	\| P\|_{r,p}
	\leq 2^p \frac{(\CalgM r)^2}{R}|f|_{R}< \infty\,.
	\end{equation} 
\end{cor}
\begin{proof}
	By definition (recall \eqref{analitico} and \eqref{hamNLS})
	\begin{equation}\label{nina}
	F(x,y)=\int_0^y f(x,s) ds = \sum_{d=2}^\infty \frac{f^{(d-1)}(x)}{d} y^{d} =: \sum_{d=2}^\infty F^{(d)}(x) y^{d}
	\end{equation}
	therefore we have
	\[
	P= \int_{\T}F(x,|u|^2)dx=  \sum_{d\ge 2}\pa{F^{(d)}\star\underbrace{ u \star \cdots\star u}_{d \;\mbox{times}} \star \underbrace{\bar u \star \cdots\star \bar u}_{d \;\mbox{times}}}_0\,.
	\]
	To each analytic function $F^{(d)}(x)$ we associate its Fourier coefficients; we have $\pa{F^{(d)}_j}_{j\in \Z}\in \th_{p,s,a_0}$ for 
	$ a_0:=a+\eta<\ta$ and $s,p\ge 0$. Indeed
	\begin{align*}
	|F^{(d)}|_{p,s,a_0}^2 &
	:=
	\sum_{j} e^{2a_0 |j|+ 2s \jap{j}^\teta}\jap{j}^{2p} |F^{(d)}_j|^2
	\stackrel{\eqref{nina}}=
	\sum_{j} e^{2a_0 |j|+ 2s \jap{j}^\teta}\jap{j}^{2p} 
	\frac{|f^{(d-1)}_j|^2}{d^2}
	\\ 
	& \le \frac{c^2(p,s,\ta- a_0)}{d^2} \sum_{j} 
	e^{2\ta  |j|} |f^{(d-1)}_j|^2 = \frac{c^2(\ta- a_0,s,p)}{d^2}|f^{(d-1)}|^2_{\T_{\mathtt a}}\,
	\end{align*}
	with 
	\[
	c(p,s,t):=e^s+ \sup_{x\ge 1} x^p e^{-t x+s x^\theta}
	\]
	Now condition \eqref{analitico} ensures that \eqref{piffero} holds and our claim follows, by Lemma \ref{neminchia}, setting $a_0= a+\eta$.	 
	\\
	(ii)
	Follows from \eqref{soralella}.
	
\end{proof}


\section{Monotonicity properties.}\label{monotono}
Given two positive sequences $\tw = \pa{\tw_j}_{j\in\Z},\tw' = \pa{\tw'_j}_{j\in\Z}$
we write that $\tw\leq \tw'$ if the inequality holds
point wise, namely
$$
\tw\leq \tw' \quad
:\iff\quad
\tw_j\leq \tw'_j\,,\ \ \ \forall\, j\in\Z\,.
$$
In this way if $r'\le r$ and $\tw\leq \tw'$ 
then $B_{r'}(\th_{\mathtt w'}) \subseteq B_r(\th_\tw)$.
Consequently if $r'\le r , \eta'\le \eta $ and $\tw\leq \tw'$ then
$ \cA_{r,\eta}(\th_\tw) \subseteq \cA_{r',\eta'}(\th_{\tw'})$. \\
We thus wish to study conditions on  $(r,\eta,\tw),(\rs,\eta',\tw')$ (with $\rs\le r$) which ensure that  $ \cH_{r,\eta}(\th_\tw) \subseteq \cH_{\rs,\eta'}(\th_{\tw'})$.
Note that this is not obvious at all, since we are asking that an Hamiltonian 
vector field of $X_H\in\cH_{r,\eta}(\th_\tw)$, 
when restricted to the smaller domain $B_{\rs}(\th_{\mathtt w'})$  belongs to
the smaller space $\th_{\tw'}$.\smallskip\\
Let us start by rewriting the norm $|\cdot|_{r,\eta,\tw}$
 in a more {\it adimensional} way.
\begin{defn}\label{tufo} 
For any $H\in \cH_{r,\eta}(\th_{\mathtt w})$ 
 we define  a map
	\[
	B_1(\ell^2)\to \ell^2 \,,\quad y=\pa{y_j}_{j\in \Z }\mapsto 
	 \pa{Y^{(j)}_{H}(y;r,\eta,\tw)}_{j\in \Z}
	\]
	by setting
	\begin{equation}
	Y^{(j)}_{H}(y;r,\eta,\tw) := \sum_\ast |H_{\bal,\bbt}| \frac{(\bal_j+\bbt_j)}{2}c^{(j)}_{r,\eta,\tw}(\bal,\bbt) y^{\bal+\bbt-e_j}
	\end{equation}
	where  $e_j$ is the $j$-th basis vector in $\N^\Z$, while the coefficient
\begin{equation}
c^{(j)}_{r,\eta,\tw}(\bal,\bbt)
	:=
	r^{|\bal|+|\bbt|-2}e^{\eta|\pi(\bal-\bbt)|}
	\frac{\tw_j^2}{\tw^{\bal+\bbt}}
	\label{persico}
	\end{equation}
	is defined for any $\bal,\bbt\in\N^\Z$. For brevity, we set
	$$
	\sum_\ast:=\sum_{\bal,\bbt: |\bal|=|\bbt|}\,.
	$$ 
The momentum
	$\pi(\cdot)$ was defined in \eqref{momento}.
\end{defn}

\begin{lemma}\label{stantuffo}	 Let
 $\ri,\rs>0,\,\etai,\etaf\geq 0,$ $\twi,\twf\in \R_+^\Z$. The following properties hold.
	\begin{enumerate}[(i)]
		\item   The norm of $H$ can be expressed as
		\begin{equation}\label{ypsilon}
		\abs{H}_{r,\eta,\tw}= 
		\sup_{|y|_{\ell^2}\le 1}\abs{Y_H(y;r,\eta,\tw)}_{\ell^2}
		\end{equation}
	\item  Given 
	$
	H^{(1)}\in \cH_{\rs,\etaf,\twf}$ 
	and $H^{(2)}\in \cH_{\ri,\etai,\twi}\,,
$
\\	
	such that for all $\bal,\bbt\in \N^\Z$ and  $j\in \Z$ with $\bal_j+\bbt_j\neq 0$ 
	one has
	\[
	|H^{(1)}_{\bal,\bbt}| c^{(j)}_{\rs,\etaf,\twf}(\bal,\bbt)  
	\le 
	c
	|H^{(2)}_{\bal,\bbt}| c^{(j)}_{\ri,\etai,\twi}(\bal,\bbt),
	\]
	for some $c>0,$
	then
	\[
	|H^{(1)}|_{\rs,\etaf,\twf}
	\le 
	c
	|H^{(2)}|_{\ri,\etai,\twi}\,.
	\]
\end{enumerate}
\end{lemma}
\begin{proof}
See appendix \ref{tec3}.
\end{proof}

As a corollary we get the following

\begin{prop}\label{cacioricotta}
 Let
 $\ri,\rs>0,\,\etai,\etaf\geq 0,$ $\twi,\twf\in \R_+^\Z.$
If
\begin{equation}\label{zucchina}
C:=\sup_{\substack{j\in\Z,\, \bal,\bbt\in\N^\Z\\\bal_j+\bbt_j\neq 0}}
\frac{c^{(j)}_{\rs,\etaf,\twf}(\bal,\bbt) }{c^{(j)}_{\ri,\etai,\twi}(\bal,\bbt) }
< \infty\,,
\end{equation}
then
\begin{equation}\label{cetriolo}
|H|_{\rs,\etaf,\twf}
	\le 
	C
	|H|_{\ri,\etai,\twi}\,.
\end{equation}
 In particular
 $\abs{\cdot}_{r,\eta,\tw}$ is  increasing in $r$ and $\eta$, namely
 if
 $\rs\leq \ri$ and $\etaf\leq\etai$ then
 $$
 |H|_{\rs,\etaf,\tw}
	\le 
	C
	|H|_{\ri,\etai,\tw}.
 $$
 Moreover, if $\rs\leq \ri,$  $\twi\leq \twf$ and 
$H\in\cK_{r,\eta}(\th_\tw)$ then
\begin{equation}\label{cetriolobis}
|H|_{\rs,\etaf,\twf}
	\le 
	|H|_{\ri,\etai,\twi}\,.
\end{equation}
Furthermore, if $H$ preserves momentum
then
\begin{equation}\label{cetriolo0}
|H|_{\rs,\etaf,\twf}
	\le 
	\CM 
	|H|_{\ri,\etai,\twi}\,,
\end{equation}
where
\begin{equation}\label{zucchina0}
\CM :=
\sup_{\substack{j\in\Z,\, \bal,\bbt\in\N^\Z,\\ \bal_j+\bbt_j\neq 0,
\\ \sum_{i}i(\bal_i-\bbt_i)=0}}
\frac{c^{(j)}_{\rs,\etaf,\twf}(\bal,\bbt) }{c^{(j)}_{\ri,\etai,\twi}(\bal,\bbt) }
< \infty\,,
\end{equation}
\end{prop}
\begin{proof}
Inequality \eqref{cetriolo} directly follows from 
Lemma \ref{stantuffo} (ii), while 
\eqref{cetriolobis} follows directly by \eqref{persico}
since in the kernel $\bal_j+\bbt_j\neq 0$
implies $\bal_j+\bbt_j\geq 2.$ The momentum preserving case follows analogously.
\end{proof}

\begin{defn}[minimal scaling degree]\label{scialla}
	We say that $H$ has { \sl minimal }  
	 scaling degree $\td=\td(H)$ (at zero) if 
	\begin{eqnarray*}
&&H_{\bal,\bbt}=0 \,,\quad \forall \; \bal,\bbt: \quad |\bal|=|\bbt|\leq \td 
	\,,
	\\
	&&H_{\bal,\bbt}\neq 0 \,,\quad \text{for some}\ \  \bal,\bbt: \quad |\bal|=|\bbt|= \td+1\,.
\end{eqnarray*}
We say that $\td(0)=+\infty.$
\end{defn}

\begin{lemma}\label{gasteropode}
	If $H\in \cH_{r,\eta}(\th_\tw)$ with $\td(H)\geq\td$, then for all $\rs\le r$ one has
\begin{equation*}
	\abs{H}^{\wc}_{\rs,\eta,\tw} \le \pa{\frac{\rs}{r}}^{2\td} \abs{H}^{\wc}_{r,\eta,\tw}\,.
\end{equation*}
\end{lemma}
\begin{proof}
Recalling \eqref{persico}, we have
	\[
	\frac{c^{(j)}_{\rs,\eta,\tw}(\bal,\bbt)}{c^{(j)}_{r,\eta,\tw}(\bal,\bbt)}= \pa{\frac{\rs}{r}}^{|\bal|+|\bbt|-2}\,.
	\]
Since $|\bal|+|\bbt|-2\ge 2\td$, the inequality
follows by Proposition \ref{cacioricotta}.
\end{proof}

\subsection{Monotonicity of $\abs{\cdot}_{r,p,s,a,\eta}$ and $\|{\cdot}\|_{r,p}$}

In this section we prove properties of monotonicity for the norms used in the Sobolev, Modified Sobolev and Gevrey cases introduced in Section \ref{intro}. To prove such properties we strongly rely on some notation and results introduced by Bourgain in \cite{Bourgain:2005} and extended later on by Cong-Li-Shi-Yuan in \cite{Yuan_et_al:2017} (Definition \ref{n star} and Lemma \ref{constance general} below).
\\
\begin{defn}\label{n star}
	Given a vector $v=\pa{v_i}_{i\in \Z}$  $v_i\in \N$, $|v|<\infty$  we denote by $\na=\na(v)$ the vector $\pa{\na_l}_{l\in I}$ (where $I\subset \N$ is finite)  which is the decreasing rearrangement 
	of
	$$
	\{\N\ni h> 1\;\; \mbox{ repeated}\; v_h + v_{-h}\; \mbox{times} \} \cup \set{ 1\;\; \mbox{ repeated}\; v_1 + v_{-1} + v_0\; \mbox{times}  }
	$$
\end{defn}
\begin{rmk}
	A good way of envisioning this list is as follows. Given $v=\pa{v_i}_{i\in \Z}$ consider the monomial $\fm(v):= \prod_i x_i^{v_i}$. We can write uniquely 
	\[
\fm(v)= \prod_i x_i^{v_i} = x_{j_1} x_{j_2}\cdots x_{j_{|v|}}
	\] 
	then $\na(v)$ is the decreasing rearrangement of the list $\pa{\jap{j_1},\dots,\jap{j_{|v|}}}$.

		As an example, consider the case $v\neq 0$. Then, by construction there exists a unique  $J\ge 0$ such that $v_j=0$ for all $|j|>J$ and $v_{J}+ v_{-J}\ne 0$  hence 
		\[
		v=(\dots,0, v_{-J},\dots,v_0,\dots,v_J,0\dots)\,.
		\]
		If $J=0$ then 
		\[
		\na= (\underbrace{1,\dots,1}_{v_0 \; \rm{times}})
		\]
		otherwise we have
		\[
		\na= (\underbrace{J,\dots,J}_{v_J+ v_{-J} \; \rm{times}}, \underbrace{J-1,\dots,J-1}_{v_{J-1}+ v_{-J+1} \; \rm{times}},\dots,\underbrace{1,\dots,1}_{v_1+ v_{-1}+ v_0 \; \rm{times}})
		\]
\end{rmk}
Given $\bal,\bbt\in\N^\Z$ 
with $1\leq |\bal|=|\bbt|<\infty,$
from now on we define
$$
\na=\na(\bal+\bbt)\,.
$$ 
We set the {\sl even} number
$$
N:=|\bal|+|\bbt|\,,
$$
which is the cardinality of $\na.$

We  observe that, given
\[ 
\pi = \sum_{i\in \Z} i\pa{\bal_i - \bbt_i}= \sum_{h> 0} h \pa{\bal_h - \bbt_h - \bal_{-h} + \bbt_{-h}} \,,
\] 
there exists a choice of $\s_i = \pm1, 0$ such that 
\begin{equation}\label{pi e cappucci}
\pi = \sum_l \sigma_l\na_l.
\end{equation}
with $\sigma_l \neq 0$  if $\na_l \neq 1$.
Hence, 
\begin{equation}\label{eleganza}
\na_1\le \abs{\pi} + \sum_{l\ge 2}\na_l.
\end{equation}
Indeed, if $\sigma_1 = \pm 1$, the inequality follows directly from \eqref{pi e cappucci}; if $\sigma_1 = 0$, then $\na_1=1$ and consequently $\na_l = 1\, \forall l$. Since the mass is conserved, the list $\na$ has at least two elements, and the inequality is achieved.
\begin{lemma}[Constance generalizzato]\label{constance general}
	Given $\bal,\bbt$ such that $\sum_i i (\bal_i-\bbt_i)=\pi\in\Z$, we have that setting $\na=\na(\bal+\bbt)$
	\begin{equation}\label{yuan 2}
	\sum_i \jap{i}^\theta(\bal_i+\bbt_i) =\sum_{l\ge 1} \na_l^\theta  \ge 2 \na^\theta_1+ (2-2^\teta) {\sum_{l\ge 3} \na_l^\theta} -\theta{\abs{\pi}}.
	\end{equation}
\end{lemma}
\begin{proof}
	In appendix \ref{costi1}.
\end{proof}

The following proposition gathers the monotonicity properties of the norm
$|\cdot|_{r,p,s,a,\eta}$ with respect to the parameters 
$p,s,a.$. 

\begin{prop}\label{crescenza}
	The following inequalities hold:
	\begin{enumerate}
			\item { \sl Variations w.r.t. the paramater $p$.} For any $0<\rho<r$ , $0<\sigma< \eta$ and $p_1>0$ we have
		\begin{equation*}
		\abs{H}_{r-\rho,p+p_1,s,a,\eta-\sigma}\le 
		\Cmon(r/\rho, \s,p_1) \abs{H}_{r,p,s,a,\eta}\,.
		\end{equation*}

		\item  { \sl Variation w.r.t. the parameter $s$.} For any $0<\sigma< \eta$ we have
		\begin{equation}\label{emiliaparanoica}
		|H|_{r,p, s+\s,a,\eta-\s} \le |H|_{r,p,s,a,\eta}
		\end{equation}
		\item { \sl Variation w.r.t. the parameter $a$.} 
		 For any $0<\sigma< \eta$ 
		\begin{equation}
		|H|_{e^{-\s}r,p,s,a+\s, \eta -\s} \le e^{2\s}|H|_{r,p,s,a,\eta}
		\end{equation}
	\end{enumerate}
\end{prop}

\begin{proof}
In all that follows we shall use systematically the fact that our Hamiltonians  preserve the mass and are zero at the origin. These facts imply that $\abs{\bal} = \abs{\bbt} \ge 1$.
	\\
	\gr{Item $(1)$} 
	First we  assume that $\rho\leq r/2.$
	By Lemma \ref{stantuffo} item (ii) we only need to show that, for any $0<\rho\leq r/2$ , $0<\sigma< \eta$ and $p_1>0$ there exists a constant 
	$\Cmon$ such that
	\begin{equation}\label{mulo}
	c^{(j)}_{r-\rho,p+p_1,s,a,\eta-\s }(\bal,\bbt) \le \Cmon 
	c^{(j)}_{r,p,s,a,\eta }(\bal,\bbt)
	\end{equation}
 for all $j$, $\bal,\bbt$ with $\abs{\bal} = \abs{\bbt} \ge 1$ and $\bal_j+\bbt_j\neq 0$.
	In order to prove our claim we need to control
	\begin{equation}\label{borgogna}
	\sup_{\substack {j,\bal,\bbt\\\bal_j+\bbt_j\neq 0}} \pa{\frac{\jap{j}^2}{\prod_i\jap{i}^{\bal_i+\bbt_i}}}^{p_1} e^{-\s |\pi|} \pa{\frac{r-\rho}{r}}^{|\bal|+|\bbt|-2}\,.
\end{equation}
We use the notations of Definition \ref{n star}, with $\na(\bal+\bbt)\equiv \na$.	Since $\bal_j+\bbt_j\neq 0$ we have that $\jap{j}\le \na_1$.  
Note that 
\begin{equation}\label{fiorentina}
\prod_i\jap{i}^{\bal_i+\bbt_i}= \prod_{l\ge 1}\na_l\,.
\end{equation}
 Hence
	\[
	{\frac{\jap{j}^2}{\prod_i\jap{i}^{\bal_i+\bbt_i}}}\le \frac{\na_1}{\prod_{l\ge 2}\na_l}
	\]
	Let us call $N=|\bal|+|\bbt|\geq 2$. 
		By \eqref{eleganza} we have that
\begin{equation}
	\label{stimap}
	\sup_{\substack {j,\bal,\bbt\\\bal_j+\bbt_j\neq 0}} {\frac{\jap{j}^2}{\prod_i\jap{i}^{\bal_i+\bbt_i}}}\le	\frac{\na_1}{\prod_{l\ge 2}\na_l} \le \frac{\sum_{l=2}^{N}\na_l +|\pi|}{\prod_{l= 2}^{N}\na_l}\le \frac{(N-1)\na_2 +|\pi|}{\prod_{l= 2}^{N}\na_l}
	\le \frac{N+|\pi|}{\prod_{l=3}^{N} \na_l}\,.
\end{equation}
We have shown that
	$$
	\sup_{\substack {j,\bal,\bbt\\\bal_j+\bbt_j\neq 0}} {\frac{\jap{j}^2}{\prod_i\jap{i}^{\bal_i+\bbt_i}}}
	\leq N+|\pi|
	\,.
	$$
		Since $\pa{N +|\pi|}^{p_1}\leq 2^{p_1}(N^{p_1}+|\pi|^{p_1})$, denoting $L:=\ln \pa{r/r-\rho}$ we repeatedly use Lemma \ref{chiappa} in order to 
	control
	\begin{align}\label{mizzica}
	&\sup_{N\geq 2,\pi\in \Z} \pa{N +|\pi|}^{p_1} e^{-\s|\pi|} \pa{\frac{r-\rho}{r}}^{N-2}
	\\
	&\leq 2^{p_1}
	\pa{
	\sup_{N\geq 2\, ,\pi\in \Z} N^{p_1} e^{-\s|\pi| - L(N-2) }
	+
	\sup_{N\geq 2\, ,\pi\in \Z} |\pi|^{p_1} e^{-\s|\pi| - L(N-2) }
	}\nonumber
	\\
	&\leq  2^{p_1}
	\pa{\max\left\{\pa{\frac{p_1}{L}}^{p_1} ,1  \right\}+\pa{\frac{p_1}{\s}}^{p_1}}
	\leq
	  2^{p_1+1}
	\max\left\{\pa{\frac{p_1}{L}}^{p_1} ,  \pa{\frac{p_1}{\s}}^{p_1}, 1\right\}\nonumber
	\\
	&\leq  
	  2^{p_1+1} p_1^{p_1}
	\max\left\{\pa{\frac{2 r}{\rho}}^{p_1} ,  \pa{\frac{1}{\s}}^{p_1}, 1\right\}=
	\Cmon\,,\nonumber
	\end{align}
	using that
	$$
	L\geq\ln(1+\rho/r)\geq 2\ln(3/2)\rho/r
	\geq \rho/2r\,,
	$$
	which holds since
	we are in the case $\rho\leq r/2.$
	This completes the proof in the case $\rho\leq r/2.$
	
	Consider now the case
	$r/2<\rho<r.$
	Using the monotonicity of the norm w.r.t. $r$ and the already proved case with
	$\rho=r/2$, we have
	\begin{align*}
		&\abs{H}_{r-\rho,p+p_1,s,a,\eta-\sigma}\le
		\abs{H}_{r/2,p+p_1,s,a,\eta-\sigma}\le  
		2^{p_1+1}
	\max\left\{\pa{4p_1}^{p_1} ,  \pa{\frac{p_1}{e\s}}^{p_1}, 1\right\} \abs{H}_{r,\eta,\tw}
	\\
	&\leq
	2^{p_1+1} p_1^{p_1}
	\max\left\{\pa{\frac{2 r}{\rho}}^{p_1} ,  \pa{\frac{1}{\s}}^{p_1}, 1\right\}
	\abs{H}_{r,\eta,\tw}
	\,,
		\end{align*}
	proving (1) also in the case 
	$r/2<\rho<r.$

	\gr{Item $(2)$} We need to show that
	\[
	c^{(j)}_{r,p,s+\s,a,\eta-\s}(\bal,\bbt) \le  c^{(j)}_{r,p,s,a,\eta }(\bal,\bbt)
	\]namely that
	\begin{equation}\label{gigina}
	\exp(-\s (\sum_i \jap{i}^\theta (\bal_i+\bbt_i) -2\jap{j}^\theta+ |\pi(\bal-\bbt)|) \le 1
	\end{equation}
	This follows by \ref{constance general} since
	\begin{equation}\label{stima1}
	\sum_i \jap{i}^\theta (\bal_i+\bbt_i) -2\jap{j}^\theta+ |\pi(\bal-\bbt)|\ge(1-\theta) \pa{\sum_{l\ge 3} \na_l^\theta +|\pi|} \ge  0
	\end{equation}
	for all $\bal,\bbt$ in $\sum_\ast$ such that $\bal_j+\bbt_j\ne 0$.

\gr{Item $(3)$} We proceed as in item $(1)-(2)$, 
	our claim follows if we  can show that
	\begin{equation}
	\sum_i \pa{\bal_i + \bbt_i}\abs{i} - 2\abs{j} + \abs{\pi} \ge  \sum_{l\ge 2} \na_l - \na_1 + \abs{\pi} - \abs{\bal_0 + \bbt_0} \ge - \pa{\abs{\bal} + \abs{\bbt}},
	\end{equation}
	This is proved in  formula \eqref{eleganza}.
\end{proof}
Incidentally we note that norm
$|\cdot|_{r,p,s,a,\eta}$ possesses 
the tameness property.
\begin{prop}\label{maspero}
		\[
		\sup_{|u|_{p_0,s,a}\le r }\frac{|X_{\und{H}}|_{p,s,a}}{|u|_{p,s,a}} \le \Ctame(\rho,\eta,p)|H|_{r+\rho,p_0,s,a,\eta}
		\]
\end{prop}
\begin{proof}
 In  Appendix \ref{app:maspero}.
\end{proof}

\begin{prop}
The norm $\|\cdot\|_{r,p}$ is 
 monotone decreasing  in $p$, namely
 $\|\cdot\|_{r,p+p_1}\leq \|\cdot\|_{r,p}$ for any $p_1>0$.
 \end{prop}
\begin{proof}
For the norm $\|\cdot\|_{r,p}$ the quantity in \eqref{persico} becomes
\begin{equation}\label{persico2}
\mathtt c^{(j)}_{r,p}(\bal,\bbt):=
r^{|\bal|+|\bbt|-2}
\pa{\frac{\jml{j}^{2}}{\prod_{i\in\Z}{\jml{i}}^{\pa{\bal_i+\bbt_i}}}}^p
\,. 
\end{equation}
	By Lemma \ref{stantuffo} item (ii) we only need to show that
		\begin{equation}\label{vacca}
	c^{(j)}_{r,p+p_1}(\bal,\bbt) \le 
	c^{(j)}_{r,p}(\bal,\bbt)
	\end{equation}
 for all $j$, $\bal,\bbt$ with $\abs{\bal} = \abs{\bbt} \ge 1$ and $\bal_j+\bbt_j\geq 1$ (recall the momentum conservation), namely we have to prove that
\begin{equation}\label{pierodellafrancesca}
	\sup_{\substack {j,\bal,\bbt\\\bal_j+\bbt_j\geq 1 }} 
	\frac{\jml{j}^2}{\prod_i\jml{i}^{\bal_i+\bbt_i}} \leq 1\,.
\end{equation}
		We first show that the inequality holds in the case $j=0,\pm 1.$ 
		Indeed we have 
		$$
		\prod_i\jml{i}^{\bal_i+\bbt_i}
		\geq \prod_i 2^{\bal_i+\bbt_i}
		=2^{\sum_i \bal_i+\bbt_i}
		\geq 4
		$$
		since $\sum_i \bal_i+\bbt_i\geq 2$
		(by the fact that $\abs{\bal} = \abs{\bbt} \ge 1$).
		\\
		Consider now the case $|j|=\jml{j}\geq 2.$
	Since $\bal_j+\bbt_j\geq 1$, inequality \eqref{pierodellafrancesca}
	follows by
\begin{equation}\label{pierodellafrancesca2}
	\sup_{j,\bal,\bbt } 
	\frac{|j|}{\prod_{i\neq j}\jml{i}^{\bal_i+\bbt_i}} \leq 1\,.	
	\end{equation}
By momentum conservation we have
\begin{equation}\label{facioli}
|j|\leq \sum_{i\neq j} |i|(\bal_i+\bbt_i)
\leq \sum_{i\neq j} \jml{i}(\bal_i+\bbt_i)
\end{equation}
and \eqref{pierodellafrancesca2}
follows if we show that 
\begin{equation}\label{ceci}
\sup_{j,\bal,\bbt } 
	\frac{\sum_{i\neq j} \jml{i}(\bal_i+\bbt_i)}{\prod_{i\neq j}\jml{i}^{\bal_i+\bbt_i}} \leq 1\,,
\end{equation}
where we can restrict the sum and the product to the indexes $i$
such that $\bal_i+\bbt_i\geq 1.$
This last estimates follows by the fact that
given $x_k\geq 1$
$$
\frac{\sum_{2\leq k\leq n} k x_k}{\prod_{2\leq k\leq n}k^{x_k}}
\leq 1\,,
$$
as it can be easly proved by induction over $n$
(noting that $n^x\geq nx$ for $n\geq 2,$ and any $x\geq 1$).
\end{proof}



\section{Small divisors and homological equation}\label{omologo}

Let us consider the set of frequencies 
	\begin{equation}
	\Omega_\fp:=\set{\omega=\pa{\omega_j}_{j\in \Z}\in \R^\Z,\quad \sup_j|\omega_j-j^2|\jap{j}^\fp < 1/2 };
	\end{equation}
	this set is isomorphic to $[-1/2,1/2]^\Z$ via the identification 
	 \begin{equation}\label{omegaxi}
	 \xi\mapsto \omega(\xi)\,,\quad \mbox{where}\quad \omega_j(\xi) = j^2 +\frac{\xi_j}{\jap{j}^\fp}\,.
	 \end{equation}
	 We endow $\Omega_\fp$ with the probability measure
	 $\mu$
	 induced\footnote{Denoting  by $\mu$ the measure in $\Omega_\fp$ and by $\nu$ the product measure on $[-1/2,1/2]^\Z$, then $\mu(A)= \nu(\omega^{(-1)}(A))$ for all sets $A\subset\Omega_\fp$ such that  $\omega^{(-1)}(A)$ is $\nu$-measurable. }  by the product measure on $[-1/2,1/2]^\Z$. 

The dependence w.r.t. the parameters $s,p $, thanks to Lemma \ref{constance general} and formulae \eqref{stima1} and \eqref{stimap}, works like an ultraviolet cut-off in the following sense. If a Hamiltonian has $\na_3>N $  for all $H_{\bal,\bbt} \neq 0$ then  its norm is  of order  $\le e^{-sN^\teta} N^{-p}$. This kind of restriction on Hamiltonians is often used in small divisor problems; since the denominators can be bounded from below in terms of $\na_3$, see for example \cite{Bambusi-Grebert:2006}, \cite{Faou-Grebert:2013}, then the solution of the homological equation is controlled. Actually we shall not use this  kind of cut-off but instead, following Bourgain, we rely on a refined version of Lemma \ref{constance general} (see Lemma \ref{constance 2 gen} below), in order to deal with the small denominators. 

We now define the set of {\sl Diophantine} frequencies, the following definition is a slight generalization of the one given by Bourgain in \cite{Bourgain:2005}.
\begin{defn} Given $\gamma>0$ and $\fp\ge 0$, we  denote by $\dgp\equiv \dgpab$ the set  of \sl{$\mu_1,\mu_2,\gamma$-Diophantine} frequencies 
	\begin{equation}\label{diofantinoBIS}
	\dgpab:=\set{\omega\in \Omega_\fp\,:\;	|\omega\cdot \ell|> \gamma \prod_{n\in \Z}\frac{1}{(1+|\ell_n|^{\mu_1} \jap{n}^{\mu_2+\fp})}\,,\quad \forall \ell\in \Z^\Z: 0<|\ell|<\infty}.
	\end{equation} 
\end{defn}
For all $\mu_1,\mu_2>1$, Diophantine frequencies are {\sl typical } in $\Omega_\fp$ and we have the following measure estimate. 
\begin{lemma}\label{misura}
	For $\mu_1,\mu_2>1$ the exists a positive constant $\Cmeas(\mu_1,\mu_2)$ such that
	\[\mu\big(\Omega_\fp\setminus \dgpab\big)
	\leq \Cmeas(\mu_1,\mu_2)\g\,.
	\]
\end{lemma}
\begin{proof}
	In Appendix \ref{cantor}
\end{proof}
Here and in the following we shall 
always assume that 
$$
\omega\in \mathtt D_{\g,\fp}^{2,2}\,.
$$
We will take
\begin{equation}\label{pippa}
0<\gamma\leq 1\,
\end{equation}
and, coherently with
\eqref{diofantinoBISintro}, we shall write for brevity
$$
\dgp=\mathtt D_{\g,\fp}^{2,2}\,.
$$

The following Lemma is the key point in the control of the small divisors appearing in the solution of the Homological equation.
	
	 \begin{lemma}\label{constance 2 gen}
	Consider $\bal,\bbt\in\N^\Z$ with
	 $1\leq|\bal|=|\bbt|<\infty$. If
	\begin{equation}\label{divisor}
	\abs{\sum_i{\pa{\bal_i-\bbt_i}i^2}}\le 10 \sum_i\abs{\bal_i-\bbt_i}\,,
	\end{equation}
	then
	for all $j$ such that $\bal_j+\bbt_j\neq 0$ one has
	\begin{equation}\label{adele}
	\sum_i\abs{\bal_i-\bbt_i}\jap{i}^{\theta/2} 
	\le 
	C_*
	 \pa{\sum_i \pa{\bal_i+\bbt_i}\jap{i}^\teta- 2\jap{j}^\teta  + \abs{\pi}}\,,
	\end{equation} 
	\begin{equation}\label{cosette4}
	\prod_i(1+\abs{\bal_i-\bbt_i}{\jap{i}}) 
	 \le e^{27}(1+\abs{\pi})^3
	N^6\prod_{l=3}^N\na_l^{\tau_0}\,.
	\end{equation}
		where
		$N=|\bal|+|\bbt|$ and
		 $\pi = \sum_i i\pa{\bal_i - \bbt_i}$ (recall \eqref{momento}.
	\end{lemma}
\begin{proof}
	In appendix \ref{costi2}
\end{proof}
	Note that
\begin{equation}\label{fiandre}
	\abs{\sum_i{\pa{\bal_i-\bbt_i}i^2}}\ge 10 \sum_i\abs{\bal_i-\bbt_i}
	\qquad
	\Longrightarrow
	\qquad
	\abs{\omega\cdot\pa{\bal-\bbt}}\ge 1\,.
\end{equation}
 Indeed denoting $\omega_j = j^2 + \xi_j \jap{j}^{-\fp}$ with $\abs{\xi_j}\le \frac{1}{2}$, \[\abs{\omega\cdot\pa{\bal-\bbt}} \ge 10\sum_j\abs{\bal_j - \bbt_j} - \frac{1}{2}\sum_j\abs{\bal_j - \bbt_j}\ge 1.\]

In the remaining part of this section, on appropriate source and target spaces, we will study the invertibility of the "Lie derivative" operator 
\begin{equation}
L_\omega: \quad H \mapsto L_\omega H: = \sum_{\ast} \im\pa{\omega \cdot(\bal-\bbt)}H_{\bal,\bbt}\buu,
\end{equation}
which is nothing but the action of the Poisson bracket $\set{\sum_j\omega_j\abs{u_j}^2, \cdot}$ on $H$. Let us start with the following lemma.\\
For any $r,\eta,\tw$ and $\bal,\bbt\in\N^\Z$ let 
\begin{equation*}
\forall j\in\Z\qquad c^{(j)}_{r,\eta,\tw}(\bal,\bbt)
	:=
	r^{|\bal|+|\bbt|-2}e^{\eta|\pi(\bal-\bbt)|}
	\frac{\tw_j^2}{\tw^{\bal+\bbt}}
	\end{equation*}
be the coefficient defined in \eqref{persico} of Definition \ref{tufo}.
\begin{lemma}[Homological equation]\label{Liederbis} 
Fix $\omega\in \dgp.$
	Consider two {\it ordered} weights $0<\rs\leq \ri,$ $ 0\leq \etaf\leq \etai,\twf\geq \twi,$  
	 such that
	\begin{equation}
	\label{melanzana}
	K:=
\gamma \sup_{\substack{j\in\Z,\, \bal\neq \bbt\in\N^\Z\\\bal_j+\bbt_j\neq 0}}
	\frac{c^{(j)}_{\rs,\etaf,\twf}(\bal,\bbt) }
	{c^{(j)}_{\ri,\etai,\twi}(\bal,\bbt) |\omega\cdot (\bal-\bbt)|}
	< \infty\,,
\end{equation}
	then for any $R\in\cR_{\ri,\etai}(\th_{\twi})$ the homological equation
	\[
	L_\omega S = R
	\]
	has a unique solution $S= L_\omega ^{-1} R$ in $\cR_{\rs,\etaf}(\th_{\twf})$, which satisfies
\begin{equation}\label{cavolfioreb}
\abs{L_\omega ^{-1} R}_{\rs,\etaf,\twf}\le \g^{-1} K\abs{R}_{\ri,\etai,\twi}\,.	
\end{equation}
Similarly, if $R$ preserves momentum, 
assuming only
\begin{equation}
	\label{melanzana0}
	K_0:=
\gamma \sup_{\substack{j\in\Z,\, \bal\neq \bbt\in\N^\Z\\ \bal_j+\bbt_j\neq 0
\\ \sum_i i (\bal_i-\bbt_i)=0}}
	\frac{c^{(j)}_{\rs,\etaf,\twf}(\bal,\bbt) }
	{c^{(j)}_{\ri,\etai,\twi}(\bal,\bbt) |\omega\cdot (\bal-\bbt)|}
	< \infty\,,
\end{equation}
we have that $S$ also preserves momentum and 
\begin{equation}\label{cavolfioreb0}
\abs{L_\omega ^{-1} R}_{\rs,\etaf,\twf}\le \g^{-1} K_0
\abs{R}_{\ri,\etai,\twi}\,.
\end{equation}
	\end{lemma}
\begin{proof}
Given any Hamiltonian $R\in\cR$, the formal solution of $L_S = R$  is given by 
\begin{equation}
 L_\omega^{-1} R 
 = 
 \sum_{ |\bal|=|\bbt|,\, \bal\neq\bbt}
 \frac{1}{\im\pa{\omega \cdot(\bal-\bbt)}}
 R_{\bal,\bbt}\buu\,,
\end{equation}
where $u\in B_{\rs}(\th_{\twf}).$
By Lemma \ref{stantuffo} (ii)
(applied to $H^{(1)}=L_\omega^{-1} R$ and
$H^{(2)}=R$)
and \eqref{melanzana}, we get
\eqref{cavolfioreb}.
The momentum preserving case is analogous. 
\end{proof}

\subsection{The homological equation}

\begin{lemma}\label{Lieder}
Let $\omega\in\dgp$
and let $0< \sigma <\eta$, $0<\rho<r/2.$ The following holds.

	$\GE)$ For any $R\in\cR_{r,\eta}(\th_\tw)$, with $\tw_j=\jap{j}^{p} e^{a|j|+s\jap{j}^\theta}$, the Homological equation $L_\omega S = R$ has a unique solution $S=L_\omega^{-1}R\in\cR_{r,\eta - \sigma}(\th_{\tw'})$, with $\tw'_j=\jap{j}^{p} e^{a|j|+(s+\sigma)\jap{j}^\theta}$, which satisfies 
	\begin{equation}\label{cavolfiore gevrey}
\abs{L_\omega ^{-1} R}_{r,\eta-\sigma,\twf}\le \g^{-1} e^{\Cuno\s^{-\frac{3}{\theta}}}\abs{R}_{\ri,\etai,\twi}	
	\end{equation}
		
		$\SO)$ For any $R\in\cR_{r,\eta}(\th_\tw)$, with $\tw_j=\jap{j}^{p} $, the Homological equation $L_S = R$ has a unique solution $S=L^{-1}R\in\cR_{r-\rho,\eta - \sigma}(\th_{\tw'})$, with $\tw'_j=\jap{j}^{p + \tau} $, which satisfies 
	\begin{equation}\label{cavolfiore sob}
\abs{L_\omega ^{-1} R}_{r-\rho,\eta-\sigma,\twf}\le \g^{-1} \Cdue(r/\rho,\s,\tau)\abs{R}_{\ri,\etai,\twi}.	
	\end{equation}

		$\MS)$ For any preserving momentum $R\in\cR_{r,0}(\th_\tw)$ , with $\tw_j=\jml{j}^{p} $, the Homological equation $L_S = R$ has a unique preserving momentum solution $S=L^{-1}R\in\cR_{r,0}(\th_{\tw'})$, with $\tw'_j=\jml{j}^{p + \tau_1} $, which satisfies 
	\begin{equation}\label{cavolfiore modisob}
\|{L_\omega ^{-1} R}\|_{r, p + \tau_1}\le \g^{-1} 6^{\tau_1} (4^6 e^{27})^{2+\fp} \|{R}\|_{r, p}.	
		\end{equation}
		
\end{lemma} 

	\comment{
\begin{lemma}\label{Liederb}
Assume that $\omega\in\dgp.$
Let $0< \sigma <\eta$, $0<\rho<r/2.$
In the following three cases:

	$\GE)$
		$
\rs= \ri\,,\ \
\etaf=\etai-\sigma\,,\  
\tw'_j=\jap{j}^{p}e^{a|j|+(s+\sigma)\jap{j}^\theta}\,,$
$\tw_j=\jap{j}^{p} e^{a|j|+s\jap{j}^\theta}\,;$ 
		
		$\SO)$
	$
\rs=\ri-\rho\,,\ 
\etaf=\etai-\sigma\,,\  
\tw'_j=\jap{j}^{p+\tau}\,,\ \
\tw_j=\jap{j}^{p}\,;
$

		$\MS)$
	$	\rs=\ri\,,\ \
\etaf=\etai=0\,,\ \
\tw'_j=\jml{j}^{p+\tau_1}\,,\ \
\tw_j=\jml{j}^{p}\,;$

	the quantities
 $K,K_0$ in \eqref{melanzana} are bounded and
 satisfy, respectively,
 
		$\GE$)
$\displaystyle  	K\leq e^{\Cuno\s^{-\frac{3}{\theta}}}\,;$

		$\SO)$
	$\displaystyle K\leq
	\Cdue(r/\rho,\s,\tau)\,;
	$

		$\MS)$
$K_0\leq 6^{\tau_1} (4^6 e^{27})^{2+\fp}$\,.
\end{lemma}
}

\begin{proof}  In the following, we will compute for each item the corresponding $K, K_0$ defined in \eqref{melanzana} and \eqref{melanzana0}, and show their finiteness in order to apply Lemma \ref{Liederbis} and give the explicit upper bounds entailed in \eqref{cavolfiore gevrey}-\eqref{cavolfiore sob}-\eqref{cavolfiore modisob}.\\
Item $\GE$) 
In this case
\begin{equation*}
K=\g\sup_{j: \bal_j+\bbt_j\neq 0}
\frac{e^{-\s\pa{\sum_i\jap{i}^\theta (\bal_i+\bbt_i) -2\jap{j}^\theta+|\pi|}}}{\abs{\omega\cdot{\pa{\bal - \bbt}}}}
\,.
\end{equation*}
	There are two cases. \\
If \eqref{divisor} does not hold, 
then by \eqref{fiandre}
 $\abs{\omega\cdot\pa{\bal-\bbt}}\ge 1$
 and by 
\eqref{yuan 2} and \eqref{pippa}
we get 
$$
\g
\frac{e^{-\s\pa{\sum_i\jap{i}^\theta (\bal_i+\bbt_i) -2\jap{j}^\theta+|\pi|}}}{\abs{\omega\cdot{\pa{\bal - \bbt}}}}
\leq 1\,
$$
and the bound is trivially achieved.\\
Otherwise, let us consider the case in which \eqref{divisor}
 holds. By applying Lemma \ref{constance 2 gen}, since $\omega\in \dgp$ we get:
\begin{eqnarray}
&&\g\frac{e^{-\s\pa{\sum_i\jap{i}^\theta (\bal_i+\bbt_i) -2\jap{j}^\theta+|\pi|}}}{\abs{\omega\cdot{\pa{\bal - \bbt}}}}  
 \nonumber
   \\
&&
\le
   e^{-\frac{\s}{C_* }\sum_i\abs{\bal_i-\bbt_i}\jap{i}^{\frac{\teta}{2}}}\prod_i\pa{1+(\bal_i-\bbt_i)^{2}\jap{i}^{{2}+\fp}}
   \nonumber
   \\
&& \le  
\exp{\sum_i\sq{-\frac{\s}{C_*} \abs{\bal_i - \bbt_i}\jap{i}^{\frac{\teta}{2}} + \ln{\pa{1 + \pa{\bal_i - \bbt_i}^{2}\jap{i}^{{2}+\fp}}}}} 
\nonumber
\\
&& 
=  \exp{\sum_i f_i(\abs{\bal_i-\bbt_i})} 
\label{spa}
\end{eqnarray}
where, for $0<\s\leq 1$, $i\in \Z$ and $x\geq 0$, we defined
$$
f_i(x) := -\frac{\s}{C_*} x\jap{i}^{\frac{\teta}{2}} + \ln{\pa{1 + x^{2}\jap{i}^{{2}+\fp}}}\,.
$$
In order to bound \eqref{spa}, we need the following lemma, whose proof is postponed  to  Appendix \ref{vaccinara}.
\begin{lemma}\label{pajata}
	Setting
	$$
	i_\sharp:=
	\left(
	\frac{8C_*(\fp+3)}{\s\theta}
	\ln
	\frac{4C_*(\fp+3)}{\s\theta}
	\right)^{\frac2\theta}
	$$
	we get 
	\begin{equation}\label{gricia}
	\sum_i f_i(|\ell_i|)\leq 
	7(\fp+3)
	i_\sharp \ln i_\sharp
	- \frac{\s}{2C_*} \big(\na_1(\ell)\big)^{\frac\theta 2}
	\end{equation}
	for every $\ell\in\Z^\Z$ with $|\ell|<\infty.$
\end{lemma}
The inequality \eqref{cavolfiore gevrey} follows from plugging \eqref{gricia}
into \eqref{spa} and evaluating the constant.

\medskip


\noindent 
Item $\SO)$
In this case $K$ in \eqref{melanzana} is 
	\begin{equation}\label{virgilio2}
	K=\g
	\sup_{j: \bal_j+\bbt_j\neq 0}\pa{1-\frac{\rho}{r}}^{N-2}\pa{\frac{\jap{j}^2}{\prod_i\jap{i}^{\bal_i+\bbt_i}}}^\tau\frac{e^{-\s|\pi|}}{\abs{\omega\cdot{\pa{\bal - \bbt}}}}\,,	\end{equation}
	where $N=|\bal|+|\bbt|$.\\
	Here two we consider two cases. \\
If \eqref{divisor} is not satisfied then\eqref{fiandre} holds
 and the right hand side of \eqref{virgilio2}
is bounded by
the quantity in \eqref{borgogna}
and it is estimated analogusly.
\\	
If
	\eqref{divisor} holds instead, by
applying formula \eqref{stimap},  Lemma \ref{constance 2 gen} and  the fact that $\omega\in \dgp$ we get:
	\begin{align*}	
	&\pa{\frac{\jap{j}^2}{\prod_i\jap{i}^{\bal_i+\bbt_i}}}^\tau\frac{1}{\abs{\omega\cdot{\pa{\bal - \bbt}}}}  \le   \pa{\frac{\jap{j}^2}{\prod_i\jap{i}^{\bal_i+\bbt_i}}}^\tau\prod_i\pa{1+|\bal_i-\bbt_i|^{2}\jap{i}^{{2}+\fp}}\\
	& \le 
	    \pa{\frac{N+|\pi|}{\prod_{l=3}^{N} \na_l}}^\tau \pa{e^{27}(1+\abs{\pi})^3 N^6\prod_{l\ge 3}\na_l^{\tau_0}}^{2+\fp} 
	\\
	& 
	\le  e^{27(2+\fp)}(N+|\pi|)^{\tau+9(2+\fp)}
	\leq   e^{27(2+\fp)}(N+|\pi|)^{3\tau}\,.
	\end{align*}
 By  using Lemma \ref{chiappa} (just like explained in detail in formula \eqref{mizzica} with $p_1=3\tau$), $K$ in \eqref{virgilio2} is bounded by
	\begin{eqnarray*}
	 && e^{27(2+\fp)} (N+|\pi|)^{3\tau} 
	 \pa{1-\frac{\rho}{r}}^{N-2}e^{-\s|\pi|} 
	 \\
&&	 \le
	  e^{27(2+\fp)}
	 2^{3\tau+1} (3\tau)^{3\tau}
	\max\left\{\pa{\frac{2 r}{\rho}}^{3\tau} ,  \pa{\frac{1}{\s}}^{3\tau}, 1\right\}
	 \end{eqnarray*}

	 \medskip

{Item $\MS)$}
Note that 
 in this case the constant in \eqref{melanzana} amounts 
 to
 $$
 K_0= \g
 \sup_{\substack{j\in\Z,\, \bal\neq \bbt\in\N^\Z\\\bal_j+\bbt_j\neq 0,
 \, \sum_i i(\bal_i-\bbt_i)=0}}
	\left(\frac{\jml{j}^2}{\prod_i \jml{i}^{\bal_i+\bbt_i}}\right)^{\tau_1}
	\frac{\gamma}
	{ |\omega\cdot (\bal-\bbt)|}
	\,,
	$$
	since (recall \eqref{persico}) we 
	have
	\begin{eqnarray*}
	c^{(j)}_{\rs,0,\twf}(\bal,\bbt)
	&=&
	r^{N-2}
	\left(\frac{\jml{j}^2}{\prod_i \jml{i}^{\bal_i+\bbt_i}}\right)^{p+{\tau_1}}
	\,,
	\\
	c^{(j)}_{\ri,0,\twi}(\bal,\bbt)
	&=&
	r^{N-2}
	\left(\frac{\jml{j}^2}{\prod_i \jml{i}^{\bal_i+\bbt_i}}\right)^p\,.
	\end{eqnarray*}
	We have two cases.
	If \eqref{fiandre} holds
	$K_0\leq \gamma$
	by \eqref{pierodellafrancesca}.
	\\
Otherwise \eqref{divisor} holds
and, therefore, \eqref{cosette4} (note that here
$\pi=0$) applies, giving
\begin{eqnarray*}
K_0
&\leq&
 \sup
	\left(\frac{\jml{j}^2}{\prod_i \jml{i}^{\bal_i+\bbt_i}}\right)^{\tau_1}
	\prod_i\pa{1+|\bal_i-\bbt_i|^{2}\jap{i}^{{2}+\fp}}
\\
&\leq&
\sup
	\left(\frac{\jml{j}^2}{\prod_i \jml{i}^{\bal_i+\bbt_i}}\right)^{\tau_1}
e^{27(2+\fp)} N^{6(2+\fp)}\prod_{l= 3}^N\na_l^{\tau_0(2+\fp)}
\end{eqnarray*}
since $\omega\in \dgp$.
We claim that
\begin{equation}\label{filomena}
N\leq 4 
\prod_{l= 3}^N\jml{\na_l}^{\frac{1}{4\ln 2}}\,.
\end{equation}
Indeed if $N=2$, the inequality is trivial.
Since $N$ is even we have to consider only the case
$N\geq 4$, which follows by 
 Lemma \ref{chiappa}.
 Recalling \eqref{fiorentina} we have
\begin{equation}\label{fiorentina2}
\prod_i\jml{i}^{\bal_i+\bbt_i}= \prod_{l\ge 1}\jml{\na_l}\,.
\end{equation}
Then
 \begin{equation*}
	\sup_{\substack {j,\bal,\bbt\\\bal_j+\bbt_j\geq 1 }} 
	\frac{\jml{j}^2}{\prod_i\jml{i}^{\bal_i+\bbt_i}}
	 \leq 
	 \frac{\jml{\hat n_1}^2}{\prod_{l\ge 1}\jml{\na_l}}
	 =
	 \frac{\jml{\hat n_1}}{\prod_{l\ge 2}\jml{\na_l}}
	 \leq
	\frac{\sum_{l\ge 2}\jml{\na_l}}{\prod_{l\ge 2}\jml{\na_l}} 
	=
	\frac{1}{\prod_{l\ge 3}\jml{\na_l}}
	+
	\frac{\sum_{l\ge 3}\jml{\na_l}}{\prod_{l\ge 2}\jml{\na_l}} 
	 \,,
\end{equation*}
where the last inequality holds
by momentum conservation.
Then\footnote{Using that
$(a+b)^{\tau_1}\leq 2^{{\tau_1}-1}(a^{\tau_1}+b^{\tau_1})$
for $a,b\geq 0,$ ${\tau_1}\geq 1.$}
\begin{eqnarray*}
K_0
&\leq&
2^{{\tau_1}-1}
\left(
\frac{1}{\prod_{l\ge 3}\jml{\na_l}^{\tau_1}}
	+
	\frac{(\sum_{l\ge 3}\jml{\na_l})^{\tau_1}}{\prod_{l\ge 2}\jml{\na_l}^{\tau_1}} 
\right)
(4^6 e^{27})^{2+\fp}\prod_{l\ge 3}
\jml{\na_l}^{{\tau_1}/2}
\\
&\leq&
2^{{\tau_1}-1}(4^6 e^{27})^{2+\fp}
\left(
1
	+
	\frac{(\sum_{l\ge 3}\jml{\na_l})^{\tau_1}}{
	\jml{\na_2}^{\tau_1}\prod_{l\ge 3}\jml{\na_l}^{{\tau_1}/2}} 
\right)
\\
&\leq&
2^{{\tau_1}-1}(4^6 e^{27})^{2+\fp}
\left(
1
	+
	\frac{(\jml{\na_3}^{1/2}+4)^{\tau_1}}{
	\jml{\na_2}^{\tau_1}} 
\right)
\end{eqnarray*}
by Lemma \ref{brabante} with $a=1/2.$
The estimate on $K_0$, hence inequality $\eqref{cavolfiore modisob}$ follows.		
\end{proof}

\section{Abstract Birkhoff Normal Form}\label{birk}
In this section we prove an abstract Birkoff normal form theorem; its setting is flexible and easy to adapt in the three cases of our interest $\SO, \MS, \GE$. The normal form will be proved iteratively by means of the following Lemma, which constitutes the main step of the procedure.
\begin{lemma}\label{dolcenera}
Fix $\omega\in \dgp.$
Let $\ri>\rf>0,\etai\geq\etaf\geq 0,$
$\twi\leq\twf.$
Consider
$$
	H = D_\omega + Z + R  \,, 
\quad  Z \in \cK^\wc_{\ri,\etai}(\th_{\twi})	\,,  
	\quad   R  \in  \cR^\wc_{\ri,\etai}(\th_{\twi})\,,
	\quad
	\td(Z)\geq 1\,, \ \ \td(R)\geq \td\geq 1\,.
$$
	Assume that
\eqref{zucchina} and \eqref{melanzana}
hold and 
that\footnote{$K$ is the constant in 
\eqref{melanzana}.}
\begin{equation}\label{viadelcampo}
|R|_{\ri,\etai,\twi}
\leq
\frac{\g\delta}{K}\,,
\qquad
\text{with}
\quad
\delta:=\frac{\ri-\rf}{16e\ri}
\,.
\end{equation}
Then there exists a change of variables
\begin{eqnarray}
	\label{pollon3}
	& \Phi\ :\ B_\rf(\th_{\twf})\ \to\ 
B_{\ri}(\th_{\twf})\,,
	\end{eqnarray}
such that
$$
	H\circ\Phi = D_\omega + Z' + R'  \,, 
\quad  Z' \in \cK_{\rf,\etaf}(\th_{\twf})	\,,  
	\quad   R'  \in  \cR_{\rf,\etaf}(\th_{\twf})\,,
	\quad
	\td(Z')\geq 1\,, \ \ \td(R')\geq \td+1\,.
$$
	Moreover\footnote{$C$ is defined in \eqref{zucchina}.}
\begin{eqnarray}
|Z'|_{\rf,\etaf,\twf}
&\leq&
|Z|_{\ri,\etai,\twi} + 
 (\g\delta)^{-1} K |R|_{\ri,\etai,\twi} (C|R|_{\ri,\etai,\twi}+ |Z|_{\ri,\etai,\twi})\,,
\nonumber
\\
|R'|_{\rf,\etaf,\twf}
&\leq&
 (\g\delta)^{-1} K |R|_{\ri,\etai,\twi} (C|R|_{\ri,\etai,\twi}+ |Z|_{\ri,\etai,\twi})\,.
\label{signorina}
\end{eqnarray}
Finally, for $\tw^\sharp\geq \twf,$ 
assume the further  conditions 
\begin{equation}
	\label{melanzana'}
\gamma \sup_{\substack{j\in\Z,\, \bal\neq \bbt\in\N^\Z\\
\bal_j+\bbt_j\neq 0}}
	\frac{c^{(j)}_{\rs,\etaf,\tw^\sharp}(\bal,\bbt) }
	{c^{(j)}_{\ri,\etai,\twi}(\bal,\bbt) |\omega\cdot (\bal-\bbt)|}=: K^\sharp
	< \infty\,,
	\qquad
	\rs:=\frac{\rf+\ri}{2}
\end{equation}
and
\begin{equation}\label{viadelcampo'}
|R|_{\ri,\etai,\twi}
\leq
\frac{\g\delta}{K^\sharp}\,.
\end{equation}
Then
\begin{eqnarray}
	\label{pollon4}
	& \Phi_{\big|B_\rf(\th_{\tw^\sharp})}\ :\ B_\rf(\th_{\tw^\sharp})\ \to\ 
B_{\ri}(\th_{\tw^\sharp})\,,
\nonumber
\\
&\sup_{u\in  B_\rf(\th_{\tw^\sharp})} 	\norm{\Phi(u)-u}_{\th_{\tw^\sharp}}
\le
   \ri \g^{-1} K^\sharp\abs{R}_{\ri,\etai,\twi}\,.
	\end{eqnarray}
	Moreover if $R$ preserves momentum, assuming only that
	\begin{equation}
	\label{melanzana'0}
	K^\sharp_0:=
\gamma \sup_{\substack{j\in\Z,\, \bal\neq \bbt\in\N^\Z\\
\bal_j+\bbt_j\neq 0,
\\ \sum_i i(\bal_i -\bbt_i)=0}}
	\frac{c^{(j)}_{\rs,\etaf,\tw^\sharp}(\bal,\bbt) }
	{c^{(j)}_{\ri,\etai,\twi}(\bal,\bbt) |\omega\cdot (\bal-\bbt)|}
	<\infty
\end{equation} 
and that 
\eqref{viadelcampo}, \eqref{viadelcampo'} hold
with $K_0,K^\sharp_0$ instead of $K,K^\sharp$
we have
that $R'$ preserves momentum and
  \eqref{pollon4} holds with 
  $K^\sharp_0$ instead of $K^\sharp.$
\end{lemma}
\begin{proof}
 By Lemma \ref{Liederbis} 
 let $S= L_\omega ^{-1} R$ in $\cR_{\rs,\etaf}(\th_{\twf})$ be the 
 unique solution of the homological equation $L_\omega S = R$ on   $ B_{\rs}(\th_{\twf})$.
 Note that $\td(S)\geq \td$.
We have
\begin{equation}\label{cavolfiore}
\abs{S}_{\rs,\etaf,\twf}\le \g^{-1} K\abs{R}_{\ri,\etai,\twi}\,.	
\end{equation} 
We now apply Lemma \ref{ham flow} with
$(r,\eta,\tw)\rightsquigarrow(\rf,\etaf,\twf)$ and $\rho:=\rs-\rf.$
Note that \eqref{viadelcampo} and \eqref{cavolfiore}
imply \eqref{stima generatrice}.
We define $\Phi:=\Phi_S^1$
and compute
	\begin{align*}
	H'&:=H\circ\Phi= D_\omega + Z +  (e^{\set{S,\cdot}}-\id -\{S,\cdot\}) D_\omega  + (e^{\set{S,\cdot}}-\id)(Z+R)=
	\\
	&= D_\omega + Z - \sum_{j=2}^\infty\frac{\pa{{\rm ad} S}^{j-1}}{j!} R  + (e^{\set{S,\cdot}}-\id)(Z+R)\,.
	\end{align*}
	We now set
	\[
	Z'= \Pi_{\cK} H' -D_\omega \,,\quad R' = \Pi_{\cR} H'\,.
	\]
	Since  the scaling degree is additive w.r.t. Poisson brackets, we have that $\td(Z')\geq 1$ 
	and $\td(R')\geq \td+1$.
	By \eqref{brubeck}
\begin{eqnarray*}
|Z'|_{\rf,\etaf,\twf}
&\leq&
|Z|_{\rf,\etaf,\twf} + 
 (\g\delta)^{-1} K |R|_{\ri,\etai,\twi} (|R|_{\rs,\etaf,\twf}+ |Z|_{\rs,\etaf,\twf})\,,
\nonumber
\\
|R'|_{\rf,\etaf,\twf}
&\leq&
 (\g\delta)^{-1} K |R|_{\ri,\etai,\twi} (|R|_{\rs,\etaf,\twf}+ |Z|_{\rs,\etaf,\twf})\,.
\end{eqnarray*}
Since \eqref{melanzana} holds we 
can apply Proposition \ref{cacioricotta}:
by \eqref{cetriolo}
and  \eqref{cetriolobis} we get
\begin{equation*}
|R|_{\rs,\etaf,\twf}
	\le 
	C
	|R|_{\ri,\etai,\twi}\,,
	\qquad
	|Z|_{\rs,\etaf,\twf}
	\le 
	|Z|_{\ri,\etai,\twi}\,.
\end{equation*}
\eqref{signorina} follows.
\\
Finally assume \eqref{viadelcampo'} and
\eqref{melanzana'}.
 By Lemma \ref{Liederbis} 
 let $S^\sharp= L_\omega ^{-1} R$ in $\cR_{\rs,\etaf}(\th_{\tw^\sharp})$ be the 
  solution of the homological equation 
 $L_\omega S^\sharp = R$ on 
  $B_{\rs}(\th_{\tw^\sharp})\subseteq B_{\rs}(\th_{\twf})$.
 Since $S$ and $S^\sharp$ solve the same linear
 equation on 
 $B_{\rs}(\th_{\tw^\sharp})$, we have that
 $$
 S^\sharp=S_{\big| B_{\rs}(\th_{\tw^\sharp})}\,.
 $$ 
By \eqref{cavolfioreb} we get
\begin{equation}\label{cavolfiore'}
\abs{S}_{\rs,\etaf,\tw^\sharp}
\leq 
\g^{-1} K^\sharp\abs{R}_{\ri,\etai,\twi}\,.	
\end{equation} 
We now apply Lemma \ref{ham flow} with
$(r,\eta,\tw)\rightsquigarrow(\rf,\etaf,\tw^\sharp)$ 
and $\rho:=\rs-\rf.$
Note that \eqref{viadelcampo'} and \eqref{cavolfiore'}
imply \eqref{stima generatrice}.
Then
\eqref{pollon4}
 follows by \eqref{pollon} and \eqref{cavolfiore'}.
 
 The momentum preserving case is analogous.
\end{proof}

We are now ready to state and prove the abstract Birkhoff normal form theorem, from which Theorem \ref{sob} will follow, as a particular case in the Gevrey, Sobolev Modified Sobolev settings.\smallskip\\

Fix any natural $\suca>1$, $\eta\geq 0$ and 
sequence of weights 
$\mathtt w_0\leq \mathtt w_1\leq\cdots
\leq \mathtt w_\suca$. 
For any given   $ r>0$ 
we set 
\begin{equation}\label{gamberetto1}
r_n =(2  - \frac{n}{\suca}) r \,,\quad  
\eta_n = (1  - \frac{n}{\suca})\eta \,,\; 
\quad 0\leq n\leq\suca\,,
\quad \quad  r_n^* =  \frac{r_{n+1}+ r_n}{2}\,,\quad  0\leq n<\suca\,. 
\end{equation} 
For brevity we set
\begin{equation}\label{gamberetto2}
\th_n:=\th_{\mathtt w_n}\,,\quad
\cH_n:= \cH_{r_n,\eta_n}(\th_n)\,,\; \quad 
0\leq n\leq\suca\,,\quad  
\cH_{n,*}:= \cH_{r^*_n,\eta_{n+1}}(\th_{n+1})\,,\quad 0\leq n<\suca\,,
\end{equation}
and, correspondingly, $\cR_n,\cK_n,  {\cR}_{n,*},  {\cK}_{n,*}$ and 
\begin{equation}\label{gamberetto3}
|\cdot|_n:=|\cdot|_{r_n,\eta_n,\tw_n}\,,\qquad
|\cdot|_{n,*}:=|\cdot|_{r^*_n,\eta_{n+1},\tw_{n+1}}\,.
\end{equation}

\begin{hyp}\label{hip0}
	Assume that 
	\begin{eqnarray}\label{dito1}
	&&\widehat C
	:=
	\max\left\{1,\ \sup_{0\leq n<\suca}
	\sup_{\substack {j,\bal,\bbt\\\bal_j+\bbt_j\neq 0 }} \frac{c^{(j)}_{r^*_{n},\eta_{n+1},\tw_{n+1}}(\bal,\bbt) }
	{c^{(j)}_{r_{n},\eta_{n},\tw_n}(\bal,\bbt) } \right\}
	<\infty\,,\\
	\label{dito2}
&&	
\widehat K:=
	\max\left\{1,\ \sup_{0\leq n<\suca}
	\sup_{\substack {j,\bal,\bbt\\\bal_j+\bbt_j\neq 0 }} \frac{
		c^{(j)}_{r^*_{n},\eta_{n+1},\tw_{n+1}}(\bal,\bbt)
	}
	{c^{(j)}_{r_{n},\eta_{n},\tw_n}(\bal,\bbt)
		|\omega\cdot (\bal-\bbt)|}\right\}
		 <\infty\,,
	\\
	\label{dito3}
	&&
	\widehat K^\sharp
	:=
		\max\left\{1,\ \sup_{0\leq n<\suca}
	\sup_{\substack {j,\bal,\bbt\\\bal_j+\bbt_j\neq 0 }} \frac{
		c^{(j)}_{r^*_{n},\eta_{\suca},\tw_\suca}(\bal,\bbt)
	}
	{c^{(j)}_{r_{n},\eta_{n},\tw_n}(\bal,\bbt)
		|\omega\cdot (\bal-\bbt)|} \right\}
		<\infty\,.
	\end{eqnarray}
	In the case of momentum preserving hamiltonians
	we define $\widehat \CM ,\widehat K_0,\widehat K^\sharp_0$
	as in \eqref{dito1}-\eqref{dito3} with the further condition
	$\sum_i i(\bal_i - \bbt_i)=0;$ and  
	we only assume that such constants are bounded.
	\end{hyp}
	
\begin{rmk}\label{lentezza}
 Recalling \eqref{persico} we note that the constants 
 $\widehat C,\widehat K,\widehat K^\sharp$
 (as well as $\widehat \CM ,\widehat K_0,\widehat K^\sharp_0$)
 do not depend on $r$. They only depend on $\tw_{n,j}/\tw_{0,j}.$
\end{rmk}

\begin{lemma}\label{hip}
By Assumption \eqref{dito1} we have the monotonicity properties 
\begin{equation}\label{orata}
	\cH_0\subseteq \cH_{0,*} \subseteq \cdots \subseteq \cH_n\subseteq \cH_{n,*} \subseteq \cH_{n+1} \subseteq\cdots\subseteq \cH_\suca\,,
\end{equation}
	with  estimates 
	\begin{eqnarray}
&&H\in \cH_n\qquad\Longrightarrow\qquad |H|_{n,*}\le \widehat  C |H|_n\,, \qquad
\quad 0\leq n\leq i\leq\suca-1
\nonumber
\\
\label{spigola}
& &	H\in \cK_n\qquad\Longrightarrow\qquad
	 	|H|_{n,*}  \le  |H|_n\,,\qquad
		\quad 0\leq n\leq i\leq\suca-1\,.
\end{eqnarray}
\end{lemma}
\begin{proof}
	We apply Proposition \ref{cacioricotta} with
	 \[
	r,\eta,\tw \rightsquigarrow r_n,\eta_{n},\tw_{n} \,,\quad 	r^*,\eta',\tw' \rightsquigarrow r^*_n,\eta_{n+1},\tw_{n+1} \,,
	\]
by 	noting that the bound \eqref{zucchina} follows from  \eqref{dito1}. The  bounds in \eqref{spigola} follow  form \eqref{cetriolo} and \eqref{cetriolobis}. The chain of inclusions \eqref{orata} follows.
\end{proof}

\bigskip

\begin{thm}[Abstract Birkhoff Normal Form]\label{weisserose}
	Consider a Hamiltonian of the form
	\[
	H= D_\omega + G\,,\quad D_\omega= \sum_j \omega_j |u|_j^2
	\]
	with $\omega\in \dgp$,  $G\in \cH_{\bar r,\eta}(\th_{\tw_0}),$
	$\bar r>0,$ $\eta\geq 0$
	 and such that 
	 $\td(G)\ge 1$.
	  Set 
	  \begin{equation}\label{borgia}
r_\star:=\bar r
			 \sqrt{\frac{\g  }
			{2^{11}e |G|_{\bar r,\eta,\tw_0} }}\,.
\end{equation}
	For any ${\suca}\in \N_+$, under Assumption \ref{hip0}, set
	\begin{equation}\label{mmm}
			\widehat{\tr}= \widehat{\tr}({\suca}) :=
			\min\left\{  
			\frac{r_\star}{\sqrt{\suca\max\{\widehat C\widehat K,\widehat K^\sharp\}} }\,, \ \frac{\bar r}{2} \right\}\,.
\end{equation}
	Then for all $0<r\le \widehat{\tr}$ there exists a symplectic change of variables 
	\begin{equation}\label{california}
\Psi  :\quad B_{r}(\th_{\tw_{\suca}})\mapsto B_{2r}(\th_{\tw_{\suca}})\,, 
	\qquad
	\sup_{u\in  B_r(\th_{\tw_\suca})} 	\norm{\Psi(u)-u}_{\tw_\suca}
\leq
\widehat{\mathtt C}_1 r^3\,,\qquad
\widehat{\mathtt C}_1:=\frac{\widehat K^\sharp}{2^7 e r_\star^2}\,,
	\end{equation}
	 such that in the new coordinates 
	\begin{eqnarray} \label{dreaming}
&&	\	H\circ \Psi= D_\omega + Z+ R\,,\quad
		Z\in \cK_{r,0,\tw_{\suca}},\ R \in  \cR_{r,0,\tw_{\suca}}\,,\quad
		 \td(Z)\ge 1\,,\quad \td(R) \ge {\suca}\,,
		 \nonumber
\\	
&&	\  |Z|_{r,0,\tw_{\suca}} \le
\widehat{\mathtt C}_2 r^2\,,
\qquad
 \widehat{\mathtt C}_2
:=
\frac{\g}{2^8 er_\star^2}\,,
 \nonumber
\\	
&& 	\
	  |R|_{r,0,\tw_{\suca}} \le 
	  \widehat{\mathtt C}_3  r^{2(\suca+1)}\,,\qquad
	  \widehat{\mathtt C}_3:=
	  	\frac{\g}{2^9 e r_\star^2}
	\pa{\frac{ \widehat C\widehat K \suca}{4 r_\star^2}}^{\suca}
	    \,.
\end{eqnarray}
In the case that $G$ preserves momentum, 
the same result holds with $\widehat \CM ,\widehat K_0,\widehat K^\sharp_0 $
instead of $\widehat C,\widehat K,\widehat K^\sharp $;
moreover also $R$
preserves momentum.
\end{thm}
\begin{proof}
We will prove the thesis inductively. Let us start by noticing that
$$
\widehat{\tr}= \min\left\{ \frac{\bar r}{8\sqrt{|G|_{\bar r,\eta,\tw_0}}} \sqrt{\frac{\g\hat\delta  }
			{\max\{\widehat C\widehat K,\widehat K^\sharp\}  }}
			\,,\ \frac{\bar r}{2}\right\}
			\,,\qquad
			\hat\delta:=\frac{1}{32 e\suca}\,
$$
and, for all $0<r\le \widehat{\tr}$, let us set
	$$
	\e:=\g^{-1}\pa{\frac{2  r}{\bar r}}^2 |G|_{\bar r,\eta,\tw_0}
	=
	\frac{1}{2^9 e}\pa{\frac{  r}{r_\star}}^2\,.
	$$
	From definition \eqref{mmm} we thus deduce that
\begin{equation}
\label{minni}
8\, \e  \max\{\widehat C\widehat K,\widehat K^\sharp\}\hat\delta^{-1} \leq 1.
\end{equation}
Recalling the notations introduced in
\eqref{gamberetto1}-\eqref{gamberetto3},
	by Lemma \eqref{gasteropode} we have
	\[
	\g^{-1}|G|_0 
	\le  
	\e\,,
	\]
hence, setting $ Z^{(0)}:=\Pi_{\cK} G$ and $R^{(0)}:=\Pi_{\cR} G,$ from \eqref{fame} it follows that
	$$
	\g^{-1}|Z^{(0)}|_0\,,\ \g^{-1}|R^{(0)}|_0\leq \e\,.
	$$
	We perform an iterative procedure producing a sequence of Hamiltonians, for $n= 0,\dots,{\suca}$
	\begin{eqnarray}\label{sticazzi}
	&&H^{(n)}= D_\omega + Z^{(n)}+ R^{(n)} \,,
	\nonumber
	\\ 
	&&  Z^{(n)}\in \cK_{n}\,,
	\ \    R^{(n)} \in  \cR_{n}  \,,
	\quad
	\td(Z^{(n)})\geq 1\,,\ \ 
	\td(R^{(n)})\geq n+1	\,,
	\nonumber
	\\
&&\g^{-1}|Z^{(n)}|_{n} \le  \e \sum_{h=0}^n 2^{-h} \,,\quad 
	\g^{-1}|R^{(n)}|_{n}  \le  \e^{n+1} \pa{4 \widehat C \widehat K \hat\delta^{-1}}^n
	\stackrel{\eqref{minni}}\leq 
	2^{-n}\e\,.
\end{eqnarray}	
Fix any $k < \suca$. Let us assume that we have constructed $H^{(0)},\ldots,H^{(k)}$
satisfying \eqref{sticazzi} for all $0\leq n\leq k.$
We want to apply Lemma \ref{dolcenera}
with
$$
H, \ri,\etai,\twi \ \rightsquigarrow
\ H^{(k)}, r_k, \eta_k, \tw_k
\quad
\text{and}\quad
\rf,\etaf,\twf,\tw^\sharp,\td
\ \rightsquigarrow\
r_{k+1},\eta_{k+1},\tw_{k+1},\tw_\suca,k+1\,.
$$
By construction the bounds \eqref{zucchina}, \eqref{melanzana} and \eqref{melanzana'}
 hold since 
$C\leq \widehat C,$ $K\leq \widehat K$, $K^\sharp\leq \widehat K^\sharp$, where $\widehat{C}, \widehat{K}, \widehat{K}^\sharp$ were 
defined in \eqref{dito1},\eqref{dito2},\eqref{dito3}. We just have to verify that
 \eqref{viadelcampo} holds, namely
 $$
 |R^{(k)}|_k\leq \frac{\g}{\widehat K}\frac{r_k-r_{k+1}}{16 e  r_k}\,.
 $$
In fact, by applying the inductive hypothesis \eqref{sticazzi} and the smallness condition \eqref{minni}, we get
$$
|R^{(k)}|_k\leq  \g \pa{4 \widehat C \widehat K \hat\delta^{-1}}^k \e ^{k+1} \leq
\frac{\g\e}{2^k}\le 
 \frac{\g}{16 e \widehat K (2\suca- k)}
 =
\frac{\g}{\widehat K}\frac{r_k-r_{k+1}}{16 e  r_k}\,.
$$
The verification of  \eqref{viadelcampo'} is completely analogous.\\
So, by applying Lemma \ref{dolcenera} we construct a change of variable $\Phi_k$ as in \eqref{pollon3} with
$$
 \Phi_k\ :\ B_{r_{k+1}}(\th_{\tw_{k+1}})\ \to\ 
B_{r_k}(\th_{\tw_{k+1}})\,.
$$
Let us now set  
$$
H^{(k+1)}=D_\omega+Z^{(k+1)}+R^{(k+1)}:= H_k\circ\Phi_k  
$$  
with   $Z^{(k+1)}\in \cK_{k+1}, R^{(k+1)}\in \cR_{k+1}$
and
$\td(Z^{(k+1)})\geq 1,$ $\td(R^{(k+1)})\geq k+2.$
 It remains to prove the bounds 
in the second line of \eqref{sticazzi} (with $n=k+1$). 
By \eqref{signorina} we have 

\begin{eqnarray}
|Z^{(k+1)}|_{k+1}
&\leq&
|Z^{(k)}|_{k} + 
 (\g\hat\delta)^{-1} \widehat K |R^{(k)}|_{k} (\widehat C|R^{(k)}|_{k}+ |Z^{(k)}|_{k})\,,
\nonumber
\\
|R^{(k+1)}|_{k+1}
&\leq&
  (\g\hat \delta)^{-1} \widehat K |R^{(k)}|_{k} (\widehat C|R^{(k)}|_{k}+ |Z^{(k)}|_{k})\,.
\label{signorinab}
\end{eqnarray}
By substituting the inductive hypothesis  \eqref{sticazzi}, we have the following chain of inequalities
	\begin{eqnarray*}
	\g^{-1}|R^{(k+1)}|_{k+1} 
	 & \le&
	 \hat\delta^{-1} \e^2 \widehat K(4\widehat C \widehat K \hat\delta^{-1}\e )^k (\widehat C(4 \widehat C \widehat K \hat\delta^{-1}\e )^k + 2)\\ 
	 & 
	\stackrel{\eqref{minni}}\le &
	 \hat\delta^{-1} \e^2 \widehat K(4\widehat C \widehat K \hat\delta^{-1}\e )^k (\widehat C+ 2)
	\\ & \le & (4\widehat C \widehat K \hat\delta^{-1} )^{k+1}   \e^{k+2} = (4\widehat C \widehat K \hat\delta^{-1}\eps )^{k+1}\eps ,
	\end{eqnarray*}
which proves the bound on $R^{(n)}$ in \eqref{sticazzi} for any $n$.
\\
En passant,	we note that
	\begin{equation}\label{sweet dream R}
	\g\e \pa{4 \widehat C\widehat K \hat\delta^{-1}\e}^{\suca}
	=  
	\frac{\g}{2^9 e r_\star^2}
	\pa{\frac{ \widehat C\widehat K \suca}{4 r_\star^2}}^{\suca}
	r^{2(\suca+1)}
\,.
\end{equation}
	Finally, using the same strategy as above, we also get
	$$
	\g^{-1}|Z^{(k+1)}|_{k+1 } \le  \e \pa{\sum_{h=0}^k 2^{-h}  + 
	 (4\widehat C \widehat K \hat\delta^{-1} )^{k+1}    \e^{k+1}} 
	 \stackrel{\eqref{minni}}\le 
	 \e \sum_{h=0}^{k+1} 2^{-h}\,,
	$$
	which completes the proof of the inductive hypothesis \eqref{sticazzi},
and remark that
	\begin{equation}\label{sweet dream Z}
	\eps \sum_{h=0}^{\suca} 2^{-h} = \frac{r^2}{ 2^8 e r_\star^2} \pa{1 - 2^{-\suca - 1}}.
	\end{equation}

By \eqref{pollon4} we have
\begin{eqnarray}
	\label{daitarn3}
	& \Phi_k\ :\ B_{r_{k+1}}(\th_{\tw_\suca})\ \to\ 
B_{r_k}(\th_{\tw_\suca})\,,
\nonumber
\\
&\sup_{u\in  B_{r_{k+1}}(\th_{\tw_\suca})} 	\norm{\Phi_k(u)-u}_{\tw_\suca}
\le
   r_k \g^{-1} \widehat K^\sharp |R^{(k)}|_k\,.
	\end{eqnarray}
	In conclusion we define
	\[
	\Psi:= \Phi_0\circ \Phi_1\circ\dots\circ \Phi_{\suca-1}\ :\ 
	B_{ r}( \th_{\suca}) \to B_{2r}(\th_{\suca}).
	\]
Since we have
\begin{eqnarray*}
&&\Phi_0\circ \Phi_1\circ\dots\circ \Phi_{\suca-1}-\id
\\
&&=
(\Phi_0-\id)\circ \Phi_1\circ\dots\circ \Phi_{\suca-1}+
(\Phi_1-\id)\circ \Phi_2\circ\dots\circ \Phi_{\suca-1}+
\ldots
\Phi_{\suca-1}-\id\,.
\end{eqnarray*}
By \eqref{daitarn3} we get
\begin{eqnarray*}
\sup_{u\in  B_r(\th_{\tw_\suca})} 	\norm{\Psi(u)-u}_{\tw_\suca}
\leq
\sum_{k=0}^{\suca-1}
r_k \g^{-1} \widehat K^\sharp |R^{(k)}|_k
\stackrel{\eqref{sticazzi}}\leq
2r \e \widehat K^\sharp \sum_{k=0}^{\suca-1} 2^{-k}
\leq 4r  \widehat K^\sharp\e\,,
\end{eqnarray*}
proving \eqref{california}.
We finally set
	 $Z= Z_{\suca}, R= R_{\suca}$ and the  estimates \eqref{dreaming} follow by \eqref{sweet dream R}-\eqref{sweet dream Z}.
\end{proof}

\section{Proof of Theorem \ref{sob}}\label{provasob}

Theorem \ref{sob} is a particular case of the 
general Theorem \ref{weisserose}
in the usual three cases $\SO,\MS,\GE.$ 
We have to check Assumption
\ref{hip0}.

\begin{lemma}\label{bonhoeffer}
Consider the constants introduced in Assumption
\ref{hip0}.
We have the usual three cases\footnote{
Recall Remark \ref{lentezza}.}.

$\SO$) 
When
$$
\tw_{n,j} = \tw_{0,j} \jap{j}^{ n \tau}\,,\qquad
\tau:=\tau_0(2+\fp)
$$ 
we have 
\begin{eqnarray*}
&&\widehat C 
\le
 \Cmon(4\suca,\eta/\suca,\tau)
 \,,\quad 
 \widehat K \le  
 \Cdue(4\suca,\eta/\suca,\tau)\,,
\\ 
&& \widehat K^\sharp \le  
 \Cdue(4\suca,\eta/\suca,\suca\tau)\,.
\end{eqnarray*}

 $\MS$)  When   
 $$
  \tw_{n,j} = \tw
  _{0,j} \jml{j}^{ n \tau_1}
  $$ 
and assuming the momentum conservation,  we have  
\[
\widehat \CM  = 1\,,\quad \widehat K_0,\, \widehat K^\sharp_0 \le 
6^{\tau_1} (4^6 e^{27})^{2+\fp}
\,.
\]

$\GE$) When 
 $$
 \tw_{n,j} = \tw_{0,j}  e^{\frac{n\eta}{{\suca}} \jap{j}^{\theta}}
 $$ 
 we have  
\[
\widehat C=1\,,\quad  
\widehat K,\, \widehat K^\sharp
\leq
e^{\Cuno\pa{\frac{\suca}{\eta}}^{\frac{3}{\theta}}}\,.
\]

\end{lemma}
\begin{proof}

$\SO$) The computation of $\widehat C$ follows from \eqref{mulo};
 the ones of $ \widehat K,\widehat K^\sharp$ again from Lemma \ref{Lieder}.

$\MS$) The computation of $\widehat \CM $ follows from \eqref{vacca};
 the ones of $ \widehat K_0,\widehat K^\sharp_0$ again from Lemma \ref{Lieder}. 
 
  $\GE$) The computation of $\widehat C$ follows from \eqref{emiliaparanoica};
 the ones of $ \widehat K,\widehat K^\sharp$ from Lemma \ref{Lieder}.
\end{proof}

 We use Theorem \ref{weisserose}
 with
 \begin{equation}\label{salisano}
 G \rightsquigarrow P\,,\ \ 
 \tw_\suca \rightsquigarrow \tw\,.
\end{equation}
We distinguish the usual three cases.

 $\SO$) Set
 \begin{equation}\label{casaprota}
 \tw_{n,j}:=\jap{j}^{1+\tau n}\,,
  \quad 
\eta:=\ta/2\,, \quad  
 \bar r:=\frac{\sqrt R}{\Calgo}\,.
 \end{equation}
  Then Assumption
\ref{hip0}   is satisfied by Lemma \ref{bonhoeffer}.
 We have that
 \begin{equation}\label{poggiomirteto}
 |P|_{\bar r,\eta,\tw_0}=| P|_{\bar r,1,0,0,\ta/2}
 \stackrel{\eqref{stimazero}}\leq
 \CNem(1,0,\ta/2)|f|_{\ta,R}
\end{equation}
Noting that
\begin{eqnarray}
		\Cmon(4\suca,\eta/\suca,\tau)
		&=&
		 2^{2\tau+1} \tau^{\tau}  \suca^\tau
		\max\left\{4 ,  (1/2\eta)\right\}^\tau = 2 (4\tau \max\left\{4 ,  (1/2\eta)\right\}\suca)^\tau \,,
				\nonumber
		\\
		 \Cdue(4\suca,\eta/\suca,\tau)
&=&
		 2 e^{27(2+\fp)}
	 (12\tau \max\left\{ 4, (1/2\eta)\right\}\suca)^{3\tau} 
	\nonumber
		\\
		 \Cdue(4\suca,\eta/\suca,\suca\tau)
&=&
		 2 e^{27(2+\fp)}
	 (6\suca\tau)^{3\suca\tau}
	\max\left\{\pa{8\suca}^{3\suca\tau} ,  (\suca/\eta)^{3\suca\tau}\right\}
	\nonumber
		\\
	&=& 2 e^{27(2+\fp)}(12\tau \max\left\{4 , (2\eta)^{-1}\right\} \suca^2)^{3\suca\tau}\,,
	\label{fara}
\end{eqnarray}
we have that
for $\suca\geq 3$ 
\begin{equation}\label{lucrezia}
\Cmon(4\suca,\eta/\suca,\tau) 
 \Cdue(4\suca,\eta/\suca,\tau)
 \leq 
  \sqrt{\Cdue(4\suca,\eta/\suca,\suca\tau)}\,.
\end{equation}
By \eqref{borgia}
\begin{eqnarray*}
r_\star
&\geq&
 \frac{\sqrt{\g R}}{\Calgo\sqrt{2^{11}e \CNem(1,0,\ta/2)|f|_{\ta,R}}}
 \geq
 \dede\,.
\end{eqnarray*}
Then, recalling \eqref{mmm} and \eqref{lucrezia}, 
 \begin{eqnarray}  \label{poggiomoiano0}
\widehat{\tr}
&\geq& 
\mathtt r(\SO)
\,.
\end{eqnarray}
Moreover
$$
\widehat{\mathtt C}_1
\leq
\mathtt C_1(\SO)\,,
\qquad
 \widehat{\mathtt C}_2
\leq 
\mathtt C_2(\SO)
$$
and, recalling  \eqref{dreaming},
$$
\widehat{\mathtt C}_3
	\leq
	\mathtt C_3(\SO)
	\,.
$$
Finally
$$
\mathtt C_1(\SO) (\mathtt r(\SO))^2\leq \frac18\,,
$$
proving the last inequality in \eqref{stracchinobis}.

 $\MS$) Set
 \begin{equation}\label{casaprotaMS}
 \tw_{n,j}:=\jml{j}^{1+\tau_1 n}\,,
  \quad 
\eta:=0\,, \quad  
 \bar r:=\frac{\sqrt R}{\CalgMo}={\sqrt {R/10}}\,.
\end{equation}
  Then Assumption
\ref{hip0}   is satisfied by Lemma \ref{bonhoeffer}.
 We have that
 \begin{equation}\label{poggiobustone}
 |P|_{\bar r,0,\tw_0}=\| P\|_{\bar r,1}
 \stackrel{\eqref{stimazero2}}\leq
2 |f|_{R}
\end{equation}
By \eqref{borgia}
$$
r_\star\geq \frac{\sqrt{\g R}}{2^6\sqrt{10 e|f|_{R}}} 
$$
 Then, recalling \eqref{mmm}
 $$
 \widehat{\tr}
\geq
\mathtt r(\MS)
\,.
 $$
Moreover
$$
\widehat{\mathtt C}_1
\leq
\mathtt C_1(\MS)\,,
\qquad
 \widehat{\mathtt C}_2
\leq
\mathtt C_2(\MS)\,,\qquad
\widehat{\mathtt C}_3
	\leq
	\mathtt C_3(\MS)
	\,,
\qquad
\mathtt C_1(\MS) (\mathtt r(\MS))^2\leq \frac18\,.
$$

$\GE$) Set
 \begin{equation}\label{casaprotaGE}
 \tw_{n,j}:=e^{a|j|+ \big(s+(\frac{n}{\suca}-1)\eta\big) \jap{j}^\theta}
		\jap{j}^p\,,\qquad
\eta=\eta_\GE:=\min\left\{\frac{\ta-a}{2},s\right\}\,, \ \  
 \bar r:=\frac{\sqrt R}{\Calg}\,.
 \end{equation}
 Then Assumption
\ref{hip0}   is satisfied by Lemma \ref{bonhoeffer}.
 We have that
 \begin{equation}\label{poggibonsi}
 |P|_{\bar r,\eta,\tw_0}=| P|_{\bar r,p,s-\eta,a,\eta}
 \stackrel{\eqref{stimazero}}\leq
\CNem(p,s-\eta,\ta- a -\eta)|f|_{\ta,R}
\end{equation}
By \eqref{borgia}
\begin{eqnarray*}
r_\star
\geq
\delta_\GE\,.
\end{eqnarray*}
Then, recalling \eqref{mmm} and Lemma \ref{bonhoeffer}, 
 \begin{eqnarray*}  
\widehat{\tr}
\geq
\mathtt r(\GE)\,,
\qquad
\widehat{\mathtt C}_1\leq
\mathtt C_1(\GE)\,,
\qquad
\widehat{\mathtt C}_2\leq
\mathtt C_2(\GE)\,,
\qquad
\widehat{\mathtt C}_3\leq
\mathtt C_3(\GE)\,,
\qquad
\mathtt C_1(\GE) (\mathtt r(\GE))^2\leq \frac18
\,.
\end{eqnarray*}
The proof of Theorem \ref{sob} is now completed.


\section{Proof of Theorem \ref{tarzanello} }\label{fine}

Theorem \ref{tarzanello} follows from Theorem \ref{sob}.
 
 We need the following auxiliary result,
 whose proof is postponed to the Appendix.
 
 \begin{lemma}\label{cobra}
 On the Hilbert space $\th_\tw$ consider the dynamical system
  \[
  \dot v = X_{\mathcal N}+X_R\,,
  \qquad
  v(0)=v_0\,, \qquad
  |v_0|_\tw\leq \frac34 r\,,
 \]
 where 
$\mathcal N\in \mathcal{A}_r(\th_{\mathtt w}) $
and $R\in\cH_{r,\eta}(\th_{\mathtt w})$
for some $r>0,\eta\ge 0.$
Assume that 
$$
\Re (X_{\mathcal N},v)_{\th_\tw}=0\,.
$$  
Then
\begin{equation}\label{giuncata}
\Big| |v(t)|_\tw-|v_0|_\tw\Big|< \frac{r}{8}\,, \qquad
\forall\,  |t|\leq 
\frac{1}{8|R|_{r,\eta,\tw}}\,.
\end{equation}
\end{lemma}

 Let us now prove Theorem \ref{tarzanello}, starting with
 the case $\SO)$.
 
 $\SO$) 
 Set
 $$
 r:=2\delta\,.
 $$
 Recalling \eqref{casaprota} and \eqref{salisano},
 by Corollary \ref{neminchione}
 solutions of the PDE  \eqref{NLSb}
 in the Sobolev space $\th_{p,0,0},$
correspond, by Fourier identification \eqref{silvacane},
  to orbits of  the Hamiltonian System  
 \eqref{hamNLS} in the space (recall Definition \ref{farfa})
 $$
 \th_\tw=\th^p \qquad
  {\rm with}
  \qquad 
  \tw_j
 =\jap{j}^{p}
 \quad({\rm and} \ |\cdot|_\tw=|\cdot|_p)\,.
 $$
An initial datum $u_0$
 satisfying 	$
|u_0|_{L^2}+ |\partial_x^p u_0|_{L^2} 
\le
\delta
$
 corresponds to\footnote{We still denote it by
 $u_0$.}
 $u_0\in \th^p$ with
 $|u_0|_p\leq \delta$
 by \eqref{vallinsu}.
 We want to apply the Birkhoff Normal Form Theorem \ref{sob}
 with 
 $$
 \suca=(p-1)/\tauSO\,.
 $$
 Recalling the definition of $\mathtt r(\SO,\suca)$ and noting that
 $\Calgo=2\sqrt{1+\pi^2/3}$,
 we have to verify that, for any $\suca\geq 1$
 \begin{equation}\label{cunegonda}
 \deSO ({\mathtt k}_\SO p)^{ -3 p }
 \leq 
 \frac{ \dede }
{2\sqrt{\suca 
\Cdue(4\suca,\ta/2\suca,\suca\tau) }}
\,.
\end{equation}
Indeed
\begin{eqnarray*}
&& \frac{ \dede }
{2\sqrt{\suca 
\Cdue(4\suca,\ta/2\suca,\suca\tau) }}
\stackrel{\eqref{fara}}=
\frac{\dede }{2\sqrt 2 e^{27(2+\fp)/2}}
\frac{1}{\sqrt\suca}
\frac{1}{(\sqrt{12\tau}\suca \max\left\{2 , \ta^{-1/2}\right\})^{3\suca\tau}}
\\
&&=
\frac{\dede \sqrt{\tau} }{2\sqrt 2 e^{27(2+\fp)/2} }\left(\sqrt{\frac{12}{\tau}}\max\left\{2 , \ta^{-1/2}\right\}\right)^{-3(p-1)} (p-1)^{-3(p-1)-1/2}
\\
&&=
 \deSO ({\mathtt k}_\SO)^{ -3 p } (p-1)^{-3(p-1)-1/2}
\end{eqnarray*}
and \eqref{cunegonda} follows since 
$
p^{-3p}< (p-1)^{-3(p-1)-1/2}
$
for $p>1.$

 Then by Theorem \ref{sob} and setting
 $v_0 := \Psi^{-1} (u_0)$, 
 by \eqref{stracchinobis} we get
 $$
 |v_0|_p =
 |\Psi^{-1} (u_0)-u_0|_p +|u_0|_p 
 \leq \frac18 r+\frac12 r
 =\frac58 r\,.
 $$
By \eqref{duspaghibis} and \eqref{fara} 
\begin{eqnarray*}
&&
|R|_{r,0,\tw}
\leq
\mathtt C_3(\SO,\suca)(2\delta)^{2(\suca+1)}
\\
&& =  	\delta^2 \frac{2^{10} \CNem(1,0,\ta/2)|f|_{\ta,R}}{e R}
\pa{\frac{  \suca\Cmon(4\suca,\ta/2\suca,\tau) 
		\Cdue(4\suca,\ta/2\suca,\tau) \delta^2}{\dede^2}}^{\suca}\\
	&=& \delta^2 \frac{2^{10} \CNem(1,0,\ta/2)|f|_{\ta,R}}{e R}
\pa{\frac{   4 e^{27(2+\fp)} 3^{3\tau}  (4 \max\left\{4 ,  (1/\ta)\right\})^{4\tau}  \delta^2 }{\tau \dede^2}}^{\suca} ( \tau\suca)^{\frac{4\tau+1}{\tau}(\tau\suca) }\\
\end{eqnarray*}  
We claim that  (remember that  $\suca= (p-1)/\tau$)
\[
|R|_{r,0,\tw} < \frac{1}{8} \frac{\delta^2}{\mathtt T_\SO}({\mathtt K}_\SO p)^{ 5 p }
\left(\frac{\delta}{\deSO}\right)^{2\suca}\,.
\]
This holds true since
\begin{eqnarray*}
&&\frac{2^{13} \CNem(1,0,\ta/2)|f|_{\ta,R}}{e R}
\pa{\frac{   4 e^{27(2+\fp)} 3^{3\tau}  (4 \max\left\{4 ,  (1/\ta)\right\})^{4\tau}   \deSO^2}{\tau \dede^2}}^{\suca} (p-1)^{\frac{4\tau+1}{\tau}(p-1)}
\\
&&< \frac{1}{\mathtt T_\SO}({\mathtt K}_\SO p)^{ 5 p }\,,
\end{eqnarray*}
noting that
$
(p-1)^{\frac{4\tau+1}{\tau}(p-1)} < p^{5p}
$ for $p>1$
(recall that 
 $\tau\geq 15$).
We apply Lemma \ref{cobra} with $\mathcal N\rightsquigarrow D_\omega +Z,$
 $\eta\rightsquigarrow 0;$
  then by \eqref{giuncata}
  \begin{equation*}
  |v(t)|_p \leq
\Big| |v(t)|_p -|v_0|_p \Big|+|v_0|_p < \frac{r}{8}+\frac58 r
=\frac34 r\,, \qquad
\forall\,  |t|\leq 
({\bf C}_\SO p)^{ -\frac{4\tauSO+1}{\tauSO} p }(\frac{1}{\delta})^{\frac{2(p-1)}{\tauSO}+2}
\,.
\end{equation*}
Since $\Psi$ is symplectic we have that
$u(t)=\Psi(v(t))$; then  by \eqref{stracchinobis}
$$
|u(t)|_p
\leq |\Psi(v(t))-v(t)|_p+|v(t)|_p
\leq \frac18 r +\frac34 r<2\delta\,,\qquad
\forall\,  |t|\leq 
\frac{\mathtt T_\SO}{\delta^2}({\mathtt K}_\SO p)^{ -5 p }
\left(\frac{\deSO}{\delta}\right)^{\frac{2(p-1)}{\tauSO}}\,.
$$  
Finally by \eqref{vallinsu}
we get
$$
|u(t)|_{L^2}+ |\partial_x^p u(t)|_{L^2}\leq 
2|u(t)|_p
\leq 4\delta\,,\qquad
\forall\,  |t|\leq 
\frac{\mathtt T_\SO}{\delta^2}({\mathtt K}_\SO p)^{ -5 p }
\left(\frac{\deSO}{\delta}\right)^{\frac{2(p-1)}{\tauSO}}\,,
$$
proving \eqref{tommaso3}. 

$\MS$)  It is similar to the previous case but now
 $$
 \th_\tw=\th^p \qquad
  {\rm with}
  \qquad 
\tw_j
 =\jml{j}^{p}
 \quad({\rm and} \ |\cdot|_\tw=\|\cdot\|_p)\,.
 $$
An initial datum $u_0$
 satisfying 
 $
2^p |u_0|_{L^2}+ |\partial_x^p u_0|_{L^2} 
\le
\delta
$
 corresponds to
 $u_0\in \th^p$ with
 $\|u_0\|_p\leq \delta$
 by \eqref{vallinsuMS}.
 Now we can apply the Birkhoff Normal Form Theorem \ref{sob}
 with $\suca=(p-1)/\tauMS$
 since, for any $\suca\geq 1$
 \begin{equation}\label{cunegonda2}
 \min\left\{
 \frac{2\sqrt \tauMS \delta_\MS}{\sqrt p}\,,
\ 
\frac{\sqrt R}{4\sqrt{10}}
\right\}
 \leq 
\frac12\mathtt r(\MS,\suca)
=\min\left\{
\frac{2 \delta_\MS }{\sqrt{\suca}}\,,
\ 
\frac{\sqrt R}{4\sqrt{10}}
\right\}
\,.
\end{equation}
 Proceeding as in the case $\SO$ and noting that now
\begin{eqnarray*}
8|R|_{r,0,\tw}
\leq
8\mathtt C_3(\MS,\suca) (2\delta)^{2(\suca+1)}
=
5\cdot 2^9\frac{|f|_{R}}{R}
 \left( \frac{\suca \delta^2}
 {4 \delta_\MS^2}\right)^{\suca}
 \delta^{2}
 =\frac{1}{\mathtt T_\MS}
  \left( \frac{(p-1) \delta^2}
 {4\tauMS \delta_\MS^2}\right)^{\frac{p-1}{\tauMS}}
 \delta^{2}
\,,
\end{eqnarray*}  
we get
$$
\|u(t)\|_p
\leq 2\delta\,,\qquad
\forall\,  |t|
\leq 
 \frac{\mathtt T_\MS}{\delta^{2}}
  \left( 
  \frac{4\tauMS \delta_\MS^2}{(p-1) \delta^2}
 \right)^{\frac{p-1}{\tauMS}}
\,.
$$  
Finally by \eqref{vallinsuMS}
we get
$$
2^p|u(t)|_{L^2}+ |\partial_x^p u(t)|_{L^2}\leq 
2\|u(t)\|_p
\leq 4\delta\,,\qquad
\forall\,  |t|
\leq 
 \frac{\mathtt T_\MS}{\delta^{2}}
  \left( 
  \frac{4\tauMS \delta_\MS^2}{(p-1) \delta^2}
 \right)^{\frac{p-1}{\tauMS}}
\,,
$$
proving \eqref{boh3}.

 \section{Proof of Theorems \ref{sorbolev!} and \ref{gegge}}\label{fine2}

 We start considering 
 Theorem \ref{sorbolev!}, case $\SO.$

 $\SO$)  We start by noticing that for $3p \ln (\mathtt k_\SO p) \le \ln(\deSO/\delta)$ the function 
 $\frac{\mathtt T_\SO}{\delta^2}({\mathtt K}_\SO p)^{ -5 p }
 \left(\frac{\deSO}{\delta}\right)^{\frac{2(p-1)}{\tauSO}}$ is increasing in $p$. 
 
 Let us check that $p(\delta)$ defined in \eqref{checco} satisfies \eqref{moro},
 namely, passing to the logarithms and setting
 $y :=  \ln(\deSO/\delta)$, 
 we have to check that   $ 3p \ln (\mathtt k_\SO p) \le y.$
 Indeed we have
\begin{eqnarray*}
3p \ln 
(\mathtt k_\SO p) 
\le 
3 \left(1+ \frac16 \frac{y}{\ln(y)}\right)  \left(\ln (\mathtt k_\SO)+\ln(1+ \frac16  \frac{y}{\ln(y)} )\right) 
\leq y
\end{eqnarray*}
 provided that\footnote{
 Note that the function
 $$
 y \mapsto 
 y-3 \left(1+ \frac16 \frac{y}{\ln(y)}\right)  \left(\ln y+\ln(1+ \frac16  \frac{y}{\ln(y)} )\right) 
 $$ is positive for $y\geq 40.$}
 \[
 y \ge \max\{\mathtt k_\SO,40\}.
 \]
 Now we have to show that 
 \[
  \frac{\mathtt T_\SO}{\delta^2} e^{ \ \frac{\ln^2 (\deSO/\delta)}{4 \tau \ln \ln (\deSO/\delta)}} \le 
  \frac{\mathtt T_\SO}{\delta^2}({\mathtt K}_\SO p)^{ -5 p }
  \left(\frac{\deSO}{\delta}\right)^{\frac{2(p-1)}{\tau}}
 \]
 wich amounts to
 \[
 e^{ \ \frac{y^2}{4 \tau  \ln y }}({\mathtt K}_\SO p)^{ 5 p } e^{-\frac{2(p-1)}{\tau} y }\le 1
 \]
or equivalently
\[
 \frac{y^2}{4 \tau \ \ln y } + 5 p \ln({\mathtt K}_\SO p) - \frac{2(p-1)}{\tau} y \le 0\,.
\]
Substituting $1+\frac{y}{6 \ln y}-\tau< p< \frac{y}{3 \ln y}$ we get
\begin{eqnarray*}
&&\frac{y^2}{4 \tau  \ln y } + 5 p \ln({\mathtt K}_\SO p) - \frac{2(p-1)}{\tau} y \le \frac{y^2}{4 \tau  \ln y } + \frac{5}{3} \frac{y}{\ln(y)} \ln(\frac{{\mathtt K}_\SO}{3} \frac{y}{\ln(y)})- 2 y(\frac{1}{6\tau} \frac{y}{\ln(y)}-1)
\\
&&
\le
- \frac{y^2}{12 \tau\ln y } + \frac{5}{3} \frac{y}{\ln(y)}\ln(\frac{{\mathtt K}_\SO}{3\ln(y)})+
\frac{11}{3} y <0
\end{eqnarray*}
if
\[ 
y \ge 50 \tau \ln({\mathtt K}_\SO /3)\,,\quad \frac{y}{\ln(y)} \ge 88\tau\,.
\]
Note that the last inequality holds if 
$y\geq 88 \tau^2$ (recall that $\tau\geq 15$).
Recollecting the condition that $y$ has to satisfy is
$$
 y \ge \max\{\mathtt k_\SO,\, 50 \tau \ln({\mathtt K}_\SO /3),\, 88\tau^2\}\,,
$$
namely $\delta\le \bar{\delta}_\SO$.

$\MS$) Since we are assuming
 $\delta\leq\bar\delta_\MS$ we have that $p$ defined in 
 \eqref{elisabetta} satisfies
 $p>1$ and that \eqref{eisenach} holds.
 Then 
 Theorem \ref{tarzanello} applies
 and \eqref{boh} follows by\footnote{
 Noting that $(4x/[x])^{[x]}\geq e^x$ for $x=\delta_\MS^2/\delta^2\geq 1.$} 
 \eqref{boh3}.


 

\bigskip

We finally prove Theorem \ref{gegge}.
\\
We proceed as in the proof of Theorem  \ref{sorbolev!}.
Let us choose
\begin{equation}\label{contiglianoGE}
\suca(r):=\left[\pa{2\ln \frac{2\delta_\GE}{r}}^{\teta/4}\right]
=\left[\pa{2\ln \frac{\delta_\GE}{\delta}}^{\teta/4}\right]\,.
\end{equation}
Recalling \eqref{casaprotaGE} and \eqref{salisano},
 by Corollary \ref{neminchione}
 solutions of the PDE  \eqref{NLSb}
 in the  space $\th_{p,s,a}$,
correspond, by Fourier identification \eqref{silvacane},
  to orbits of  the Hamiltonian System  
 \eqref{hamNLS} in the space
 $$
 \th_\tw\qquad
  {\rm with}
  \qquad 
  \tw_j=e^{a|j|+ s \jap{j}^\theta}
		\jap{j}^p\,.
 $$
An initial datum $u_0$
 satisfying $|u_0|_{p,s,a} \le \delta$ 
 corresponds to\footnote{We still denote it by
 $u_0$.}
 $u_0\in \th_\tw$ with
 $|u_0|_\tw\leq \delta$.
\\
We claim that  $r\leq 2\mathtt c_\GE\delta_\GE$  
implies 
\begin{equation}\label{bottargaGE}
\frac{r \suca e^{\Cuno\pa{\frac{\suca}{\eta_\GE}}^{\frac{3}{\theta}}}}{2 \delta_\GE}
\leq 1\,.
\end{equation}
Indeed we have $\suca(r)\geq\suca_\GE$
and
 by \eqref{contiglianoGE}
$r\leq 2 \delta_\GE e^{-\frac12 (\suca(r)/2)^{4/\theta}}$
and \eqref{bottargaGE} follows if we show that the function
$$
\suca \ \to\ 
e^{-\frac12 (\suca/2)^{4/\theta}}
\suca e^{\Cuno\pa{\frac{\suca}{\eta_\GE}}^{\frac{3}{\theta}}}
$$
is $\leq 1$ for $\suca\geq \suca_\GE.$
This is true since the function is decreasing for 
$\suca\geq \suca_\GE$ and is $\leq 1$ for
 $\suca= \suca_\GE.$ This proves the claim \eqref{bottargaGE}.
\\
Then by \eqref{duspaghibis} and \eqref{bottargaGE} 
\begin{eqnarray*}
|R|_{r,0,\tw}
&\leq&
\mathtt C_3(\GE) r^{2(\suca+1)}
\leq
\frac{\g r^2}{2^9 e \delta_\GE^2}
	\pa{\frac{r }{2 \delta_\GE}}^{\suca(r)}
=
\frac{\g \delta^2}{2^7 e \delta_\GE^2}
	\pa{\frac{\delta }{ \delta_\GE}}^{\suca(r)}
	\\
	&\leq&
	\frac{\g \delta^2}{2^7 e \delta_\GE^2}
e^{-\pa{\ln\frac{\delta_\GE}{\delta}}^{1+\theta/4}}\,,
\end{eqnarray*}  
since
$\suca(r)\geq\pa{\ln \frac{2\delta_\GE}{r}}^{\teta/4}
=\pa{\ln \frac{\delta_\GE}{\delta}}^{\teta/4}$.
Proceeding as in the proof of Theorem \ref{sorbolev!}
we get
$$
|u(t)|_{p,s,a}
\leq 2\delta\,,\qquad
\forall\,  |t|\leq 
\frac{\mathtt T_\GE}{\delta^2} 
	e^{\pa{\ln\frac{\delta_\GE}{\delta}}^{1+\theta/4}}\,.
$$  
The proof 
of Theorem 
\ref{gegge} is completed.

\appendix

 \section{Constants.}\label{ossobuco}
 
 In this subsection are listed all the constants appearing
 along the paper.
 We first introduce some auxiliary constants.
 Given $t,\s,\zeta>0,$  $p>1/2,$ $0<\teta<1,$ $s,\fp\geq 0,$ 
 we
 set\footnote{Regarding $\CNem$
 	note that
 	$$
 	\sup_{x\ge 1} x^p e^{- t x+s x^\theta} 
 	\le 
 	\exp\pa{(1-\teta)\pa{\frac{s}{t^\teta}}^{\frac{1}{1-\teta}}}\max
 	\left\{\frac{p }{e (1-\teta)t}, e^{-\frac{t(1-\teta)}{p}}\right\}^p\,.
 	$$
 	}
  \begin{eqnarray}
 		\Cmon(t, \s,p)
 		&:=&
 		 2^{p+1} p^{p}
 		\max\left\{\pa{2t}^{p} ,  \s^{-p}, 1\right\}\,,
 		\nonumber
 		\\	
 \Calg
 &:=&
 2^p\Big(\sum_{i\in\Z} \jap{i}^{-2p}\Big)^{1/2}
 \,,
 		\nonumber
 		\\
 \CalgM
 &:=&
 \sqrt{4+2\frac{2p+1}{2p-1}}
 \,,
 		\nonumber
 		\\
 \CNem(p,s,t)
 &:=&
  \Calg
 \big(e^s+\sup_{x\ge 1} x^p e^{- t x+s x^\theta} \big)\,,
 	\nonumber
 		\\			
 		C_*  
 	 &:=&	13/(1-\teta)\,,			
 \nonumber
 		\\	
 \Cuno
 &:=&
 	28 \,\teta^{-1}(\fp+3)
 	\Big(8(\fp+3)C_*\theta^{-1}\Big)^{\frac{2}{\teta}}
 	\pa{\ln \big(8(\fp+3)C_*\theta^{-1}\big)}^{\frac{2}{\teta}+1}
 \,,
 		\nonumber
 		\\
 \Cdue(t,\s,\zeta)
 &:=& e^{27(2+\fp)}\Cmon(t,\s,3\zeta)\,,
 	\nonumber
 		\\
 	\tau&:=& \tau_0(2+\fp)\,,\qquad
 	 \tau_0:=15/2\,, \qquad \tau_1:=2\pa{\tau_0+\frac{3}{2\ln 2}}(2+\fp)
 	\nonumber
 \end{eqnarray}
 Here are the constants appearing
 in Theorem \ref{tarzanello}:
 \begin{eqnarray*}
 \tauSO&:=& \tau\,,
 \\
  \deSO&:=& \frac{  \sqrt{\g R} \left(\sqrt{3}\max\left\{2 , \ta^{-1/2}\right\}\right)^3}{2^7 \tau  e^{27(2+\fp)/2}\sqrt{ \CNem(1,0,\ta/2)|f|_{\ta,R}} }\,,\\
  {\mathtt k}_\SO&:=& \sqrt{\frac{12}{\tau}}\max\left\{2 , \ta^{-1/2}\right\}
 \\
 \mathtt T_\SO &:=& \frac{ e R \tau^4  \mathtt k_\SO^{8+6/\tau} } {3 \cdot 2^{16} \CNem(1,0,\ta/2)|f|_{\ta,R}}   
 \\
 \mathtt K_\SO&:=&\pa{ \frac{      \tau^{4\tau} \mathtt k_\SO^{8\tau+6} }{ 2 \cdot 3^{\tau} } }^{1/(5\tau)}
 \\
 \tauMS&:=&\tau_1\,,
 \\ 
 \delta_\MS
 &:=&
 \frac{\sqrt{ \g R }}
 {\sqrt{5\cdot 2^{17} e 6^{\tau_1} (4^6 e^{27})^{2+\fp}|f|_{R}}}\,,\\
 \mathtt T_\MS
 &:=&
 \frac{R}{5\cdot 2^{9}|f|_{R}}
 \,.
 \end{eqnarray*}
 \
  Here are the constants appearing
 in Theorem \ref{sorbolev!}:
 \begin{eqnarray*}
 \bar{\delta}_\SO
 &:=&
 \min\left\{
\delta_\SO
\exp\big(-\max\{\mathtt k_\SO,\, 50 \tau \ln({\mathtt K}_\SO /3),\, 88\tau^2\}\big)
 \,,\ 
  \frac{\sqrt R}{10}
 \right\}
   \\
 \bar\delta_\MS
 &:=&
 \min\left\{\delta_\MS\,,\ 
 \frac{\sqrt R}{4\sqrt{10}}\right\}
 \,.
 \end{eqnarray*}
 Here are the constants appearinging
 in Theorem \ref{gegge}:
 \begin{eqnarray*}
 \bar\delta_\GE
 &:=&
 \min\left\{\mathtt c_\GE \delta_\GE 
 \,,\ 
 \frac{\sqrt R}{4\Calg}\right\}
 \,,
 \\
 \delta_\GE
 &:=&
  \frac{\sqrt{\g R}}{\Calg\sqrt{2^{11}e 
  \CNem(p,s-\eta_\GE,\ta- a-\eta_\GE)|f|_{\ta,R}}}\,,
  \qquad
  \eta_\GE:=\min\left\{\frac{\ta-a}{2},s\right\}\,,
 \\
 \mathtt c_\GE
 &:=&
 \exp\left( - \left( \max\left\{
 \frac{16(4\Cuno)^\theta}{\eta_\GE^{3}},
 2^{\frac{2\theta+4}{4-\theta}}
 \right\}\right)^{4/\theta}\right)
 \,,
 \\
 \mathtt T_\GE
&:=&
 \frac{2^4 e \delta_\GE^2}{\g}\,.
 \end{eqnarray*}
 Finally, with respect to the three cases $\SO,\MS,\GE,$
 we define  the constants
 $\mathtt r, \mathtt C_1, \mathtt C_2, \mathtt C_3$
 of Theorem \ref{sob}.
 To stress their dependence on the three cases 
 $\SO,\MS,\GE,$ we denote them with
 $\mathtt r(\SO),\mathtt r(\MS),\mathtt r(\GE),$
 $\mathtt C_1(\SO),\mathtt C_1(\MS),\mathtt C_1(\GE)$ etc.
 or also $\mathtt r(\SO,\suca)$ etc. when we want to emphasise 
 the dependence on
 $\suca:$
 \begin{eqnarray*}
 \mathtt r(\SO)
 &=&
 \mathtt r(\SO,\suca)
 :=
 \min\left\{
 \frac{ \dede }
 {\sqrt{\suca 
 \Cdue(4\suca,\ta/2\suca,\suca\tau) }}\,,\ 
 \frac{\sqrt R}{5}
 \right\}\,,
 \\
 \mathtt C_1(\SO)
 &:=&
 \frac{\Cdue(4\suca,\ta/2\suca,\suca\tau)}{2^{7}
 e \dede^2}
 \,,
 \\
 \mathtt C_2(\SO)
 &:=&
 \frac{\g}{2^8 e \dede^2}
 \,,
 \\
 \mathtt C_3(\SO)
 &=&
 \mathtt C_3(\SO,\suca)
 :=
 \frac{2^8 \CNem(1,0,\ta/2)|f|_{\ta,R}}{e R}
 	\pa{\frac{  \suca\Cmon(4\suca,\ta/2\suca,\tau) 
  \Cdue(4\suca,\ta/2\suca,\tau)}{4\dede^2}}^{\suca}
 \,,
 \\
 \dede
 &:=&
 \frac{\sqrt{\g R}}{\sqrt{2^{17} \CNem(1,0,\ta/2)|f|_{\ta,R}}}\,,
\end{eqnarray*}
 \begin{eqnarray*}
 \mathtt r(\MS)
 &:=&
 \mathtt r(\MS,\suca):=
 \min\left\{
 \frac{ 4\delta_\MS }{\sqrt{\suca}}\,,
 \ 
 \sqrt{R/40}\right\}
 \,,\qquad\qquad\qquad\qquad\qquad\qquad\qquad\qquad
 \\
 \mathtt C_1(\MS)
 &:=&
 \frac{1}{2^9 e \delta_\MS}
 \,,
 \\
 \mathtt C_2(\MS)
 &:=&
 \frac{5\cdot 2^5|f|_R}{R}
 \,,
 \\
 \mathtt C_3(\MS)
 &=&
 \mathtt C_3(\MS,\suca):=
  80\frac{|f|_{R}}{R}
  \Big( \suca
  /16 \delta_\MS^2\Big)^{\suca}
 \,,
\end{eqnarray*}
\begin{eqnarray*}
 \mathtt r(\GE)
 &:=&
 \min\left\{\frac{ \delta_\GE }
 {\sqrt{\suca} 
 e^{\frac12\Cuno\pa{\frac{\suca}{\eta_\GE}}^{\frac{3}{\theta}}}}
 \,,\ 
 \frac{\sqrt R}{2\Calg}\right\}
 \,, \qquad\qquad\qquad\qquad\qquad\qquad\qquad
 \\
 \mathtt C_1(\GE)
 &:=&
 \frac{e^{\Cuno\pa{\frac{\suca}{\eta_\GE}}^{\frac{3}{\theta}}}}{2^7 e \delta_\GE^2}
 \,,
 \\
 \mathtt C_2(\GE)
 &:=&
 \frac{\g}{2^8 e \delta_\GE^2}
 \,,
 \\
 \mathtt C_3(\GE)
 &:=&
 \frac{\g}{2^9 e \delta_\GE^2}
 	\pa{\frac{ \suca e^{\Cuno\pa{\frac{\suca}{\eta_\GE}}^{\frac{3}{\theta}}}}{4 \delta_\GE^2}}^{\suca}
 \,.
 \end{eqnarray*}

\section{Proofs of the main properties of the norms}

\begin{lemma}\label{chiappa}
	For $p,\beta>0$ and $x_0\geq 0$ we have that 
	$$
	\max_{x\geq x_0} x^p e^{-\beta x}=
	\begin{cases}
	&(p/\beta)^p e^{-p}
	\quad \mbox{if}  \; \ \  
	x_0\leq p/\beta
	,\\
	&x_0^p e^{-\beta x_0}
	\quad \mbox{if} \; \ \  
	x_0> p/\beta.
	\end{cases}
	$$
\end{lemma}

\begin{lemma}\label{brabante}
 Let $0<a<1$ and $x_1\geq x_2\geq \ldots\geq x_N\geq 2.$ Then
 $$
 \frac{\sum_{1\leq\ell\leq N} x_\ell}{\prod_{1\leq\ell\leq N} x_\ell^a}
\leq 
x_1^{1-a}+\frac{2}{a x_1^a}\,.
 $$
\end{lemma}
\begin{proof}
 By induction over $N$. It is obviously true for $N=1.$
 Assume that it hols for $N$ and prove it for $N+1.$
\end{proof}

\noindent

\subsection{Proof of Lemma \ref{veronese}}\label{tec1}

We first note that (see, e.g. Lemma 17 of Biasco-Di Gregorio, ARMA 2010)
for $p >1/2$ and every sequence $\{x_i\}_{i\in\Z}$, $x_i\geq 0,$
\begin{equation*}
    \left( \sum_{i\in \Z} x_i \right)^2
    \leq
    c\sum_{i\in\Z} \left(\frac{\jap{i}^p\jap{j-i}^p}{\jap{j}^p} x_i\right)^2\,,
\end{equation*}
with
$c:=4^p\sum_{i\in\Z} \jap{i}^{-2p}=(\Calg)^2.$
Then
\begin{eqnarray*}
\norm{f\star g}_{p,s,a}^{2}
&\le&
 \sum_{j}e^{2s\jap{j}^\teta} e^{2a\abs{j}}\jap{j}^{2p}\Big(\sum_{i}|f_{i}||g_{j-i}|\Big)^2\\
 &\le& 
 c
 \sum_{j}e^{2s\jap{j}^\teta} e^{2a\abs{j}}\sum_{i}
 \jap{i}^{2p}\jap{j-i}^{2p}|f_{i}|^2
|g_{j-i}|^2\\
  &=& 
 c
 \sum_{i}
 e^{2s\jap{i}^\teta} e^{2a\abs{i}}
 \jap{i}^{2p}
 |f_{i}|^2
  \sum_{j}
 \jap{j-i}^{2p}
 e^{2s\jap{j-i}^\teta} e^{2a\abs{j-i}}
|g_{j-i}|^2\\
&=&
c
\norm{f}_{p,s,a}^{2}
\norm{g}_{p,s,a}^{2}\,.\qed
\end{eqnarray*}

\noindent
\subsection{Proof of Lemma \ref{tiepolo}}\label{latempesta}
Set
$$
\phi(i,j):=\frac{\jml{j}}{\jml{i}\jml{j-i}}\,,\qquad
\qquad
\forall\, i,j\in\mathbb Z
\,.
$$
Note that
\begin{equation}\label{lalla}
\phi(i,j)=\phi(j,i)=\phi(-i,-j)\,.
\end{equation}
We claim that
\begin{equation}\label{lalla2}
\phi(i,j)\leq 1\,.
\end{equation}
Indeed by \eqref{lalla} we can consider only the case
$j\geq 0.$
Since
$
\phi(-|i|,j)\leq \phi(|i|,j)
$ we can consider only the case
$i\geq 0$.
Again by \eqref{lalla} we can assume $j\geq i.$
In particular we can take $j>i>0,$  \eqref{lalla2}  being trivial
in the  cases
$j=i,$ $i=0$.
We have
$$
\phi(i+1,i)=\frac{i+1}{2\jml{i}}\leq\frac34\,,
\qquad
\phi(j,1)=\frac{j}{2(j-1)}\leq 1\,.
$$
Then it remains also to discuss
the case $j-2\geq i\geq 2;$ we have
$$
\phi(i,j)=\frac{j}{i(j-i)}=\frac{1}{i}+\frac{1}{j-i}\leq1\,,
$$
proving \eqref{lalla2}.

For $q\geq 0$ set
\begin{equation}\label{giorgione}
c_q:=\sup_{j\in\mathtt Z}\sum_{i\in\mathtt Z}(\phi(i,j))^q
=
\sup_{j\geq 0}\sum_{i\in\mathtt Z}(\phi(i,j))^q\,.
\end{equation}
We claim that 
\begin{equation}\label{verrocchio}
c_q\leq 4+2\frac{q+1}{q-1}<\infty\,,\qquad \forall\, q>1\,.
\end{equation}
Indeed, since $\jml{j}/\jml{j+1}\leq 1$ and
$\jml{j}/\jml{j-1}\leq 3/2$ for $j\geq 0$, we 
have\footnote{Note that the term $\left(\frac{1}{i}+\frac{1}{(j-i)}\right)^q$
for $j=4$ and $i=2$ is 1 for every $q$.}
\begin{eqnarray*}
c_q
&=&
\sup_{j\geq 0}
\left(
\frac{\jml{j}^q}{2^{q-1}\jml{j+1}^q}
+\frac{1}{2^{q-1}}
+\frac{\jml{j}^q}{2^{q-1}\jml{j-1}^q}
+
\sum_{i\leq -2,\, 2\leq i\leq j-2,\, i\geq j+2}
(\phi(i,j))^q
\right)
\\
&\leq&
2^{3-q}
+
\sup_{j\geq 0}
\left(
\sum_{i\geq 2} \frac{\jml{j}^q}{i^q(j+i)^q}
+\sum_{2\leq i\leq j-2}\left(\frac{1}{i}+\frac{1}{(j-i)}\right)^q
+\sum_{i\geq j+2} \frac{\jml{j}^q}{i^q(i-j)^q}
\right)
\\
&\leq&
2^{3-q}
+\sum_{i\geq 2} \frac{1}{i^q}
+2^{q-1}\sum_{2\leq i\leq j-2}\left(\frac{1}{i^q}+\frac{1}{(j-i)^q}\right)
+\sum_{i\geq j+2} \frac{1}{(i-j)^q}
\\
&\leq&
4+2\frac{q+1}{q-1}\,,
\end{eqnarray*}
using that
$(x+y)^q\leq 2^{q-1}(x^q+y^q)$ for $x,y\geq 	 0$
and that\footnote{
$\sum_{i\geq 2} 1^{-q}\leq 2^{-q}+\int_2^\infty x^{-q}dx$} 
$$
\sum_{i\geq 2} i^{-q}\leq \frac{q+1}{2^q(q-1)}\,.
$$

Note that for every 
$q,q_0\geq 0$
we have
\begin{equation}\label{beatoangelico}
c_{q_0+q}\leq c_{q_0}
\end{equation}
since
$$
c_{q_0+q}
:=
\sup_{j\in\mathtt Z}\sum_{i\in\mathtt Z}(\phi(i,j))^{q_0}
(\phi(i,j))^q
\stackrel{\eqref{lalla2}}
\leq
\sup_{j\in\mathtt Z}\sum_{i\in\mathtt Z}(\phi(i,j))^{q_0}
=c_{q_0}\,.
$$

We now note that 
for $p >1/2$, $j\in\mathbb Z$ and every sequence $\{x_i\}_{i\in\Z}$, $x_i\geq 0,$
we have by Cauchy-Schwarz inequality
\begin{equation*}\label{zampa1}
    \left( \sum_{i\in \Z} x_i \right)^2
    =
   \left( \sum_{i\in \Z} (\phi(i,j))^p (\phi(i,j))^{-p} x_i \right)^2 
    \leq
    c_{2p}\sum_{i\in\Z} \left((\phi(i,j))^{-p} x_i\right)^2\,,
\end{equation*}
with $c_{2p}$ defined in \eqref{giorgione}.
Using the above inequality we get
\begin{eqnarray*}
\|f\star g\|_p^{2}
&\le&
 \sum_{j}\jml{j}^{2p}\Big(\sum_{i}|f_{i}||g_{j-i}|\Big)^2\\
 &\le& 
c_{2p}
 \sum_{j}\sum_{i}
 \jml{i}^{2p}\jml{j-i}^{2p}|f_{i}|^2
|g_{j-i}|^2\\
  &=& 
c_{2p}
 \sum_{i}
 \jml{i}^{2p}
 |f_{i}|^2
  \sum_{j}
 \jml{j-i}^{2p}
|g_{j-i}|^2\\
&=&
c_{2p} 
\|f\|_p^2
\|g\|_p^2\,.
\end{eqnarray*}
The proof ends recalling \eqref{verrocchio}.
\qed


\noindent
\subsection{Proof of Lemma \ref{neminchia}}\label{tec2}
(i)
	By definition the $\eta$-majorant Hamiltonian is 
	\begin{align*}
	\und{H}_\eta =	\sum_d  \sum_{\substack{j_0,j_1\dots, j_{2d}\\
	j_0+\sum_{i=1}^{2d} (-1)^{i} j_i = 0 }} e^{\eta |\pi_{j_1,\dots,j_{2d}}|}|F^{(d)}_{j_0}| \overline{u_{j_1}}  u_{j_2}\overline{u_{j_3}}\dots u_{j_{2d}}
	\end{align*} where
	\[
	\pi_{j_1,\dots,j_{2d}} = \sum_{i=1}^{2d} (-1)^{i} j_i= -j_0\,,
	\]
	hence
	\[
	\und{H}_\eta =	\sum_d \pa{\und{F}_\eta^{(d)}\star\underbrace{ u \star \cdots\star u}_{d \;\mbox{times}} \star \underbrace{\bar u \star \cdots\star \bar u}_{d \;\mbox{times}}}_0\,,\quad \und{F}^{(d)}_\eta:= \pa{ e^{\eta|j|}|F^{(d)}_j|}_{j\in \Z}\,.
	\]
	consequently 
	\[
	X^{(j)}_{\und{H}_\eta}= \sum_d d \pa{\und{F}_\eta^{(d)}\star\underbrace{ u \star \cdots\star u}_{d \;\mbox{times}} \star \underbrace{\bar u \star \cdots\star \bar u}_{d-1 \;\mbox{times}}}_j\, .
	\]
	Moreover
	\[
	|X_{\und{H}_\eta}|_{p,s,a} \le \sum_d d 
	(\Calg)^{2d-1}| \und{F}_\eta^{(d)}|_{p,s,a} (|u|_{p,s,a} )^{2d-1}\,.
	\]
Since 
	\[ 
	|\und{F}_\eta^{(d)}|_{p,s,a} = |F^{(d)}|_{s,a+\eta,p} \le |F^{(d)}|_{ p,s,a_0} 
	\]
	we get
	\[
	|X_{\und{H}_\eta}|_{p,s,a} \le \sum_d d (\Calg)^{2d-1}
	|F^{(d)}|_{ p,s,a_0} 
	(|u|_{p,s,a} )^{2d-1}\,.
	\]
	Therefore
	\[
	|H|_{r,p,\eta}^{(p,s,a,0)} =
	r^{-1} \pa{\sup_{\norm{u}_{p,s,a} < r} \norm{{X}_{{\underline H}_\eta}}_{p,s,a} }
	\le 
	r^{-1}\sum_d d |F^{(d)}|_{p,s,a_0}  (\Calg r)^{2d-1}<\infty.
	\]
(ii) The proof is analogous to point (i).	
\qed


\noindent


\begin{lemma}\label{palis}
 Let $0<r_1<r.$ Let $E$ be a Banach space endowed with the norm $|\cdot|_E$.
 Let $X:B_r \to E$ a vector field satisfying
 $$\sup_{B_r}|X|_E\leq \delta_0\,.$$
 Then the flow $\Phi(u,t)$ of the vector field\footnote{Namely the solution 
 of the equation $\partial_t \Phi(u,t)=X(\Phi(u,t))$ with initial datum
 $\Phi(u,0)=u.$} is well defined for every 
 $$|t|\leq T:=\frac{r-r_1}{\delta_0}$$
 and $u\in B_{r_1}$
 with estimate
 $$
 |\Phi(u,t)-u|_E\leq \delta_0 |t|\,,\qquad
 \forall\, |t|\leq T
 \,.
 $$
\end{lemma}
\begin{proof}
Fix $u\in B_{r_1}$.
Let us first prove that $\Phi(u,t)$  exists $\forall\, |t|\leq T.$
Otherwise there exists a time\footnote{We assume $t_0$ positive, the negative case is analogous.} $0<t_0<T$ such that 
$|\Phi(u,t)|_E<r$ for every $0\leq t<t_0$ but
$|\Phi(u,t_0)|_E=r.$
Then, by the fundamental theorem of calculus 
\begin{equation}\label{integro}
\Phi(u,t_0)-u=\int_0^{t_0}X(\Phi(u,\tau))d\tau\,.
\end{equation}
Therefore
\begin{eqnarray*}
r-r_1 &\leq&
|\Phi(u,t_0)|_E-|u|_E\leq|\Phi(u,t_0)-u|_E\leq\int_0^{t_0}|X(\Phi(u,\tau))|_Ed\tau\leq \delta_0 t_0
\\
&<&\delta_0 T= r-r_1\,,
\end{eqnarray*}
which is a contradiction
Finally, for every $|t|\leq T,$
$$
|\Phi(u,t)-u|_E\leq\left|\int_0^{t}|X(\Phi(u,\tau))|_Ed\tau\right|
\leq \delta_0 |t|\,.
$$
\end{proof}
\subsection{Proof of Lemma \ref{stantuffo}}\label{tec3}
We first prove (i).
It is easily seen that:
\[
X_{\und H_\eta}^{(j)}(u) = \im \sum_{\bal,\bbt\in\N^\Z} \abs{H_{\bal,\bbt}}\bbt_j e^{\eta|\pi(\bal-\bbt)|}u^\bal \bar u^{\bbt-e_j}\,.
\]
Now 
\[
|X_{\und H_\eta}(u) |_\tw  \le |X_{\und H_\eta}(\und u) |_\tw
\,,\quad \und u=\pa{|u_j|}_{j\in \Z}
\]
hence, in evaluating  the supremum of
$|X_{\und H_\eta}|_\tw$ over $|u|_\tw \leq r$
we ca restrict to the case in which 
$u=(u_j)_{j\in\Z}$ has all real positive components. Hence
\begin{align*}
\abs{H}_{r,\eta,\tw} 
=
r^{-1}\sup_{\abs{u}_\tw \leq r} \abs{\pa{\sum^\ast \abs{H_{\bal,\bbt}}\bbt_j e^{\eta \abs{\pi(\bal-\bbt)}} \abs{u}^{\bal + \bbt - e_j} }_{j\in\Z}}_\tw  \,.
\end{align*}
Then
\begin{equation}
\label{norma2}
|H|_{r,\eta,\tw} 
= 
\frac{1}{2 r} \sup_{|u|_\tw\leq r} 
\norm{\pa{ W^{(j)}_{\eta}(u)}_{j\in\Z}}_{\tw} 
\end{equation}
where 
\[
W^{(j)}_{\eta}(u)= \sum^\ast \abs{H_{\bal,\bbt}}\pa{\bal_j + \bbt_j} e^{\eta \abs{\pi(\bal-\bbt)}}{u}^{\bal + \bbt - e_j}\,,
\]
since,
by the reality condition \ref{real},  we have
\[
\sum^\ast \abs{H_{\bal,\bbt}}\bbt_j 
e^{\eta \abs{\pi(\bal-\bbt)}}
{u}^{\bal + \bbt - e_j}
=
\sum^\ast \abs{H_{\bal,\bbt}}\bal_j
e^{\eta \abs{\pi(\bal-\bbt)}}
{u}^{\bal + \bbt - e_j} 
=
\frac12 W^{(j)}_{\eta}(u).
\]
By the linear map
\[
L_{r,\tw}:\ell^2\to \th_\tw\,,\qquad
y_j\mapsto \frac{r}{\tw_j}y_j  = u_j\,,
\] 
the ball of radius $1$ in $\ell^2$ is isomorphic to the
the ball of radius $r$ in $\th_\tw$,  namely
$L_{r,\tw}(B_1(\ell^2))=B_r(\th_\tw).$
We have
$$
Y^{(j)}_{H}(y;r,\eta,\tw) 
=
\frac12 W^{(j)}_{\eta}(L_{r,\tw} y)\,.
$$
Then (i) follows.

In order to prove item {\rm (ii)} we rely on the fact that, since we are using the $\eta$-majorant norm, the supremum over $y$ in the norm is achieved on the real positive cone. Moreover, given $u,v\in \ell^2$, if 
\[
|u_j|\le |v_j|\,,\quad \forall j\in \Z 
\] 
then $|u|_{\ell^2}\le |v|_{\ell^2}$.
\qed

\subsection{Proof of Proposition \ref{maspero}}\label{app:maspero}
We start by Taylor expanding $H$ in homogeneous components. The majorant analiticity implies that for a homogeneous component of degree $d$ one has
\[
|H^{(d)}|_{r,p,\eta}^\wc  \le |H|_{r,p,\eta}^\wc
\]
Now let us consider the polinomial map (homogeneous of degree $d-1$) $X_{H^{(d)}}: \th_{p,s,a} \to \th_{p,s,a}$; as is habitual we identify the polynomial map with the corresponding symmetric  multilinear operator $M^{(d-1)}:   \th_{p,s,a}^{d-1} \to \th_{p,s,a}$. Since we are in a Hilbert space, one has that 
\begin{align*}
|\und{M}|^{\rm op}_{p,s,a}&:=
\sup_{\substack{u_1,\dots u_{d-1}\in \th_{p,s,a}\\|u_i|_{p,s,a}\le 1}} | \und{M}^{(d-1)}(u_1,\dots,u_{d-1})|_{p,s,a}=
\sup_{|u|_{p,s,a}\le 1} | \und{M}^{(d-1)}(u,\dots,u)|_{p,s,a} \\
&= \sup_{|u|_{p,s,a}\le  1} |X_{\und{H}^{(d)}}|_{p,s,a} \le  r^{-d+2} |H|_{r,p,s,a,\eta}^\wc
\end{align*}
for all $\eta\ge 0$. 
Now let us compute the tame norm on a homogeneous component, i.e. 
\[
\sup_{|u|_{p_0,s,a}\le  r-\rho} \frac{| \und{M}^{(d-1)}(u^{d-1})|_{p,s,a}}{|u|_{p,s,a}} = \sup_{|u|_{ p_0,s,a}\le  r-\rho} \frac{| \und{N_p}^{(d-1)}(u^{d-1})|_{ p_0,s,a}}{|u|_{p,s,a}}
\]
where
\[
\und{N_p}^{(d-1,j)}(u^{d-1})={ \jap{j}^{p-p_0} \sum_{j_1,\dots,j_{d-1}} |M_{j_1,\dots j_{d-1}}^{(d-1,j)}| u_{j_1}\dots u_{j_{d-1}}}
\]
now  setting $\pi= \sum_j j_i- j $ we have 
\begin{align*}
&{\und{N_p}^{(d-1)}(u_1,\dots,u_{d-1})}\\ & \le (d-1)  \jap{j}^{p-p_0} \sum_{\substack{j_1,\dots,j_{d-1}: \\  |j_1|\ge |j_i|}} |M_{j_1,\dots j_{d-1} }^{(d-1,j)}| u_{j_1}\dots u_{j_{d-1}}\\ &\le (d-1) \sum_{\substack{j_1,\dots,j_{d-1}: \\  |j_1|\ge |j_i|}}\pa{\sum_i\jap{j_i}+|\pi| }^{p-p_0}  |M_{j_1,\dots j_{d-1}}^{(d-1,j)}| u_{j_1}\dots u_{j_{d-1}}\\ &
\le (d-1)2^{p-p_0}  C(\eta,p)\sum_{\substack{j_1,\dots,j_{d-1}}} e^{\eta|\pi|} |M_{j_1,\dots j_{d-1}}^{(d-1,j)}| u_{j_1}\dots u_{j_{d-1}} \\ 
&+ (d-1)2^{p-p_0} (d-1)^{p-p_0} \sum_{{j_1,\dots,j_{d-1}}} |M_{j_1,\dots j_{d-1}}^{(d-1,j)}| \jap{j_1}^{p-p_0}u_{j_1}\dots u_{j_{d-1}}
\end{align*}
which means that for any $|u|_{ p_0,s,a}\le  r-\rho$ one has
\begin{align*}
&| \und{N_p}^{(d-1)}(u^{d-1})|_{ p_0,s,a}\\
& \le (d-1)2^{p-p_0}  C(\eta,p) |H^{(d)}|_{r-\rho,p_0,s,a,\eta}|u|_{p_0,s,a} + 2^{p-p_0} (d-1)^{p-p_0+1} |\und{M}|^{\rm op}_{ p_0,s,a} (r-\rho)^{d-2} |u|_{p,s,a}\\
& \le (d-1)2^{p-p_0}(  C(\eta,p)  +  (d-1)^{p-p_0})(1-\frac{\rho}{r})^{d-2} |H|_{r,p,s,a,\eta_0} |u|_{p,s,a}
\end{align*}
We conclude that
\[
\sup_{|u|_{ p_0,s,a}\le r }\frac{|X_{\und{H}}|_{p,s,a}}{|u|_{p,s,a}} \le 2^{p-p_0}|H|_{r,p_0,s,a,\eta} \sum_{d\ge 2} (d-1)\pa{  C(\eta,p)  +  (d-1)^{p-p_0}}(1-\frac{\rho}{r})^{d-2}\,
\]
and the thesis follows since the right hand side is convergent. \qed

\subsection{Proof of Lemma \ref{cobra}}

Let us look at the time evolution of $|v(t)|_\tw^2$.  
 By construction and Cauchy-Schwarz 
 inequality
 \begin{eqnarray*}
 2|v(t)|_\tw  \left|\frac{d}{dt} |v(t)|_\tw \right|
 &=&
 \left|\frac{d}{dt} |v(t)|_\tw^2 \right|
 = 
2| \Re (v,\dot v)_{\th_\tw}|
=
2| \Re (v,X_R)_{\th_\tw}|
 \\
 &\le&
  2  |v(t)|_\tw |X_{\und R}|_\tw 
\leq
  2 r |v(t)|_\tw |R|_{r,\eta,\tw}\
\end{eqnarray*}
 as long as $|v(t)|_\tw\leq r$;
 namely
 \begin{equation}\label{bufala}
\left|\frac{d}{dt} |v(t)|_\tw\right|
 \leq
 r |R|_{r,\eta,\tw}
\end{equation}
  as long as $|v(t)|_\tw\leq r.$

Assume by contradiction that there exists
a time\footnote{The case $T_0<0$ is analogous.}  
  $$
  0<T_0<\frac{1}{8|R|_{r,\eta,\tw}}
  $$ 
  such that 
  \begin{equation}\label{giuncata2}
\Big| |v(t)|_\tw-|v_0|_\tw\Big|
< \frac{r}{8}\,, \quad
\forall\, 0\leq t<T_0\,,\qquad {\rm but}\ \ 
\Big| |v(T_0)|_\tw-|v_0|_\tw\Big|= \frac{r}{8}
\,.
\end{equation}
  Then
  $$
  |v(t)|_\tw\leq |v_0|_\tw+ \frac{r}{8}
 < 
 r\,
  \qquad
  \quad
\forall\, 0\leq t\leq T_0
  \,.
  $$
  By \eqref{bufala}  we get
  $$
  \Big| |v(T_0)|_\tw-|v_0|_\tw\Big|
  \leq
  r |R|_{r,\eta,\tw} T_0
  <\frac{r}{8}\,,
  $$
  which contradicts \eqref{giuncata2},
  proving \eqref{giuncata}.



\section{Small divisor estimates}
%
\subsection{Proof of Lemma \ref{constance general}}\label{costi1}
The fact that this \eqref{yuan 2} holds true when $\pi=0$ is proven in \cite{Bou1} and \cite{Yuan_et_al:2017}.
The bound  \eqref{yuan 2} is equivalent to proving
\begin{equation}
\sum_{l\ge 1} \na_l^\theta   -  2 \na^\theta_1+ \teta \abs{\pi} -  \pa{2 - 2^\teta}\sum_{l\ge 3} \na_l^\theta\ge 0.
\end{equation}
i.e. 
\begin{equation}\label{cheppalle}
\sum_{l\ge 2} \na_l^\theta   -   \na^\theta_1+ \teta \abs{\pi} -  \pa{2 - 2^\teta}\sum_{l\ge 3} \na_l^\theta\ge 0.
\end{equation}
\\
Inequality \eqref{cheppalle} then follows from 
\begin{equation}
f\pa{\abs{\pi}}:= \sum_{l\ge 2} \na_l^\theta   -   \pa{\abs{\pi} + \sum_{l\ge 2}\na_l}^\teta + \teta \abs{\pi} -  \pa{2 - 2^\teta}\sum_{l\ge 3} \na_l^\theta\ge 0,
\end{equation}
which we are now going to prove. We shall show that the function $f(x)$ is increasing in $x\ge 0$; then the result follows by showing $f(0) \ge 0$, which is what was proven by Yuan and Bourgain.\\
We now verify that $f'(x)\ge 0$. By direct computation we see that 
\begin{equation*}
f'(x)  = - \teta \pa{x + \sum_{l\ge 2}\na_l}^{\teta - 1} +\teta , 
\end{equation*}
so it suffices to prove that
\begin{equation}
1 \le \pa{x + \sum_{l\ge 2}\na_l}^{1 - \teta},
\end{equation}
which is indeed true, since  $\sum_{i\ge 2}\na_i\ge  \na_2\ge   1$ holds, by mass conservation. 
\qed


\subsection{Proof of Lemma \ref{constance 2 gen}}\label{costi2}

In this subsection we will take
\begin{equation}\label{acquafresh}
\bal,\bbt\in\N^\Z\quad {\rm with} \quad
	 1\leq|\bal|=|\bbt|<\infty\,.
	 \end{equation}

Given $u\in \Z^\Z$, with $|u|<\infty,$  consider the set
		\[
		\set{j\neq 0 \,,\quad \mbox{repeated}\quad  \abs{u_j} \;\mbox{times}}\,,
		\]
where $D<\infty$ 		is its cardinality.
Define the vector $m=m(u)$ as the reordering of the elements of the set above 
such that
 $|m_1|\ge |m_2|\ge \dots\geq |m_D|\ge 1.$ 	
 Given $\bal\neq\bbt\in\N^\Z,$ with $|\bal|=|\bbt|<\infty$
 we consider $m=m(\bal-\bbt)$ and $\na=\na(\bal+\bbt).$	
If we denote by $D$ the cardinality of $m$ and $N$ the one of $\na$ we have 
\begin{equation}\label{cappella}
D+\bal_0+\bbt_0\le N
\end{equation}
 and
\begin{equation}\label{abbacchio}
(|m_1|,\dots,|m_D|,\underbrace{1,\;\dots \;,1}_{N-D\;\rm{times}} )\, \leq\,
	 \pa{\na_1,\dots \na_N}\,.
\end{equation}
Set
	$$
	\s_l= {\rm sign}(\bal_{m_l}-\bbt_{m_l})\,.
	$$

	For every function $g$ defined on $\Z$ we have that
\begin{eqnarray}\label{pula2}
\sum_{i\in\Z} g(i) |\bal_i-\bbt_i|
&=&
g(0)|\bal_0-\bbt_0|+
\sum_{l\geq 1} g(m_l)\,,
\nonumber
\\
\sum_{i\in\Z} g(i) (\bal_i-\bbt_i)
&=&
g(0)(\bal_0-\bbt_0)+
\sum_{l\geq 1} \s_l g(m_l)\,.
\end{eqnarray}

\begin{lemma}\label{mizza}
Assume that $g$ defined on $\Z$ is non negative,  even
and not decreasing on $\N.$ 
 Then, if $\bal\neq\bbt$,
 \begin{equation}\label{pula}
\sum_{i\in\Z} g(i) |\bal_i-\bbt_i|
\leq
2g(m_1)+
\sum_{l\geq 3} g(\na_l)\,.
\end{equation}
\end{lemma}
\begin{proof}
 By \eqref{pula2}
 \begin{eqnarray*}
\sum_{i\in\Z} g(i) |\bal_i-\bbt_i|
&=&
g(0)|\bal_0-\bbt_0|+
\sum_{l\geq 1} g(m_l)
\\
&\leq& 
g(1)(\bal_0+\bbt_0)+
2g(m_1)+
\sum_{l\geq 3} g(m_l)
\end{eqnarray*}
and \eqref{pula} follows by
\eqref{cappella} and \eqref{abbacchio}.
\end{proof}

	We denote as before the momentum by $\pi$ so by \eqref{pula2}
	\begin{equation}\label{somma sigma p}
	\pi= \sum_{i\in \Z}\pa{\bal_i-\bbt_i}i =  \sum_l \s_l m_l 
	\end{equation}
	and
	\begin{equation} \label{somma sigma quadro}
	{\sum_i{\pa{\bal_i-\bbt_i}i^2}}=
	 \sum_l \s_l m^2_l\,.
	\end{equation}
Analogously 
\begin{equation}
	\label{enne}
	{\sum_i{\abs{\bal_i-\bbt_i}}}
	= D+|\bal_0-\bbt_0|
	\stackrel{\eqref{cappella}}\leq N\,.
\end{equation}
 Finally note that 
 \begin{equation}\label{gina}
\sigma_l\sigma_{l'} =-1\qquad \Longrightarrow\qquad m_l \neq m_{l'}\,.
\end{equation}
Note that
\begin{equation}\label{sonno}
\bal\neq\bbt \quad \Longrightarrow\quad
N\geq 3 \ \ {\rm or}\ \ \pi\neq 0\,,
\end{equation}
indeed, by mass conservation,
$|\bal|=|\bbt|=1$ therefore  if $N=2$ we get $\bal-\bbt= e_{j_1}-e_{j_2}$ so if $\pi=0$  we have $\bal=\bbt$.
Note also that 
\begin{equation}\label{abbiocco}
\bal\neq\bbt
\qquad
\Longrightarrow
\qquad
D	\geq 1\,,
\end{equation}
indeed, if $D=0$ then
  $\bal_l-\bbt_l=0$ for every $|l|\geq 1$
and, by mass conservation $\bal_0=\bbt_0$, contradicting 
 $\bal\neq\bbt$ .

\begin{lemma}
Given $\bal\neq\bbt\in\N^\Z,$ with $1\leq|\bal|=|\bbt|<\infty$
and satisfying  \eqref{divisor}, we 
have\footnote{Note that by \eqref{sonno}
the r.h.s. of \eqref{chiappechiare} is at least 20. }
\begin{equation}
\label{chiappechiare}
 \abs{m_1}\le 20 |\pi|+ 31\sum_{l\ge 3}\na_l^2\,. 
\end{equation}
\end{lemma}

\begin{proof} 
In the case $D=1$
by  \eqref{somma sigma p}
$|\pi|=|m_1|$ and \eqref{chiappechiare} follows.
 Let us now consider the case $D=2$,
 i.e.
 $$
 \bal-\bbt=\s_1 e_{m_1}+\s_2 e_{m_2}
 +(\bal_0-\bbt_0)e_0\,.
 $$ 
 Let us start with the case 
 $\s_1\s_2=1.$
 By mass conservation $|\s_1+\s_2|=|\bbt_0-\bal_0|=2.$
 By \eqref{enne} $N\geq 4.$
  Then  conditions \eqref{divisor} and \eqref{enne} imply that 
  \[ 
   m_1^2 + m_2^2 \le 20+10|\bal_0-\bbt_0|=40.
  \]
 Then
 $$
 |m_1|\leq \sqrt{40}
 \leq \frac{\sqrt{40}}{2}
 \sum_{\ell=3}^N \na_\ell^2
 $$
 since $N\geq 4$ and $\na_\ell\geq 1.$
When 
  $\s_1\s_2=-1$
   we have $m_1\neq m_2$,  $|\pi|=|m_1-m_2|\geq 1$
   and by mass conservation $\bal_0-\bbt_0=0.$
 Then
  \[ 
  (|m_1|+|m_2|)(|m_1|-|m_2|)=
   m_1^2 - m_2^2 \le 20\,.
  \] 
  If $|m_1|>|m_2|$ then
\begin{equation}
    |m_1|\leq 20\leq 20 |\pi|\,. 
\end{equation}
 Otherwise $m_1=-m_2$ and, therefore,  $|\pi|=2|m_1|,$
 completing the proof in the case $D=2.$

	Let us now consider the case $D \ge 3$. 
	By \eqref{divisor},\eqref{somma sigma quadro}
	and \eqref{enne}
\begin{eqnarray*}
	m_1^2 +\s_1\s_2 m_2^2 
	&\le& 
	10 N + \sum_{l=3}^Dm_l^2 
	\leq
		10 N + \sum_{l=3}^N\na_l^2 
	\\
	&=&
	 20 +\sum_{l=3}^N (10 + \na_l^2) {\le} 20+ 11
	 \sum_{l=3}^N
	 \na_l^2 
	 \stackrel{N\ge 3}{\le} 31\sum_{l=3}^N
	 \na_l^2\,.
\end{eqnarray*}
	If $\sigma_1\sigma_2 = 1$ then
	\[ 
	\abs{m_1}, \abs{m_2} \le \sqrt{31\sum_{l\ge 3} \na_l^2}.
	\]
	If  $\s_1\s_2 = -1$
	\[
	 (|m_1|+|m_2|)(|m_1|-|m_2|)=
	m_1^2 - m_2^2 \le 
	 31\sum_{l\ge 3} \na_l^2.
	\]
	Now, if $\abs{m_1}\ne \abs{m_2}$ then 
	\[
	\abs{m_1} + \abs{m_2} 
	\le 	 31\sum_{l\ge 3} \na_l^2.
	\]
	Conversely, if $\abs{m_1} = \abs{m_2}$, by \eqref{gina}, $m_1\ne m_2$, hence $m_1 = - m_2$. By substituting this relation into \eqref{somma sigma p}, we have 
	\[
 2\abs{m_1} \le \abs{\pi} + \sum_{l\ge 3}\abs{m_l} \le \abs{\pi} + \sum_{l\ge 3}\na_l^2\,,
 \]
	concluding the proof.
	\end{proof}

		
	\begin{proof}[Conclusion of the proof of Lemma \ref{constance 2 gen}]
	
As above,	
given $\bal,\bbt\in\N^\Z,$ with $1\leq |\bal|=|\bbt|<\infty$
 we consider $m=m(\bal-\bbt)$ and $\na=\na(\bal+\bbt).$	
Note that $N:=|\bal+\bbt|\geq 2.$

We have\footnote{Using that for $x,y\geq 0$ and $0\leq c\leq 1$
we get $(x+y)^c\leq x^c+y^c.$} 
	\begin{eqnarray}\label{cosette}
	\sum_i\abs{\bal_i-\bbt_i}\jap{i}^{\theta/2} 
	&\stackrel{\eqref{pula}}\leq&
	2\abs{m_1}^{\frac{\teta}{2}} +  
	\sum_{l\ge 3} \na_l^{\frac{\teta}{2}}  
	\nonumber
	\\
	& \stackrel{\eqref{chiappechiare}}\le &
	2\pa{20|\pi| + 31 \sum_{l\ge 3} \na_l^2 }^{\frac{\teta}{2}} +
	 \sum_{l\ge 3} \na_l^{\frac{\teta}{2}}   
	\nonumber
	\\
	& \le&
	 2 \pa{20|\pi| }^{\frac{\teta}{2}} + 2(31)^{\frac{\teta}{2}}\sum_{l\ge 3}\na_l^\teta + \sum_{l\ge 3} \na_l^{\frac{\teta}{2}} 
	 \nonumber
	 \\
	& \le  & 
\frac{13 }{1-\teta}\pa{ (1-\theta) \abs{\pi}+ (2-2^\teta)\pa{\sum_{l\ge 3}\na_l^\teta }}\,,
\end{eqnarray}
using that $1-\teta\leq 2-2^\teta$
for $0\leq \teta\leq 1.$
Then by Lemma \ref{constance general} and \eqref{cosette} we get
\begin{eqnarray*}
\sum_i\abs{\bal_i-\bbt_i}\jap{i}^{\theta/2}
	&\le &
	\frac{13 }{1-\teta}\pa{ (1-\theta) \abs{\pi}+ \sum_i \jap{i}^\teta\pa{\bal_i +\bbt_i} + \theta \abs{\pi} - 2\na_1^\teta}\\
	& \le &
	\frac{13 }{1-\teta} \sq{\sum_i \jap{i}^\teta\pa{\bal_i +\bbt_i} + \abs{\pi} - 2\jap{j}^\teta}\,,
\end{eqnarray*}
proving \eqref{adele}.

\medskip

Let us now prove \eqref{cosette4} passing to the logarithm.
We have	
\begin{equation}\label{cosette3}
\begin{aligned}
&\sum_i\ln(1+\abs{\bal_i-\bbt_i}{\jap{i}}) 
\\
&=
\sum_{|i|\leq 1}\ln(1+\abs{\bal_i-\bbt_i}) 
+
\sum_{|i|\geq 2}\ln(1+\abs{\bal_i-\bbt_i}|i|) 
\\
&\leq 3\ln(1+N ) 
+
\sum_{|i|\geq 2}\ln(1+\abs{\bal_i-\bbt_i}|i|)
\\
&\leq 3\ln 2+3\ln N
+
\frac32 \sum_{|i|\geq 2}\abs{\bal_i-\bbt_i}\ln|i| \,,
\end{aligned}
\end{equation}
using that $1+cx \leq \frac32 x^c$ for $c\geq 1,$ $x\geq 2.$
If $\bal_i-\bbt_i=0$ for every $|i|\geq 2$
then \eqref{cosette4} follows.
Assume now that 
$\bal_i-\bbt_i\neq 0$ for some $|i|\geq 2.$
By
\eqref{sonno}
we have
\begin{equation}\label{giotto}
N\geq 3\quad {\rm or}\quad \pi\neq 0\,.
\end{equation}
We claim that, when $N\geq 3,$
\begin{equation}\label{scrovegni}
\ln\pa{\sum_{l= 3}^N \na_l^2}
\leq \ln N +\sum_{l= 3}^N \ln\na_l^2\,.
\end{equation}
Let
$\mathcal S:=\{3\leq l\leq N,\  {\rm s.t.}\ \na_l\geq 2\}.$
If $\mathcal S=\emptyset$ we have the equality in \eqref{scrovegni}.
Otherwise 
$\sum_{l\in\mathcal S}\na_l^2\geq 4$
and\footnote{\label{cimabue} Use that
$\ln(x+y)\leq \ln x+\ln y$ if $x,y\geq 2.$}
$$
\ln\pa{\sum_{l= 3}^N \na_l^2}
\leq \ln\pa{N+\sum_{l\in\mathcal S}\na_l^2}
\leq \ln N +\sum_{l\in\mathcal S} \ln\na_l^2\,,
$$
proving \eqref{scrovegni}.
\\
We claim that
\begin{equation}\label{sistina}
\ln\pa{20 \abs{\pi}+ 31 \sum_{l= 3}^N \na_l^2 }
\leq 
\ln(1+|\pi|)+\ln N+\sum_{l= 3}^N \ln\na_l^2 
+\ln 20+\ln 31\,.
\end{equation}
Indeed consider first the case $\pi=0,$
then $N\geq 3$ by \eqref{giotto}
and \eqref{sistina} follows by \eqref{scrovegni}.
Consider now the case $|\pi|\geq 1.$
If $N<3$ \eqref{sistina} follows (there is no sum).
If $N\geq 3$ we have\footnote{Recall footnote \ref{cimabue}.}
\begin{eqnarray*}
&&\ln\pa{20 \abs{\pi}+ 31 \sum_{l= 3}^N \na_l^2 }
\leq 
\ln\pa{20 \abs{\pi}}+\ln\pa{ 31 \sum_{l= 3}^N \na_l^2 }
\\
&&
\leq
\ln(|\pi|)+\ln\pa{ \sum_{l= 3}^N \na_l^2 } 
+\ln 20+\ln 31\,.
\end{eqnarray*}
Recalling \eqref{scrovegni} this complete the proof of 
\eqref{sistina}.

Let us continue the proof of \eqref{cosette4}.
Set $g(i):=0$ if $|i|\leq 1$ and $g(i):=\ln|i|$ if $|i|\geq 2$
and apply \eqref{pula} to \eqref{cosette3}; we get
\begin{eqnarray*}
 &&\sum_{|i|\geq 2}\abs{\bal_i-\bbt_i}\ln|i|
 \leq 2 \ln |m_1|+\sum_{l\geq 3} \ln |\na_l|
 \\
&&\stackrel{\eqref{chiappechiare}}\leq
 2\ln\pa{20 \abs{\pi}+ 31 \sum_{l\ge 3} \na_l^2 }
 +\sum_{l\geq 3} \ln \na_l
 \\
 &&\stackrel{\eqref{sistina}}\leq
 2\ln(1+|\pi|)+2\ln N+5\sum_{l= 3}^N \ln\na_l 
+16\,.
\end{eqnarray*}
Inserting in \eqref{cosette3} we obtain
\begin{eqnarray*}
\sum_i\ln(1+\abs{\bal_i-\bbt_i}{\jap{i}}) 
\leq
3\ln(1+|\pi|)+6\ln N+\frac{15}{2}\sum_{l= 3}^N \ln\na_l 
+27\,.
\end{eqnarray*}
concluding the proof of \eqref{cosette4}.
\end{proof}

\subsection{Proof of Lemma \ref{pajata}}\label{vaccinara}

	First of all we note that
	$$
	\sum_i f_i(|\ell_i|)=
	\sum_{i\ \text{s.t.} \ \ell_i\neq 0} f_i(|\ell_i|)
	$$
	since $f_i(0)=0.$
	We have that\footnote{Using that
		$\ln(1+y)\leq 1+\ln y$ for every $y\geq 1.$}
	$$
	f_i(x) \leq
	-\frac{\s}{C_* }\jap{i}^{\frac\theta 2}x + {2} \ln(x)+  \pa{{2}+\fp}\ln \jap{i} +1\,,\qquad
	\forall\, x\geq 1\,.
	$$
	We have that
	\begin{equation*}
	\max_{x\geq 1} \left( -\frac{\s}{C_* }\jap{i}^{\frac\theta 2}x + 2 \ln(x)\right)=
	\left\{
	\begin{array}{l} 
	\displaystyle -\frac{\s}{C_* }\jap{i}^{\frac\theta 2} 
	\qquad\qquad\qquad\ \,\qquad\text{if}\quad \jap{i}\geq i_0\,,
	\\
	 \empty
	\\  
	\displaystyle
	-{2}+{2}\ln \frac{2C_* }{\s}-\teta \ln \jap{i}
	\qquad\text{if}\quad \jap{i}< i_0\,,
	\end{array}
	\right.
	\end{equation*}
	where 
	$$
	i_0:=\left(\frac{2C_* }{\s}\right)^{\frac 2 \theta}\,,
	$$
	since the maximum is achieved for 
	$x=1$ if $\jap{i}\geq i_0$ and 
	$x=\frac{2C_* }{\s \jap{i}^{\theta/2}}$
	if $\jap{i}< i_0$.
	Note that $i_0\geq e.$
	Then we get 
	\begin{eqnarray*}
		&& \sum_i f_i(|\ell_i|)
		=
		\sum_{i\ \text{s.t.} \ \ell_i\neq 0} f_i(|\ell_i|)
		\leq
		\\
		&& \sum_{\jap{i}\geq i_0\ \text{s.t.} \ \ell_i\neq 0}
		\left(
		\pa{{2}+\fp}\ln \jap{i} +1
		-\frac{\s}{C_* }\jap{i}^{\frac\theta 2}
		\right)
		+
		\sum_{\jap{i}< i_0}
		\left(
		2\ln \frac{2C_* }{\s}+\Big({2}+\fp
		-\teta\Big) \ln \jap{i}
		\right)\,.
	\end{eqnarray*}
	We immediately have that
	\begin{eqnarray*}
		&& \sum_{\jap{i}< i_0}
		\left(
		2\ln \frac{2C_* }{\s}+\Big({2}+\fp
		-\teta\Big) \ln \jap{i}
		\right)
		\leq 3 i_0 
		\left(
		2\ln \frac{2C_* }{\s}+({2}+\fp
		) \ln i_0
		\right)
		\\
		&&
		=
		3\left(2+ \frac{2}{\theta}({2}+\fp)\right)\left(\frac{2C_* }{\s}\right)^{\frac 2 \theta}
		\ln \frac{2C_* }{\s}\,.
	\end{eqnarray*}
	Moreover, in the case $\jap{i}\geq i_0\geq e,$
	$$
	\pa{{2}+\fp}\ln \jap{i} +1
	-\frac{\s}{C_* }\jap{i}^{\frac\theta 2}
	\leq 
	\pa{{2}+\fp+1}\ln \jap{i}
	-\frac{\s}{C_* }\jap{i}^{\frac\theta 2} 
	=\frac{2}{\theta}\pa{{2}+\fp+1}\Big(
	\ln \jap{i}^{\frac\theta 2}-2{\frak C } \jap{i}^{\frac\theta 2}
	\Big)
	$$
	where 
	$$
	{\frak C }:=\frac{\s\theta}{4C_* \pa{{2}+\fp+1}}<1
	\,.
	$$
	Therefore
	$$
	S_*:=\sum_{\jap{i}\geq i_0\ \text{s.t.} \ \ell_i\neq 0}
	\left(
	\pa{{2}+\fp}\ln \jap{i} +1
	-\frac{\s}{C_* }\jap{i}^{\frac\theta 2}
	\right)
	$$
	satisfies
	$$
	S_*
	\leq
	\sum_{\jap{i}\geq i_0\ \text{s.t.} \ \ell_i\neq 0}
	\frac{2}{\theta}\pa{{2}+\fp+1}\Big(
	\ln \jap{i}^{\frac\theta 2}-2{\frak C} \jap{i}^{\frac\theta 2}
	\Big)\,.
	$$
	We have that\footnote{Using that, for every fixed
		$0<{\frak C} \leq 1,$ we have
		${\frak C} x\geq \ln x$ for every $x\geq
		\frac{2}{{\frak C}}\ln
		\frac{1}{{\frak C}} .$}
	$$
	\ln \jap{i}^{\frac\theta 2}-2{\frak C} \jap{i}^{\frac\theta 2}
	\leq -{\frak C} \jap{i}^{\frac\theta 2}\,,
	\qquad
	\text{when}\qquad
	\jap{i}\geq 
	i_*:=
	\left(\frac{2}{{\frak C}}\ln
	\frac{1}{{\frak C}}\right)^{\frac2\theta}
	\,.
	$$
	Note that
	$$
	i_\sharp
	\geq \max\{ i_0, i_*\}\,.
	$$
	Therefore
	\begin{eqnarray*} S_*
		&\leq&
		\frac{2}{\theta}\pa{{2}+\fp+1}
		\left(
		\sum_{\jap{i}<i_\sharp}
		\ln \jap{i}^{\frac\theta 2}
		-
		\sum_{\jap{i}\geq i_\sharp\ \text{s.t.} \ \ell_i\neq 0}
		\Big(
		{\frak C} \jap{i}^{\frac\theta 2}
		\Big)
		\right)
		\\
		&\leq&
		\pa{{2}+\fp+1}
		\left( 
		3 i_\sharp \ln i_\sharp
		-\frac{2{\frak C}}{\theta} M_\ell^{\frac\theta 2}
		\right)
		\,.
	\end{eqnarray*}
	where
	$$
	M_\ell:=\max\{  
	|i|\geq i_\sharp,\ \ \text{s.t.}\ \ \ell_i\neq 0
	\}
	$$
	and $M_\ell:=0$ if $|\ell_i|=0$
	for every $|i|\geq i_\sharp.$
	In conclusion we get
	\begin{eqnarray*}
		\sum_i f_i(|\ell_i|)
		&\leq&
		3\left({2}+ \frac{2}{\theta}({2}+\fp)\right)\left(\frac{2C_* }{\s}\right)^{\frac 2 \theta}
		\ln \frac{2C_* }{\s}
		+
		\pa{{2}+\fp+1}
		\left( 
		3 i_\sharp \ln i_\sharp
		-\frac{2{\frak C}}{\theta} M_\ell^{\frac\theta 2}
		\right)
		\\
		&\leq&
		6(\fp+3)
		i_\sharp \ln i_\sharp
		- \frac{\s}{2C_* } M_\ell^{\frac\theta 2}
		\\
		&\leq&
		7(\fp+3)
		i_\sharp \ln i_\sharp
		- \frac{\s}{2C_* } \big(\na_1(\ell)\big)^{\frac\theta 2}\,,
	\end{eqnarray*}
	noting that $\na_1(\ell)=M_\ell$ if $M_\ell\neq 0,$
	otherwise 
	$\na_1(\ell)< i_\sharp,$
	and, therefore, 
	$$
	\frac{\s}{2C_* } \big(\na_1(\ell)\big)^{\frac\theta 2}
	<
	\frac{\s}{2C_* } i_\sharp^{\frac\theta 2}
	\leq
	(\fp+3)
	i_\sharp \ln i_\sharp
	$$
\qed

\subsection{Measure Estimates}\label{cantor}
\begin{proof}[Proof of Lemma  \ref{misura}]
	For $\ell\in \Z^\Z$ with $ 0<|\ell|<\infty$ we define 
	\[
	\mathcal R_\ell := \set{\omega\in \Omega_\fp\,:\;	|\omega\cdot \ell|\leq \frac{\gamma}{1+|\ell_0|^{\mu_1}} \prod_{n\neq 0}\frac{1}{(1+|\ell_n|^{\mu_1} | n|^{{\mu_2}+\fp})}}
	\] 
	\begin{itemize}[leftmargin=*]
		\item if $\ell$  is such that $\ell_n=0$ $\forall n\neq 0$ then 
		$$
		\mu(\mathcal R_\ell) = \frac{\g}{1+|\ell_0|^{\mu_1} }
		$$\\
		\item Otherwise: let $s=s(\ell)>0$ be the smallest positive index $i$ such that $|\ell_i |+|\ell_{-i}|\neq 0$ and $S=S(\ell)$ be the biggest. 
		Then  we have\footnote{Assume, e.g. that $\ell_s\neq 0$, then
			$|\partial_{\xi_s}\omega\cdot\ell|\geq s^{-\fp}\,.$} 
		$$
		\mu(\mathcal R_\ell) \le \frac{\g  s^\fp }{\pa{1+|\ell_0|^{\mu_1}} }\prod_{n\neq 0}\frac{1}{(1+|\ell_n|^{\mu_1} |n|^{{\mu_2}+\fp})}.
		$$
	\end{itemize}
	Let us  write
	
	\begin{align*}
	\frac{1}{1 + |\ell_0|^{\mu_1}}\prod_{n\neq 0}\frac{1}{(1+|\ell_n|^{\mu_1} |n|^{{\mu_2}+\fp})} &= \frac{1}{1 + |\ell_0|^{\mu_1}}\prod_{n>0}\frac{1}{(1+|\ell_n|^{\mu_1} |n|^{{\mu_2}+\fp})} \frac{1}{(1+|\ell_{-n}|^{\mu_1} |n|^{{\mu_2}+\fp})}\\
	&= \frac{1}{1 + |\ell_0|^{\mu_1}}\prod_{s(\ell)\le n\le S(\ell)}\frac{1}{(1+|\ell_n|^{\mu_1} |n|^{{\mu_2}+\fp})} \frac{1}{(1+|\ell_{-n}|^{\mu_1} |n|^{{\mu_2}+\fp})}
	\end{align*}	
	Now
	\begin{align}\label{mamma}
	&\mu(\Omega_\fp\setminus \dgp)\le \sum_{\ell} \mu(\cR_\ell)= \sum_{ \ell_0}\frac{\g}{1+|\ell_0|^{\mu_1} } \\
	+& \sum_{s>0} \sum_{\substack{ \ell: s(\ell)= S(\ell)=s}}\frac{1}{1 + |\ell_0|^{\mu_1}}\frac{\g s^\fp}{|\ell_s|(1+|\ell_s|^{\mu_1} |s|^{{\mu_2}+\fp})} \frac{1}{(1+|\ell_{-s}|^{\mu_1} |s|^{{\mu_2}+\fp})}\label{mamma2}\\
	+& \sum_{0<s< S} \sum_{\substack{ \ell: s(\ell)=s,\\ S(\ell)=S}}\frac{\g s^\fp}{1 + |\ell_0|^{\mu_1}}\prod_{s\le n\le S }\frac{1}{(1+|\ell_n|^{\mu_1} |n|^{{\mu_2}+\fp})} \frac{1}{(1+|\ell_{-n}|^{\mu_1} |n|^{{\mu_2}+\fp})}.\label{mamma3}
	\end{align}
	Let us estimate \eqref{mamma2}
	\begin{align*}
	&\sum_{s>0}\sum_{\ell_0\in \Z}\frac{1}{1 + |\ell_0|^{\mu_1}}\sum_{ \substack{\ell_s,\ell_{-s}\in\Z\\ |\ell_s|+|\ell_{-s}|>0}}\frac{\g s^\fp}{(1+|\ell_s|^{\mu_1} |s|^{{\mu_2}+\fp})} \frac{1}{(1+|\ell_{-s}|^{\mu_1} |s|^{{\mu_2}+\fp})}\\ 
	&\le c(\mu_1)\g \sum_{s>0} s^\fp \sum_{ \substack{\ell_s,\ell_{-s}\in\Z\\ |\ell_s|+|\ell_{-s}|>0}}\frac{1}{(1+|\ell_s|^{\mu_1} |s|^{{\mu_2}+\fp})} \frac{1}{(1+|\ell_{-s}|^{\mu_1} |s|^{{\mu_2}+\fp})}
	\end{align*}
	Now  since
	\[
	\sum_{h=1}^\infty  \frac{1}{(1+h^{\mu_1} |n|^{{\mu_2}+\fp})} \le \sum_{h=1}^\infty  \frac{1}{h^{\mu_1} |n|^{{\mu_2}+\fp}} \le  \frac{c(\mu_1)}{|n|^{{\mu_2}+\fp}}
	\]
	we have 
	\[
	\sum_{h\in\Z} \frac{1}{(1+|h|^{\mu_1} |n|^{{\mu_2}+\fp})} \le  1+ \frac{2c(\mu_1)}{|n|^{{\mu_2}+\fp}}.
	\]
	Then  we have
	\[
	\sum_{ \substack{\ell_s,\ell_{-s}\in\Z\\ |\ell_s|+|\ell_{-s}|>0}}\frac{1}{(1+|\ell_s|^{\mu_1} |s|^{{\mu_2}+\fp})} \frac{1}{(1+|\ell_{-s}|^{\mu_1} |s|^{{\mu_2}+\fp})} \le \frac{c_1(\mu_1)}{|s|^{{\mu_2}+\fp}}
	\]
	and consequently
	\eqref{mamma2} is bounded by  
	\[
	c_2(\mu_1)\g \sum_{s>0} |s|^b \le c_3(\mu_1,\mu_2)\g.
	\]
	\\
	Regarding the third line in \eqref{mamma}, 
	we note that for all $n$ we have
	\[
	\sum_{ \ell_n,\ell_{-n}\in \Z}  \frac{1}{(1+|\ell_n|^{\mu_1} |n|^{{\mu_2}+\fp})} \frac{1}{(1+|\ell_{-n}|^{\mu_1} |n|^{{\mu_2}+\fp})}\le \pa{1 + 2 \frac{c(\mu_1)}{|n|^{{\mu_2}+\fp}}}^2\,.
	\]
	Hence
	\begin{align*}
	&\sum_{\substack{ \ell: s(\ell)=s,\\ S(\ell)=S}}\frac{1}{1 + |\ell_0|^{\mu_1}}\prod_{s\le n\le S }\frac{1}{(1+|\ell_n|^{\mu_1} |n|^{{\mu_2}+\fp})} \frac{1}{(1+|\ell_{-n}|^{\mu_1} |n|^{{\mu_2}+\fp})}
	\\ &= \sum_{ \ell_0\in \Z}\frac{1}{1 + |\ell_0|^{\mu_1}}\times \sum_{ \substack{\ell_s,\ell_{-s}\in\Z\\ |\ell_s|+|\ell_{-s}|>0}}\frac{1}{(1+|\ell_s|^{\mu_1} |s|^{{\mu_2}+\fp})} \frac{1}{(1+|\ell_{-s}|^{\mu_1} |s|^{{\mu_2}+\fp})}\\
	&\times \sum_{ \substack{\ell_S,\ell_{-S}\in\Z\\ |\ell_S|+|\ell_{-S}|>0}}\frac{1}{(1+|\ell_S|^{\mu_1} |S|^{{\mu_2}+\fp})} \frac{1}{(1+|\ell_{-S}|^{\mu_1} |S|^{{\mu_2}+\fp})}\\
	&\times \prod_{s< n< S }\sum_{ \ell_n,\ell_{-n}\in \Z}\frac{1}{(1+|\ell_n|^{\mu_1} |n|^{{\mu_2}+\fp})} \frac{1}{(1+|\ell_{-n}|^{\mu_1} |n|^{{\mu_2}+\fp})}
	\\
	&\le \frac{c_4(\mu_1)}{s^{{\mu_2}+\fp}S^{{\mu_2}+\fp}} \prod_{s< n< S }\pa{1 + 2 \frac{c(\mu_1)}{|n|^{{\mu_2}+\fp}}}^2
	\\
	&\le \frac{c_4(\mu_1)}{s^{{\mu_2}+\fp}S^{{\mu_2}+\fp}} 
	\exp\left(
	\sum_{n\geq 1 } \ln\pa{1 + 2 \frac{c(\mu_1)}{|n|^{{\mu_2}+\fp}}}^2 \right)
	\\
	&\le \frac{c_5(\mu_1,\mu_2)}{s^{{\mu_2}+\fp}S^{{\mu_2}+\fp}}\,.
	\end{align*}
	Then, multiplying by $\gamma s^\fp$ and taking the $\sum_{0<s<S},$
	we have that also  \eqref{mamma3} is bounded  by  some constant $\Cmeas(\mu_1,\mu_2)\g$.
\end{proof}

\bibliographystyle{alpha}
\bibliography{bibliografiaBMP}

\end{document}